\documentclass[11pt,a4paper]{scrartcl}
\pdfoutput=1 
\usepackage[left=1.1in,right=1.1in,top=1in,bottom=1in]{geometry}
\usepackage[utf8]{inputenc}
\usepackage{amsmath}
\usepackage{amsfonts}
\usepackage{amsthm}
\usepackage{amssymb}
\usepackage[english]{babel,csquotes}
\usepackage{microtype}
\usepackage{url}
\usepackage{hyperref}
\usepackage{float}
\usepackage{enumitem}
\usepackage{ifthen}
\usepackage{tikz}
\usetikzlibrary{shapes,backgrounds,plothandlers,plotmarks,calc,arrows,fadings,patterns,decorations.pathmorphing,decorations.pathreplacing,fadings,decorations.markings,arrows.meta,math}
\usepackage{xspace}
\usepackage{comment}
\usepackage{graphicx}
\usepackage{wrapfig}
\usepackage[numbered]{bookmark}
\usepackage{etoolbox}
\title{Hedetniemi's conjecture and strongly multiplicative graphs\texorpdfstring{\footnotetext{The second author has been supported by the National Science Centre, Poland, via the PRELUDIUM grant 2016/21/N/ST6/00475 and by the Foundation for Polish Science via the START stipend programme.}}{}}
\author{Claude Tardif and Marcin Wrochna}

\newtheorem{define}{Definition}

\newtheorem{theorem}[define]{Theorem}
\newtheorem{lemma}[define]{Lemma}

\newtheorem{corollary}[define]{Corollary}

\makeatletter
\renewcommand\maketitle{
   \begin{center}
     {\Large\bfseries\@title\par\vspace{0.3em}}
     {\scshape\@author}
   \end{center}
}
\hypersetup{
    colorlinks,
    linkcolor={green!40!black},
    citecolor={blue!30!black},
    urlcolor={blue!70!black},
	pdfdisplaydoctitle=true,
    pdftitle={\@title},
    pdfauthor={\@author},
    pdfkeywords={Hedetniemi's conjecture, multiplicative graphs, graph homomorhism, universal cover, polymorphism},
}
\makeatother

\newfont{\Bb}{msbm10 scaled\magstep1}
\newcommand{\eps}{\varepsilon}

\newcommand{\ZZ}{\ensuremath{\mathbb{Z}}}
\newcommand{\NN}{\ensuremath{\mathbb{N}}}
\newcommand{\Uu}{\ensuremath{\mathcal{U}}}
\newcommand{\Ueq}[1]{\Uu(#1)}
\newcommand{\eq}{\mathrel{\diamondsuit}}

\renewcommand\Join{\mathbin{\begin{tikzpicture}[xscale=0.25,yscale=0.2,baseline=-1]
	\draw (0,0)--(1,0)--(0,1)--(1,1)--(0,0);
\end{tikzpicture}}}

\usepackage[backend=bibtex,bibencoding=ascii,style=alphabetic,bibstyle=alphabetic,giveninits=true,doi=false,isbn=false,url=false,maxbibnames=99]{biblatex}
\renewbibmacro{in:}{}
\AtEveryBibitem{\clearname{editor}}

\newbibmacro{string+link}[1]{%
	\iffieldundef{ee}
	{\iffieldundef{url}
		{\iffieldundef{doi}
			{\iffieldundef{eprint}
				{#1}
				{\href{http://arxiv.org/abs/\thefield{eprint}}{#1}}}
			{\href{http://dx.doi.org/\thefield{doi}}{#1}}}
		{\href{\thefield{url}}{#1}}}
	{\href{\thefield{ee}}{#1}}} 
\DeclareFieldFormat{title}{\usebibmacro{string+link}{\mkbibemph{#1}}}
\DeclareFieldFormat[article,inproceedings,inbook,incollection,mastersthesis,thesis]{title}{\usebibmacro{string+link}{\mkbibquote{#1}}}

\begin{document}
{
	\pagestyle{empty}
	\maketitle

	\begin{abstract}
		A graph $K$ is multiplicative if a homomorphism from any product $G \times H$ to $K$ implies a homomorphism from $G$ or from $H$.
Hedetniemi's conjecture states that all cliques are multiplicative.
In an attempt to explore the boundaries of current methods, we investigate \emph{strongly multiplicative} graphs, which we define as $K$ such that for any connected graphs $G,H$ with odd cycles $C,C'$, a homomorphism from $(G \times C') \cup (C \times H) \subseteq G \times H$ to $K$ implies a homomorphism from $G$ or $H$.

Strong multiplicativity of $K$ also implies the following property, which may be of independent interest: if $G$ is non-bipartite, $H$ is a connected graph with a vertex $h$, and there is a homomorphism $\phi \colon G \times H \to K$ such that $\phi(-,h)$ is constant, then $H$ admits a homomorphism to $K$.

All graphs currently known to be multiplicative are strongly multiplicative.
We revisit the proofs in a different view based on covering graphs and replace fragments with more combinatorial arguments.
This allows us to find new (strongly) multiplicative graphs: all graphs in which every edge is in at most square, and the third power of any graph of girth $>12$.
Though more graphs are amenable to our methods, they still make no progress for the case of cliques.
Instead we hope to understand their limits, perhaps hinting at ways to further extend them.

	\end{abstract}
	{
		\renewcommand{\baselinestretch}{0.95}
		\tableofcontents
		\renewcommand{\baselinestretch}{1.0}
	}
}

\section{Introduction}

Hedetniemi's conjecture states that the chromatic number
of the tensor product of two graphs is the minimum of the chromatic numbers of the factors.
It is a notoriously difficult problem in graph theory. However, as it often happens,
the conjecture is relatively ``trivial for bipartite graph''; that is, if the product 
of two graphs is bipartite, then one of the factors is bipartite. The proof
builds on the fact that a graph is bipartite if and only if it does not have an odd cycle,
that is, the structural characteristic that makes 2-colouring distinctively simple. 
So, the common wisdom suggests that for product-graphs with
chromatic numbers three or more, Hedetniemi's conjecture should
be refuted or proved in one fell swoop.

In 1986, twenty years later after Hedetniemi's statement, El-Zahar and Sauer~\cite{El-ZaharS85} proved that if the product of two graphs is
3-colourable, then one of the factors is 3-colourable. Nevertheless, the general case did not follow, and even the case of 4-colourings remains wide open. In 1988, H\"{a}ggkvist, Hell, Miller and 
Newman-Lara~\cite{HaggkvistHMN88} generalized the result ``in the wrong direction'': they proved that 
if a product of graphs admits a homomorphism
(that is, an edge-preserving map) to a given odd cycle, then one of the factors
admits a homomorphism to the same odd cycle. They called this the ``multiplicativity''
of the odd cycles: A graph $K$ is called {\em multiplicative} if whenever a product of two
graphs admits a homomorphism to $K$, then one of the factors admits a homomorphism to $K$.
(The term ``productivity'' had earlier been used in~\cite{NesetrilP78}.) Since a $n$-colouring
of a graph corresponds to a homomorphism to the complete graph $K_n$, Hedetniemi's conjecture
is equivalent to the statement that all complete graphs are multiplicative. Thus, 
El-Zahar and Sauer had proved that $K_3$ is multiplicative, and H\"{a}ggkvist, Hell, Miller and Newman-Lara
generalized the result to all odd cycles rather than to all complete graphs.

For a while afterwards, $K_1$, $K_2$ and all odd cycles constituted
the essential list of graphs known to be multiplicative.
More precisely, two graphs $K$ and $K'$ are called {\em homomorphically 
equivalent} if there exist homomorphisms both from $K$ to $K'$ and
from $K'$ to $K$. A graph homomorphically equivalent to $K$ is 
multiplicative if and only if $K$ is, so the list of known 
multiplicative graphs consisted of $K_1$, $K_2$, odd cycles, and homomorphically equivalent graphs.

Then in 2005, the first author~\cite{Tardif05}
proved that all circular cliques $K_r, r \in [2,4)$ are multiplicative,
and in 2017 the second author~\cite{Wrochna17} proved that all square-free graphs
are multiplicative. With this, it might look as though the techniques for handling multiplicativity
are slowly expanding, and perhaps some day it will be possible to show
that $K_4$ and the other complete graphs are multiplicative, proving
Hedetniemi's conjecture.

However, on closer inspection, all the techniques developed so far
to prove the multiplicativity of some graphs also
prove a stronger property, which we will call ``strong multiplicativity''.
In this paper, we will review these techniques, and find new (strongly) multiplicative
graphs along the way. Our list is far from exhaustive, but it hints at how far the current
methods can be pushed. Our hope is that some new benchmark may become apparent, 
outside the scope of the known methods, yet reachable. 
Indeed, in a follow-up paper~\cite{followup}, we develop techniques for disproving the existence of certain homomorphisms and show that strong multiplicativity, and hence all known methods, fail for simple graphs very close to those we consider here.

\subsection{Multiplicativity and strong multiplicativity}

The \emph{tensor} or \emph{categorical product} of two graphs $G$ and $H$ is the graph
$G \times H$ with vertex-set $V(G\times H) = V(G) \times V(H)$, 
whose edges are the pairs $\{(g_1,h_1),(g_2,h_2)\}$ such that $\{g_1,g_2\}$ is an edge of $G$
and $\{h_1,h_2\}$ is an edge of $H$. A graph $K$ is called {\em multiplicative}
if it satisfies the following property:
\begin{quote}
If $G \times H$ admits a homomorphism to $K$,\\
then $G$ or $H$ admits a homomorphism to $K$.
\end{quote}
(The converse is easy to prove, as $G\times H$ always admits homomorphisms to $G$ and to $H$).
We will call a graph $K$ {\em strongly multiplicative} if it satisfies the following property:
\begin{quote}
If $G$ and $H$ are connected graphs and $C$, $C'$ odd cycles in $G$ and $H$, respectively,\\
such that $(G \times C') \cup (C \times H)$ admits a homomorphism to $K$,\\
then $G$ or $H$ admits a homomorphism 
to $K$.
\end{quote}
It is easy to show that the strongly multiplicative graphs are multiplicative, despite the extra assumption of connectedness.
The main result of El-Zahar and Sauer~\cite{El-ZaharS85} was that $K_3$ is multiplicative. In the concluding comments, they
noted that their proof actually shows that $K_3$ is strongly multiplicative. This lead them to conjecture that if 
$(G \times H') \cup (G' \times H)$ is $n$-colourable, where $G'$ and $H'$ are $n$-chromatic subgraphs of the 
connected graphs $G$ and $H$, then $G$ or $H$ is $n$-colourable. However this conjecture was later refuted 
in~\cite{TardifZ02}, for all values of $n$ greater than 3. Meanwhile, the proof of the multiplicativity
of $K_3$ had been extended to all odd cycles in~\cite{HaggkvistHMN88}. There was no mention of strong multiplicativity in that paper,
but nonetheless, the proof again establishes the stronger property. 
Therefore the following result can be
be credited to \cite{El-ZaharS85} and \cite{HaggkvistHMN88}:

\begin{theorem}[\cite{El-ZaharS85,HaggkvistHMN88}] \label{ocasm}
The odd cycle graphs are strongly multiplicative.
\end{theorem}

In~\cite{Tardif05}, this result is used as a black box to find new multiplicative graphs
from known ones, using adjoint functors. In~\cite{Wrochna17}, the use of the result is 
combined with an elaboration of the proof technique to prove that square-free graphs are strongly multiplicative.
We review and modify this approach here, and emphasise the fact that it is strong multiplicativity 
being established. Then we extend it to some graphs containing squares (cycles of length 4 as subgraphs).
\begin{wrapfigure}[5]{r}{0.2\textwidth}
	\centering
	\vspace*{-10pt}
	\tikzset{
	v/.style={circle,draw=black!75,inner sep=0pt,minimum size=7pt,fill=yellow!20!gray!80},
}
\begin{tikzpicture}[semithick,rotate=90,scale=0.7]
	\node[v] (v0) at (0*72:2) {};
	\node[v] (v1) at (1*72:2) {};
	\node[v] (v2) at (2*72:2) {};
	\node[v] (v3) at (3*72:2) {};
	\node[v] (v4) at (4*72:2) {};
	\node[v] (va) at (  70:0.57) {};
	\node[v] (vb) at ( -70:0.57) {};
	\draw (v0)--(v1)--(v2)--(v3)--(v4)--(v0);
	\draw (v4)--(vb)--(v0)--(va)--(v1);
	\draw (v3)--(vb) (v2)--(va);
\end{tikzpicture}
	\vspace*{-20pt}
\end{wrapfigure}
In particular, we show the following:
\begin{theorem}\label{thm:mainmult}
Let $K$ be a graph such that every edge is contained in at most one square.
Then $K$ is strongly multiplicative.
\end{theorem}
An interesting small example is the 4-chromatic graph known as the Moser spindle (right). 

\bigskip

Strong multiplicativity of a graph $K$ turns out to imply the following property, which may be of independent interest and which will play a crucial part in our proofs.
\begin{quote}
	If $G$ is non-bipartite and $H$ is a connected graph with a vertex $h$ such that there exists a homomorphism $\phi\colon G \times H \to K$ with $\phi(-,h)$ constant, then $H$ admits a homomorphism to $K$.
\end{quote}

\subsection{Exponential graphs}
It is useful to rephrase multiplicativity in terms of exponential graphs.
The \emph{exponential graph} $K^G$ is the graph whose vertices are all functions $V(G)\to V(K)$ (not only homomorphisms), with two such functions $f,f'$ adjacent whenever they give a homomorphism $K_2\times G \to K$, that is, $\{f(u),f'(v)\}$ is an edge of $K$ whenever $\{u,v\}$ is an edge of $G$.
For example, the constant functions in $V(K^G)$ induce a subgraph isomorphic to $K$, see Figure~\ref{fig:exp}.
The defining property of exponential graphs is that homomorphisms $\phi\colon G \times H \to K$ correspond to homomorphisms $\phi^*\colon H \to K^G$, namely $\phi^*(h)$ is the function $g \mapsto \phi(g,h)$ (this is sometimes known as \emph{currying}).
It is not hard to show that a graph $K$ is multiplicative if and only if $K^G$ admits a homomorphism to $K$, for all $G$ that do not.

\begin{figure}[H]
	\centering
	\begin{tikzpicture}
	\colorlet{C0}{red!50!black}
	\colorlet{C1}{green!50!yellow!70!gray!25!white}
	\colorlet{C2}{blue!75!white}
	\tikzstyle{v}=[rounded corners=3pt,draw=none,inner sep=1pt, minimum size=6pt, outer sep=-0.4pt]
	\tikzstyle{vv}=[v,draw=black!50,minimum size=7pt]
	\tikzstyle{b}=[draw=black,rounded corners=2pt,inner sep=1pt, minimum width=6pt, minimum height=39pt]
	\def\drawV[#1,#2,#3,#4,#5,#6]{
		\node[v,fill=C#2] at ($#1+(0,0.4)$) {};
		\node[v,fill=C#3] at ($#1+(0,0.2)$) {};
		\node[v,fill=C#4] at ($#1+(0,0)$) {};
		\node[v,fill=C#5] at ($#1-(0,0.2)$) {};
		\node[v,fill=C#6] at ($#1-(0,0.4)$) {};
	}
	\def\drawX[#1,#2,#3,#4,#5,#6]{
		\node[vv,fill=C#2] (#1_0) at ($(#1)+(0,0.5)$) {};
		\node[vv,fill=C#3] (#1_1) at ($(#1)+(0,0.25)$) {};
		\node[vv,fill=C#4] (#1_2) at ($(#1)+(0,0)$) {};
		\node[vv,fill=C#5] (#1_3) at ($(#1)-(0,0.25)$) {};
		\node[vv,fill=C#6] (#1_4) at ($(#1)-(0,0.5)$) {};
	}
	\begin{scope}[scale=1.4]
		\node[b] (v2) at (0,0) {};  \drawV[(v2),2,2,2,2,2];
		\node[b] (v1) at (1,-1) {};  \drawV[(v1),1,1,1,1,1];
		\node[b] (v0) at (2,0.1) {};  \drawV[(v0),0,0,0,0,0];
		\draw[double] (v0)--(v1)--(v2)--(v0);
		\node[b] (va) at (2.7,0.5) {};  \drawV[(va),1,2,1,2,2];
		\node[b] (vb) at (3.2,-0.5) {};  \drawV[(vb),1,2,1,1,2];
		\node[b] (vc) at (4,0) {};  \drawV[(vc),1,2,0,0,0];
		\draw[double] (v0)--(va) (v0)--(vb) (vb)--(vc) (va)--(vc);
		\draw[double] (v2)--(-0.5,-0.2) (v2)--(-0.5,0.3) (v2)--(-0.5,-0.4);
		\draw[double] (v1)--(0.6,-1.2) (v1)--(0.6,-1.5) (v1)--(1.5,-1.4);
		\draw[double] (v0)--(1.9,-0.9) (v0)--(2.2,-1.1) (v0)--(2.4,-0.8);
		\draw[double] (va)--(2.4,0.9) (va)--(3.0,1.0) (va)--(3.2,0.8);
		\draw[double] (vb)--(3.,-0.8) (vb)--(3.6,-0.9) (vb)--(3.7,-0.7);
	\end{scope}

	\begin{scope}[shift={(0.3,0)},xscale=2,yscale=1.6]
		\node (vb) at (3.3,-0.5) {};  \drawX[vb,1,2,1,1,2];
		\node (vc) at (4.3,0) {};  \drawX[vc,1,2,0,0,0];
		\foreach \i in {0,...,4} {
			\node[right of=vc_\i,node distance=10pt] {\small$\i$};
			\pgfmathtruncatemacro\ip{mod(\i+1,5)};
			\draw (vc_\i)--(vb_\ip);
			\draw (vc_\ip)--(vb_\i);
		}
	\end{scope}

	\begin{scope}[shift={(12,-0.7)},scale=1.5]
		\tikzset{
			H/.style={circle,fill=gray!20,draw=gray!20,line width=2pt,inner sep=0pt,minimum size=27pt},
			He/.style={draw=gray!20,line width=30pt},
			G/.style={circle,fill=black,inner sep=0pt,minimum size=3pt},
			Gw/.style={G,fill=white},
			Ge/.style={draw=black},
		}
		\def\drawTriangle{%
			\node[H] (h0) at (210:1) {};
			\node[H] (h1) at (90:1) {};
			\node[H] (h2) at (-30:1) {};
			\draw[He,line join=round]
				(h0.center)--(h1.center)--(h2.center)--cycle;
			\node[H,fill=C0!90!red] (h0) at (210:1) {};
			\node[H,fill=C1!80!white] (h1) at (90:1) {};
			\node[H,fill=C2!80!white] (h2) at (-30:1) {};
		}
		\drawTriangle
		\def\r{0.28}
		\node[Gw,label=90:{\small$0$}] (c0) at ($(h1)+( 0:\r)$) {};
		\node[G] (b1) at ($(h2)+(-30:\r)$) {};
		\node[Gw,label=-120:{\small$2$}] (c2) at ($(h0)+(-130:0.16)$) {};
		\node[G] (b3) at ($(h1)+(-140:\r)$) {};
		\node[Gw,label=-180:{\small$4$}] (c4) at ($(h0)+(120:\r)$) {};
		\node[G] (b0) at ($(h1)+(105:\r)$) {};
		\node[Gw,label=-180:{\small$1$}] (c1) at ($(h2)+(140:\r)$) {};
		\node[G] (b2) at ($(h1)+(-90:0.16)$) {};
		\node[Gw,label=30:{\small$3$}] (c3) at ($(h0)+(30:\r)$) {};
		\node[G] (b4) at ($(h2)+(-90:0)$) {};

		\draw[Ge] (c0)--(b1)--(c2)--(b3)--(c4)--(b0)--(c1)--(b2)--(c3)--(b4)--(c0);
	\end{scope}

\end{tikzpicture}
	\vspace*{-5pt}
	\caption{Left: a part of the exponential graph ${K_3}^{C_5}$.
		Each vertex is shown as a column vector with dark red, blue, and light green representing  values in $V(K_3)$.
		Middle: the $K_3$-coloring of $K_2 \times C_5 \simeq C_{10}$ corresponding to the rightmost edge visible in the exponential graph.
		Right: the corresponding walk of length $2\cdot 5$ in $K_3$.}
	\label{fig:exp}
\end{figure}
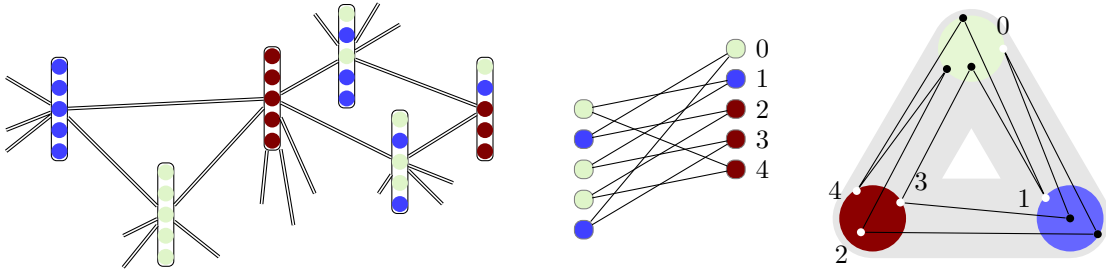

\pagebreak[3]


For strong multiplicativity we instead look at a homomorphism
	$\phi \colon (G \times C') \cup (C \times H) \to K$
for some odd cycles $C,C'$ in connected graphs $G,H$,
which gives us homomorphisms $\phi^*: G \to K^{C'}$ and $\phi^* : H \to K^C$.
To prove that $K$ is strongly multiplicative, it suffices to show
that one of these two homomorphisms maps $G$ or $H$
into a connected component of $K^{C'}$ or $K^{C}$ that admits a homomorphism to $K$.
This is the approach used in~\cite{El-ZaharS85}, \cite{HaggkvistHMN88}, \cite{Wrochna17}, and here.

For instance, if $\phi(-,h)$ is a constant function for some fixed $h \in V(H)$,
then $\phi^* : H \to K^C$ maps all of $H$ into the connected component of $K^C$ that contains the constant functions $V(C) \to V(K)$.
In many cases, we are able to show that this component indeed admits a homomorphism to $K$, even though not all of $K^C$ does.

We hence focus on understanding connected components of $K^{C_n}$ for odd $n$ ($C_n$ denotes the length-$n$ cycle graph).
An edge $\{h,h'\}$ of $K^{C_n}$ corresponds to a homomorphism $K_2 \times C_n \to K$, and since $K_2 \times C_n$ is a cycle of length $2n$, this in turn corresponds to a closed walk of length $2n$ in $K$.
More precisely, if we choose an orientation $(h,h')$ of $\{h,h'\}$ (i.e., an arc of $K^{C_n}$) and we denote vertices of $C_n$ as elements of $\ZZ_n$, then the closed walk corresponding to $(h,h')$ is given by the sequence of vertices:
$$h(0),h'(1),h(2),\dots,h(n-1),h'(0),h(1),h'(2),\dots,h(0).$$
The approach thus relies on describing properties of closed walks in~$K$ that are shared by all walks corresponding to arcs in the same component of $K^{C_n}$.
This allows to classify those components and to show that some of them admit a homomorphism to $K$.

\subsection{Organization}
In the next section, we introduce the necessary concepts (walks, covers) and our approach in more detail, first a bit informally for square-free graphs.
In Section~\ref{sec:squareFree} we present new proofs that square-free graphs are strongly multiplicative.
Section~\ref{sec:squares} extends the concepts that we used to all graphs.
We then apply them in Section~\ref{sec:general} to find new (strongly) multiplicative graphs among graphs with few squares, in particular proving Theorem~\ref{thm:mainmult}.
We also sketch some obstacles to further generalizations there.
In Section~\ref{sec:adjoint} we show how to use another approach, namely the adjoint functors from~\cite{Tardif05}, to find other new (strongly) multiplicative graphs, namely powers of graphs of high girth.
Some proofs are deferred to appendices for an interested reader.

\pagebreak
\section{Walks and covers}\label{sec:walksAndCovers}
\subsection{Reduced walks and the fundamental group}
Following~\cite{Wrochna17}, we view a \emph{walk} in a graph $G$ as a product of arcs
$W = (v_0,v_1) (v_1,v_2)\cdots$ $(v_{\ell-1},v_\ell)$ 
where $(v_i,v_{i+1})$ is one of the two possible orientations of an edge $\{v_i,v_{i+1}\}$ of $G$, for $i=0,\dots,\ell-1$.
The same walk can be of course described by the sequence of vertices $v_0,v_1,\dots,v_\ell$.
Vertices and edges can repeat in a walk.
We say $W$ is a walk from $v_0$ to $v_\ell$ of length $\ell$ and define $\iota(W)=v_0, \tau(W)=v_\ell$.
A \emph{closed walk} rooted at $r \in V(G)$ is a walk from $r$ to $r$.
We write $\eps_r$ for the walk of length 0 with $\iota(\eps_r)=\tau(\eps_r)=r$ (we usually skip the subscript).

Identifying $(v_1,v_2)^{-1}$ with $(v_2,v_1)$ allows to simplify arcs,
and thus \emph{reduce} a walk to an equivalent {\em reduced walk}
which has no two consecutive arcs that are mutually inverse.
For example, one can show that, for a connected square-free graph $K$, closed walks corresponding to arcs $(h,h')$ in the connected component of constant functions in $K^{C_n}$ are exactly the walks that reduce to $\eps$ in $K$ (as in Figure~\ref{fig:exp}).
We hence denote this component as $K^{C_n}_\eps$.

Walks with matching endpoints can be concatenated.
The set of reduced closed walks rooted at~$r$, equipped with concatenation,
is the \emph{fundamental group} $\pi(G,r)$.
The choice of $r$ is not important, since $\pi(G,r)$ is isomorphic to $\pi(G,r')$ as long as $G$ is connected (if $W$ is a walk from $r$ to $r'$, then an isomorphism is given by $C \mapsto W^{-1} C W$).

The idea is that a closed walk $C$ in $G$ corresponds to an element in $\pi(G,r)$ that determines its ``topological type''.
For example, when $G$ is a cycle, $\pi(G,r)$ is isomorphic to $\ZZ$
and the element corresponding to $C$ is its winding number in $G$.
That is, two closed walks rooted at $r$ in $G$ reduce to the same walk
if and only if they wind the same number of times around $G$.
Concatenating two closed walks corresponds to adding their winding numbers.
In general, it is well known that $\pi(G,r)$ is a free group.
For a connected graph $G$ with an arbitrarily chosen spanning tree $T$, the generators of 
$T$ are given by the cycles closed by edges that lie outside of $T$.
Indeed, the graph-theoretic proof of the Nielsen-Schreier theorem is based on this fact,
see for example~\cite{Imrich77}.

\subsection{Covers (for square-free graphs)}
To understand components of $K^{C_n}$ such as $K^{C_n}_\eps$ and to describe homomorphism from them into $K$,
we will replace $K$ with its \emph{covers}, which are larger but simpler graphs.
While we always assume graphs denoted as $K,G,H$ to be finite, covers can be countably infinite.

The \emph{universal cover} $\Uu(K,r)$ of a connected, square-free graph $K$ is
the tree whose vertex-set consists of
all reduced walks starting at a fixed vertex $r$,
and whose edges are the pairs $\{W,W'\}$ such that $W'=W(u,v)$ for some arc $(u,v)$.
For example, the universal cover of a cycle is a bi-infinite path.
To see that the universal cover is a tree, in general, arrange the walks in layers by length;
then in the universal cover, every reduced walk has neighbours only in the next layer
and a single neighbour in the previous layer (except for the empty walk $\eps_r$).

\begin{figure}[bh!]
	\centering
	\vspace*{-3pt}
	\tikzset{v/.style={circle,draw=black!75,inner sep=0pt,minimum size=7pt,fill=yellow!20!gray!50!white}}
\begin{tikzpicture}[every node/.style={font=\footnotesize},scale=1.2]
\begin{scope}
	\begin{scope}[yscale=0.3,xscale=-1]
		\node[v,label=left:$2$] (v0) at (0.2*72:1) {};
		\node[v,label=10:$3$] (v1) at (1.2*72:1) {};
		\node[v,label=right:$4$] (v2) at (2.2*72:1) {};
		\node[v,label=0:$0$] (v3) at (3.2*72:1) {};
		\node[v,label=-180:$1$] (v4) at (4.2*72:1) {};
		\draw (v0)--(v1)--(v2)--(v3)--(v4)--(v0);
	\end{scope}

	\node at (-0.1,0.78) {\normalsize$\downarrow$};

	\def\r{0.1}
	\foreach \i in {0,...,16} {
		\pgfmathtruncatemacro\ip{mod(\i,5)};
		\pgfmathtruncatemacro\im{\i-1};
		\pgfmathtruncatemacro\tmp{\i > 0 && \i < 16};
		\ifnum \tmp > 0
			\node[v] (u\i) at ($(v\ip)+(0,1+\i*\r)$) {};
		\else
			\node[ ] (u\i) at ($(v\ip)+(0,1+\i*\r)$) {};
		\fi
		\pgfmathtruncatemacro\tmp{\i > 1 && \i < 16};
		\ifnum \tmp > 0
			\draw (u\i)--(u\im);
		\else
			\ifnum \i > 0
				\draw[dashed] (u\i)--(u\im);
			\fi
		\fi
	}
	\node[below right of=u3,anchor=north west,node distance=2pt] {$\varepsilon$};
	\node[right of=u2,anchor=west,node distance=2pt] {$(0,4)$};
	\node[below left of=u4,anchor=north east,node distance=2pt] {$(0,1)$};
	\node[below left of=u5,anchor=east,node distance=2pt] {$(0,1)(1,2)$};
	\node[right of=u7,anchor=west,node distance=2pt] (x) {$(0,1)(1,2)(2,3)(3,4)$};
\end{scope}

\begin{scope}[shift={(7,0)}]
	\begin{scope}[yscale=0.3,xscale=-1]
		\node[v] (v0) at (0.15*72:1) {};
		\node[v] (v1) at (1.15*72:1) {};
		\node[v] (v2) at (2.15*72:1) {};
		\node[v] (v3) at (3.15*72:1) {};
		\node[v] (v4) at (4.15*72:1) {};
		\draw (v0)--(v1)--(v2)--(v3)--(v4)--(v0);
	\end{scope}

	\node at (-0.05,0.78) {\normalsize$\downarrow$};

	\def\r{0.1}
	\foreach \i in {4,...,18} {
		\pgfmathtruncatemacro\ip{mod(\i,5)};
		\pgfmathtruncatemacro\im{\i-1};
		\node[v] (u\i) at ($(v\ip)+(0,0.6+\i*\r)$) {};
		\ifnum\i > 4
			\draw (u\i)--(u\im);
		\fi
	}
	\draw (u18)--(u4);
\end{scope}
\end{tikzpicture}
	\vspace*{-8pt}
	\caption{Left: the universal cover of $C_5$, a bi-infinite path. Right: a different cover, $C_{15}$.}
	\label{fig:cycleCover}
\end{figure}
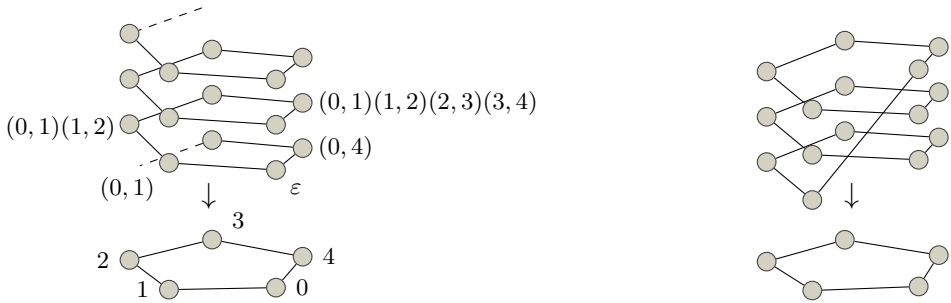

Observe that the map $\tau$, which returns the last vertex of walk,
gives a homomorphism $\tau : \Uu(K,r) \to K$ (in particular the empty walk $\eps_r$ is mapped to $r$).
Moreover, it is a \emph{covering map}: a surjective homomorphism that is locally bijective: the neighbours of each vertex are mapped bijectively to the neighbours of its image.
(In general a \emph{cover} of $K$ is a graph that admits a covering map to $K$; the name \emph{universal cover} comes from the fact that it is itself a cover of every connected cover graph of $K$.
See~\cite{kwak2007graphs} for a detailed account of the theory of graph coverings and their close relation to subgroups of the fundamental group).

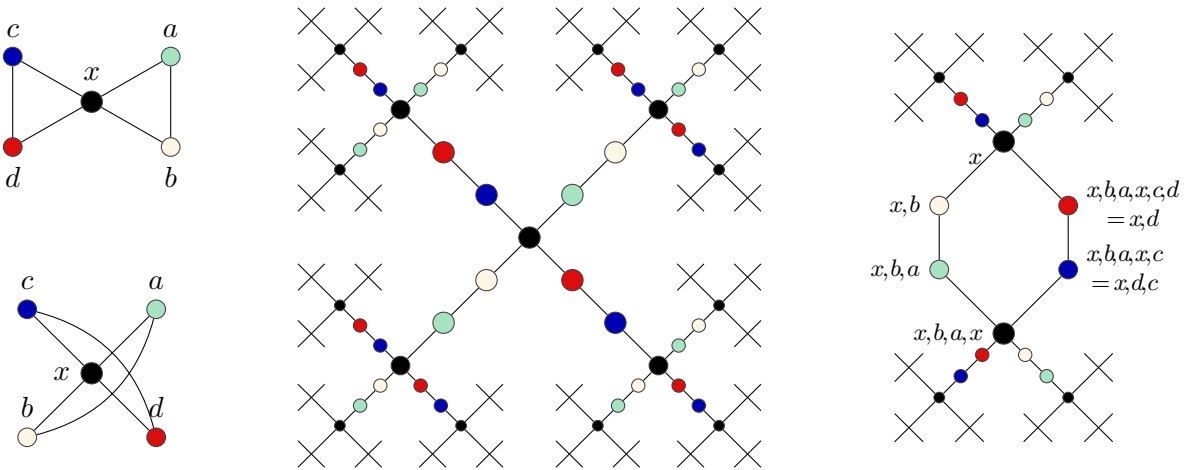
\begin{figure}[b!]
	\centering
	\vspace*{-4pt}
	\tikzset{
	X/.style={circle,draw=black!75,inner sep=0pt,minimum size=8pt,fill=black},
	v/.style={circle,draw=black!75,inner sep=0pt,minimum size=7pt},
	A/.style={v,fill=blue!30!green!35!white},
	C/.style={v,fill=blue!70!black},
	B/.style={v,fill=red!30!yellow!10!white},
	D/.style={v,fill=red!90!black!90!gray},
	e/.style={minimum size=0pt,draw=none,inner sep=0pt,outer sep=0pt}
}
\makeatletter
\newcommand\drawUTree[4]{
	\edef\c{#1}
	\edef\depth{#2}
	\edef\scale{#3}
	\edef\side{#4}
	\pgfmathtruncatemacro{\newdepth}{\depth-1}
	\pgfmathsetmacro{\newscale}{\scale*0.47}
	\pgfmathtruncatemacro{\mmm}{sqrt(\scale)*7}
	\pgfmathtruncatemacro{\nnn}{sqrt(\scale)*8}
	\tikzset{Xstyle/.style={X,minimum size=\mmm pt}}
	\tikzset{nsize/.style={minimum size=\nnn pt}}
	\ifthenelse{\depth < 2}{
		\tikzset{Xstyle/.style={e}}
		\tikzset{nsize/.style={e}}
	}{}
	\def\sidelist{null}
	\ifthenelse{\equal{\side}{a}}{
		\node[Xstyle] (\c\side) at ($(\c)+(\scale,0)$) {};
		\tikzset{S/.style={A, nsize}}
		\tikzset{T/.style={B, nsize}}
	}{
		\g@addto@macro\sidelist{,b}
	}		
	\ifthenelse{\equal{\side}{b}}{
		\node[Xstyle] (\c\side) at ($(\c)+(-\scale,0)$) {};
		\tikzset{S/.style={B, nsize}}
		\tikzset{T/.style={A, nsize}}
	}{
		\g@addto@macro\sidelist{,a}
	}
	\ifthenelse{\equal{\side}{c}}{
		\node[Xstyle] (\c\side) at ($(\c)+(0,\scale)$) {};
		\tikzset{S/.style={C, nsize}}
		\tikzset{T/.style={D, nsize}}
	}{
		\g@addto@macro\sidelist{,d}
	}
	\ifthenelse{\equal{\side}{d}}{
		\node[Xstyle] (\c\side) at ($(\c)+(0,-\scale)$) {};
		\tikzset{S/.style={D, nsize}}
		\tikzset{T/.style={C, nsize}}
	}{
		\g@addto@macro\sidelist{,c}
	}
	\node[S] (\c\side_S) at ($(\c)!0.333!(\c\side)$) {};
	\node[T] (\c\side_T) at ($(\c)!0.666!(\c\side)$) {};
	\draw (\c) -- (\c\side_S) --(\c\side_T) -- (\c\side);
	\ifthenelse{\depth > 0}{
		\foreach \newside in \sidelist {
			\ifthenelse{\equal{\newside}{null}}{}{
				\drawUTree{\c\side}{\newdepth}{\newscale}{\newside}
			}
		}
	}{}
}
\makeatother
\begin{tikzpicture}[scale=1.2]
\begin{scope}[shift={(0,1.5)}]
	\node[X,label=above:$x$] (x) at (   0:0) {};
	\node[A,label=above:$a$] (a) at (  30:1) {};
	\node[B,label=below:$b$] (b) at ( -30:1) {};
	\node[C,label=above:$c$] (c) at ( 150:1) {};
	\node[D,label=below:$d$] (d) at (-150:1) {};
	\draw (x)--(a)--(b)--(x)--(c)--(d)--(x);
\end{scope}
\begin{scope}[shift={(0,-1.5)},rotate=45]
	\node[X,label=left :$x$] (x) at (   0:0) {};
	\node[A,label=above:$a$] (a) at (   0:1) {};
	\node[B,label=above:$b$] (b) at ( 180:1) {};
	\node[C,label=above:$c$] (c) at (  90:1) {};
	\node[D,label=above:$d$] (d) at ( -90:1) {};
	\draw (x)--(a) (a) to[bend left] (b) (b)--(x)--(c) (c) to [bend left] (d) (d)--(x);
\end{scope}
\begin{scope}[shift={(4.8,0)},scale=2.0,rotate=45]
	\node[X] (r) at (0,0) {};
	\drawUTree{r}{3}{1.0}{a}
	\drawUTree{r}{3}{1.0}{b}
	\drawUTree{r}{3}{1.0}{c}
	\drawUTree{r}{3}{1.0}{d}
\end{scope}
\begin{scope}[shift={(10,0)},rotate=45,shift={(0.75,0.75)},scale=2.0]
	\begin{scope}[scale=0.5]
		\tikzset{every node/.style={font=\footnotesize}}
		\node[X,label=190:$x$] (x) at (0,0) {};
		\node[D,label={[align=center]right:\textls[-150]{$x,b,a,x,c,d$}\\\textls[-150]{$=x,d$}}] (d)  at (0,-1) {};
		\node[C,label={[align=center]right:\textls[-150]{$x,b,a,x,c$}\\\textls[-150]{$=x,d,c$}}]   (c)  at (-0.5,-1.5) {};
		\node[X,label={left:\textls[-100]{$x,b,a,x$}}]     (x2) at (-1.5,-1.5) {};
		\node[A,label={left:\textls[-100]{$x,b,a$}}]       (a)  at (-1.5,-0.5) {};
		\node[B,label={left:\textls[-100]{$x,b$}}]         (b)  at (-1,0) {};
	\end{scope}
	\draw (x)--(b)--(a)--(x2)--(c)--(d)--(x);
	\drawUTree{x}{2}{.5}{a}
	\drawUTree{x}{2}{.5}{c}
	\begin{scope}[xscale=-1,rotate=90]
		\drawUTree{x2}{2}{.5}{b}
		\drawUTree{x2}{2}{.5}{d}
	\end{scope}
\end{scope}
\end{tikzpicture}
	\vspace*{-15pt}
	\caption{Left: two drawings of the bowtie $K$. Middle: the universal cover of $K$, isomorphic to the 4-regular tree with each edge subdivided twice; nodes are colored by~$\tau$. 
	Right: the unicyclic cover corresponding to the closed walk on vertices $x,b,a,x,c,d,x$.}
	\label{fig:treeCover}
\end{figure}

For any walk $P$ from $r_1$ to $r_2$ in a connected graph $K$, the map $\alpha_P(W) := PW$ gives an isomorphism between $\Uu(K,r_2)$ and $\Uu(K,r_1)$.
In other words, changing the root amounts to relabeling the vertices of the universal cover.
We will therefore write just $\Uu(K)$ to describe the graph up to isomorphism.

Note that in particular, for a closed walk $C$ rooted at $r$ in $K$, $\alpha_C$ defines an automorphism of $\Uu(K,r)$.
The \emph{unicyclic cover} $\Uu(K,r)_{/C}$
of $K$ is the graph obtained from $\Uu(K,r)$ by identifying the vertices $W$ and $CW$ for each walk $W$; that is, each orbit of $\alpha_C$ becomes a single vertex.
The name \emph{unicyclic} comes from the fact that the unicyclic cover has a single cycle
(unless $C$ reduces to $\eps$, in which case $\Uu(K)_{/C}=\Uu(K)$).
The cycle has length $|C|$ and its vertices are the prefixes of $C$ (assuming $C$ is reduced);
the remaining vertices can again be arranged in layers.
Since $\tau(\alpha_C(W)) = \tau(W)$ for each $W \in \Uu(K)$, the map $\tau:  \Uu(K)_{/C} \rightarrow K$
is well defined and remains a covering map.
See Figure~\ref{fig:treeCover} for an example.

\bigskip
It is well known that a closed walk $C$ in $K$ reduces to $\eps$ if and only if
there is a closed walk $\widetilde{C}$ in $\Uu(K)$, called the \emph{lift} of $C$,
such that $\tau$ maps consecutive vertices and edges of $\widetilde{C}$ to those of $C$.
Indeed, the vertices of the lift $\widetilde{C}$ are given by the reductions of consecutive prefixes of $C$ (in $\Uu(K,r)$ where $r$ is the initial vertex of $C$).

A closed walk $C$ of length $n$ can be viewed as
a homomorphism from the $n$-cycle graph $C_n$ to $K$. 
Thus the closed walk reduces to $\eps$ if and only if the homomorphism factors through $\tau\colon \Uu(K) \to K$ (that is, it is a composition of $\tau$ with some homomorphism $C_n \to \Uu(K)$).
This gives us another way of describing the component $K^C_\eps$ of $K$.
The endpoint map $\tau\colon \Uu(K) \to K$ naturally induces a homomorphism $\Uu(K)^C \to K^C$
and the edges in its image are precisely those of $K^C_\eps$ (because they correspond to closed walks that lift to $\Uu(K)$ and to those that reduce to the $\eps$, respectively).
Therefore $K^C_\eps = \tau(\Uu(K)^C)$.

Similarly, if a closed walk $C$ reduces to $R \cdots R = R^n$, for some $n \in \mathbb{Z}$ and some closed walk $R$ in $K$, then it can be lifted to $\Uu(K)_{/R}$, and the corresponding homomorphism $C_n \to K$ factors through $\tau\colon \Uu(K)_{/R} \to K$. 

\subsection{Properties that imply strong multiplicativity}
As introduced in the previous section, an important part of proving that a graph $K$ is strongly multiplicative is to show that the component $K_\eps^{C_n}$ of $K^{C_n}$ admits a homomorphism to $K$, for odd~$n$.
In fact, the topological approach from~\cite{Wrochna17} shows that is is enough to prove this property, and to show strong multiplicativity for all unicyclic covers of $K$, which are often much simpler than $K$.

\begin{theorem}\label{thm:mainTopoSquareFree}
	Suppose $K$ is a square-free graph such that:
	\begin{itemize}
		\item unicyclic covers of $K$ are strongly multiplicative, and
		\item $\tau({\Uu(K)^{C_n}})$ admits a homomorphism to $K$, for every odd $n$.
	\end{itemize}
	Then $K$ is strongly multiplicative.
\end{theorem}

The proof of Theorem~\ref{thm:mainTopoSquareFree} directly generalizes to many more graphs, as long as we consider walks up to a certain equivalence relation, see Theorem~\ref{thm:mainTopo} later.
We defer the proof, in the more general version, to Appendix~\ref{sec:mainTopo},
as it only rephrases arguments of~\cite{Wrochna17}.

On the other hand, known proofs that a graph $K$ satisfies the two properties from the statement in~Theorem~\ref{thm:mainTopoSquareFree} do not easily generalize to any non-square-free graphs.
In the following sections, we shall give new, direct, combinatorial proofs. 
The strong multiplicativity of unicyclic covers follows directly from that of cycles,
while the homomorphisms  $\tau({\Uu(K)^C}) \to K$ can be constructed
by mapping tuples of vertices in the tree $\Uu(K)$ to their median.
We first give proofs for square-free graphs in Section~\ref{sec:squareFree} and then generalize the methods in Section~\ref{sec:general}.

\bigskip

Just to give a rough idea of the proof of
	Theorem~\ref{thm:mainTopoSquareFree} and its general version,
consider a homomorphism $\phi \colon (G \times C') \cup (C \times H) \to K$.
The product has many squares, which must be collapsed
in the square-free graph $K$ into backtracking 4-walks,
making many cycles equivalent up to reductions of backtrackings.
Because of that, either:
\begin{enumerate}[itemsep=1pt,parsep=0pt,partopsep=0pt,topsep=2pt]
\item all the cycles in $\phi((G \times C') \cup (C \times H))$
wind around the same closed walk $R$ in $K$;
\item cycles in $H$ induce trivial closed walks in $K$; or
\item cycles in $G$ do.
\end{enumerate}
We then conclude that either:
\begin{enumerate}[itemsep=1pt,parsep=0pt,partopsep=0pt,topsep=2pt]
\item all of $\phi$ factors through the unicyclic cover $\Uu(K)_{/R}$ (for some closed walk $R$ in $K$);
\item $\phi^* : H \to K^C$ maps into $K^C_{\eps}$; or
\item $\phi^*: G \to K^{C'}$ maps into $K^{C'}_{\eps}$.
\end{enumerate}
Therefore, it is enough to show the multiplicativity of $\Uu(K)_{/R}$ and to provide a homomorphism from $K^{C}_{\eps} = \tau({\Ueq{K}^C})$ to $K$.

\section{Square-free graphs are strongly multiplicative}\label{sec:squareFree}

Consider a square-free graph $K$.
By Theorem~\ref{thm:mainTopoSquareFree}, to show that $K$ is multiplicative or strongly multiplicative, we need to show that its unicyclic covers are, and that $\tau(\Uu(K)^C)$ admits a homomorphism to $K$.
The first follows simply from the fact that unicyclic covers of square-free graphs have a single cycle.
They retract to it (by mapping infinite trees attached to the cycle into to cycle in any way), which means they are homomorphically equivalent to a cycle graph, which are known to be strongly multiplicative (Theorem~\ref{ocasm}; we do not show a different proof of that case).

It remains to show that $\tau(\Uu(K)^{C_{2k+1}})$ admits a homomorphism to $K$, for every $k\in\NN$.
To that aim, we use properties of medians in trees. Let us define them now.

\pagebreak

Let $T := \Uu(K)$ be a tree (possibly infinite).
Let $w: V(T) \to \NN$ be a nonnegative weight function with odd total weight:
$\sum_{u \in V(T)} w(u) = 2k+1$ for some $k \in \NN$.
For a subtree $T'$ of $T$, we write
$w(T')$ for $\sum_{u \in V(T')} w(u)$.
For an edge $\{u,v\}$ of $T$ we denote $T_{u,v}$ and $T_{v,u}$ the connected components
of $T-\{u,v\}$ containing $u$ and $v$, respectively.
Since the total weight is odd, one of $w(T_{u,v})$, $w(T_{v,u})$ is strictly bigger than the other.
Define $\vec{T^w}$ to be the orientation of $T$ where each edge $\{u,v\}$ is directed towards the bigger part.
Equivalently, $\{u,v\}$ is directed towards $u$ iff $w(T_{u,v}) \geq k+1$.

\begin{lemma}\label{lem:median}
	Let $T$ be a tree and $w: V(T) \to \NN$ be a weight function with odd total weight.
	There is a unique vertex $m \in V(T)$ such that all edges of $\vec{T^w}$ are oriented towards $m$.
\end{lemma}
\begin{proof}
	For any vertex $v \in V(T)$, at most one incident edge is outgoing from $v$ in $\vec{T^w}$ (because each subtree behind an outgoing edge has strictly more than half the total weight).
	All edges outside of the smallest (finite) subtree of $T$ containing the support of $w$ must be oriented towards this subtree, hence there is no infinite outgoing path.
	Hence there is a sink, that is, a vertex $m$ with all incident edges directed towards it.
	All other edges must also be directed towards it, otherwise some vertex in between would have two outgoing edges.
\end{proof}

The {\em median} $\mu(w)$ of $w$ is defined as the unique vertex $m$ given by Lemma~\ref{lem:median}.

Let $x = (x_1, \ldots, x_{2k+1})$ be an element of the $(2k+1)$-th
categorical power $T^{2k+1} = T \times \dots \times T$ of $T$.
That is, elements are $(2k+1)$-tuples of vertices of $T$ and two such tuples are adjacent when they are adjacent on each coordinate.
We write $\vec{T^x}$ for $\vec{T^{w_x}}$ and $\mu(x)$ for $\mu(w_x)$, where $w_x: V(T) \rightarrow \mathbb{N}$ 
is defined by $w_x(u) = | \{i | x_i = u \} |$.
See Figure~\ref{fig:median}.

\begin{lemma} \label{medcontraction}
	If $x = (x_1, \ldots, x_{2k+1})$ and $y = (y_1, \ldots, y_{2k+1})$
	are neighbours in $T^{2k+1}$,
	then $\mu(x)$ and $\mu(y)$ are either adjacent or equal in $T$.
\end{lemma}
\begin{proof}
Let $m = \mu(x)$.
Every edge of $\vec{T^x}$ is oriented towards $m$,
hence for any neighbour $u$ of $m$,
more of the vertices $x_i$ are in $T_{m,u}$ that in $T_{u,m}$.
Since $y_i$ is adjacent to $x_i$ for each~$i$,
for~every neighbour $v$ of $u$ other than $m$,
there are more vertices $y_i$ in $T_{u,v}$ than in $T_{v,u}$.
Thus every edge of $\vec{T^y}$ that is not incident with $m$ is oriented towards $m$.
Therefore $\mu(y)$ is adjacent or equal to $\mu(x)$.
\end{proof}

The only problem is that adjacent tuples may have the same median rather than adjacent medians.
We now turn $\mu$ into a graph homomorphism $\psi\colon T^{2k+1} \to T$,
while trying to maintain two properties that medians have:
they do not depend on the ordering of vertices in the tuple nor on the labels of $V(T)$.

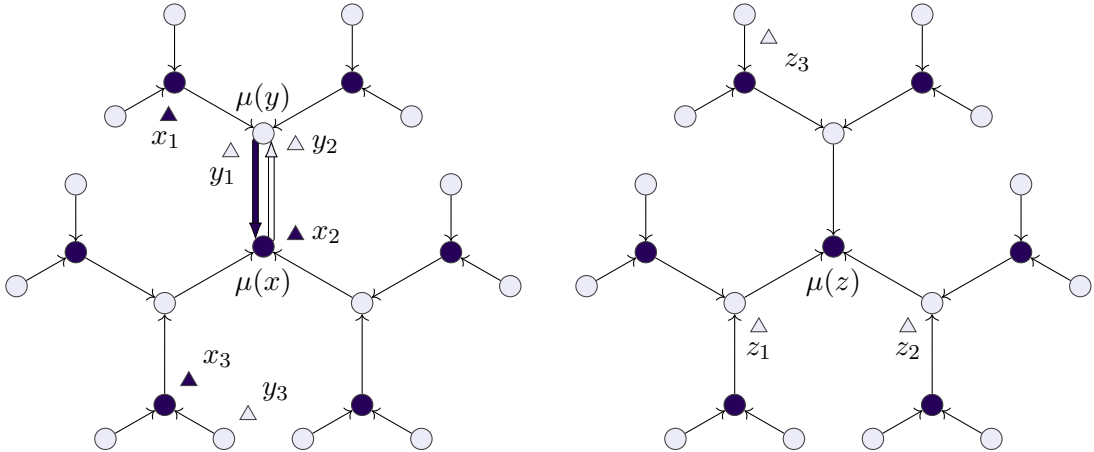
\begin{figure}[hb!]
	\tikzset{
	u/.style={circle,draw=black!75,inner sep=0pt,minimum size=8pt,fill=blue!50!gray!10!white},
	v/.style={u,fill=blue!70!red!50!black,text=white},
	mm/.style={regular polygon,regular polygon sides=3,minimum size=7pt}
}
\begin{tikzpicture}[scale=1.5,StealthFill/.tip={Latex[scale=1.]}, arrows={[round]}]
	\def\p{1.0}
	\def\r{0.9}
	\def\rr{0.6}
	\node[v,label=below:$\mu(x)$] (v)  at (   0:0) {};
	\node[u,label=above:$\mu(y)$] (v0) at (  90:\p) {};
	\node[u] (v1) at ( -30:\p) {};
	\node[u] (v2) at (-150:\p) {};
	\draw[<-] (v) edge (v1) (v) edge (v2);

	\begin{scope}[on background layer]
	\draw (v.north east) edge[-{StealthFill[fill=blue!50!gray!10!white]},double equal sign distance,line cap=round] (v0.south east);

	\draw[<-,>={StealthFill[fill=blue!70!red!50!black]},double equal sign distance,double=blue!70!red!50!black,line cap=round]  (v.north west) -- (v0.south west);
	\end{scope}

	\node[v] (v01) at ($(v0)+( 150:\r)$) {};
	\node[v] (v02) at ($(v0)+(  30:\r)$) {};
	\node[v] (v12) at ($(v1)+(  30:\r)$) {};
	\node[v] (v10) at ($(v1)+( -90:\r)$) {};
	\node[v] (v20) at ($(v2)+( -90:\r)$) {};
	\node[v] (v21) at ($(v2)+( 150:\r)$) {};
	\draw[<-] (v0) edge (v01) (v0) edge (v02) (v1) edge (v12) (v1) edge (v10) (v2) edge (v20) (v2) edge (v21);
	\node[u] (v012) at ($(v01)+(-150:\rr)$) {};
	\node[u] (v010) at ($(v01)+(  90:\rr)$) {};
	\node[u] (v020) at ($(v02)+(  90:\rr)$) {};
	\node[u] (v021) at ($(v02)+( -30:\rr)$) {};
	\node[u] (v120) at ($(v12)+(  90:\rr)$) {};
	\node[u] (v121) at ($(v12)+( -30:\rr)$) {};
	\node[u] (v101) at ($(v10)+( -30:\rr)$) {};
	\node[u] (v102) at ($(v10)+(-150:\rr)$) {};
	\node[u] (v201) at ($(v20)+( -30:\rr)$) {};
	\node[u] (v202) at ($(v20)+(-150:\rr)$) {};
	\node[u] (v212) at ($(v21)+(-150:\rr)$) {};
	\node[u] (v210) at ($(v21)+(  90:\rr)$) {};
	\draw[<-]
		(v01) edge (v012)
		(v01) edge (v010)
		(v02) edge (v020)
		(v02) edge (v021)
		(v12) edge (v120)
		(v12) edge (v121)
		(v10) edge (v101)
		(v10) edge (v102)
		(v20) edge (v201)
		(v20) edge (v202)
		(v21) edge (v212)
		(v21) edge (v210)
	;
	\def\d{0.3}
	\node[v,mm,label=45:$x_3$] (x3) at ($(v20.center) +(45:\d)$) {};
	\node[u,mm,label=45:$y_3$] (y3) at ($(v201.center)+(45:\d)$) {};
	\node[v,mm,label=0:$x_2$] (x2) at ($(v.center)   +( 20:\d)$) {};
	\node[u,mm,label=0:$y_2$] (y2) at ($(v0.center)  +(-20:\d)$) {};
	\node[v,mm,label={[xshift=-1]-90:$x_1$}] (x1) at ($(v01.center) +(-100:\d)$) {};
	\node[u,mm,label={[xshift=-3]-90:$y_1$}] (y1) at ($(v0.center)  +(-150:1.1*\d)$) {};

\begin{scope}[shift={(5,0)}]
	\node[v,label=below:$\mu(z)$] (v)  at (   0:0) {};
	\node[u] (v0) at (  90:\p) {};
	\node[u] (v1) at ( -30:\p) {};
	\node[u] (v2) at (-150:\p) {};
	\draw[<-] (v) edge (v0) (v) edge (v1) (v) edge (v2);

	\node[v] (v01) at ($(v0)+( 150:\r)$) {};
	\node[v] (v02) at ($(v0)+(  30:\r)$) {};
	\node[v] (v12) at ($(v1)+(  30:\r)$) {};
	\node[v] (v10) at ($(v1)+( -90:\r)$) {};
	\node[v] (v20) at ($(v2)+( -90:\r)$) {};
	\node[v] (v21) at ($(v2)+( 150:\r)$) {};
	\draw[<-] (v0) edge (v01) (v0) edge (v02) (v1) edge (v12) (v1) edge (v10) (v2) edge (v20) (v2) edge (v21);
	\node[u] (v012) at ($(v01)+(-150:\rr)$) {};
	\node[u] (v010) at ($(v01)+(  90:\rr)$) {};
	\node[u] (v020) at ($(v02)+(  90:\rr)$) {};
	\node[u] (v021) at ($(v02)+( -30:\rr)$) {};
	\node[u] (v120) at ($(v12)+(  90:\rr)$) {};
	\node[u] (v121) at ($(v12)+( -30:\rr)$) {};
	\node[u] (v101) at ($(v10)+( -30:\rr)$) {};
	\node[u] (v102) at ($(v10)+(-150:\rr)$) {};
	\node[u] (v201) at ($(v20)+( -30:\rr)$) {};
	\node[u] (v202) at ($(v20)+(-150:\rr)$) {};
	\node[u] (v212) at ($(v21)+(-150:\rr)$) {};
	\node[u] (v210) at ($(v21)+(  90:\rr)$) {};
	\draw[<-]
		(v01) edge (v012)
		(v01) edge (v010)
		(v02) edge (v020)
		(v02) edge (v021)
		(v12) edge (v120)
		(v12) edge (v121)
		(v10) edge (v101)
		(v10) edge (v102)
		(v20) edge (v201)
		(v20) edge (v202)
		(v21) edge (v212)
		(v21) edge (v210)
	;
	\def\d{0.3}
	\node[u,mm,label=-45:$z_3$] (z3) at ($(v010.center) +(-45:\d)$) {};
	\node[u,mm,label=below:$z_2$] (z2) at ($(v1.center)   +( -135:\d)$) {};
	\node[u,mm,label=below:$z_1$] (z1) at ($(v2.center) +(-45:\d)$) {};
\end{scope}
\end{tikzpicture}
	\caption{Left: medians of adjacent 3-tuples $x$ and $y$. Right: a tuple $z$ with an ``incorrectly coloured median'', also adjacent to $x$.}
	\label{fig:median}
\end{figure}

\begin{lemma}\label{lem:Tpower}
	Let $T=\Uu(K)$ for a square-free graph $K$ and let $k \in \NN$.
	There is a homomorphism $\psi\colon T^{2k+1} \to T$ which is:
	\begin{itemize}
		\item 	ordering-invariant: $\psi(x_{\sigma(1)},\dots,x_{\sigma(2k+1)})=\psi(x_1,\dots,x_{2k+1})$ for every permutation $\sigma$,
		\item covariant under automorphisms $\alpha_C$ of $\Ueq{K}$, for each closed walk $C$ in $K$:\\
		$\psi(\alpha_C(x_1),\dots,\alpha_C(x_{2k+1})) = \alpha_C(\psi(x_1,\dots,x_{2k+1}))$.
	\end{itemize}
\end{lemma}
\begin{proof}
To fix the problem with equal medians, observe that $T$ is bipartite, split into two colour classes.
We first define $\psi$ for monochromatic tuples only.
Let us say that a monochromatic tuple $x \in V(T^{2k+1})$ has a \emph{correctly coloured} median
if $\mu(x)$ is in the same colour class as all the vertices in $x$.
We claim that if $\mu(x)$ is incorrectly coloured and $y$ is a neighbour of $x$,
then $\mu(y)$ is correctly coloured and equal to $\mu(x)$.

Indeed, for any neighbour $v$ of $\mu(x)$, $T_{\mu(x),v}$ contains more values $x_i$ than $T_{v,\mu(x)}$ (with multiplicities taken into account).
But for each $x_i \in T_{\mu(x),v}$, its neighbour $y_i$ is also in $T_{\mu(x),v}$,
since otherwise $x_i$ would be equal $\mu(x)$, which is impossible since $x$ has an incorrectly coloured median.
Thus $T_{\mu(x),v}$ also contains more values $y_i$ than $T_{v,\mu(x)}$.
Therefore $\mu(y) = \mu(x)$,
and $y$ has a correctly coloured median.  

This means that the following is a valid homomorphism from the subgraph of monochromatic tuples in $T^{2k+1}$ to $T$:

$$\psi(x) = \left \{ \begin{array}{l}
\mbox{$\mu(x)$ if $x$ has a correctly coloured median,} \\
\mbox{any neighbour of  $\mu(x)$ otherwise.}
\end{array} \right.$$

We need however to define the choice of ``any neighbour'' in $\psi$ consistently so that it is covariant under automorphisms of $T$.
This is impossible in general, but recall that $T=\Uu(K)$ and that we only require covariance under automorphisms of the form $\alpha_C$ for closed walks $C$ in~$K$.
These automorphisms preserve the endpoint map $\tau\colon \Uu(K) \to K$, that is $\tau(\alpha_C(x))=\tau(x)$.
We hence choose any fixed ordering of $V(K)$ and for incorrectly colored medians, we choose the neighbour $y$ of $\mu(x)$ that has the smallest value $\tau(y)$; since $\tau$ is locally bijective, this is well defined.

In the end we will not need to define $\psi$ for non-monochromatic tuples, but let us do this for completeness.
If a tuple in $T^{2k+1}$ has $a$ vertices in one colour class and $b$ in the other, then exactly one of these numbers is odd, say $a$, and every neighbour of the tuple will also have exactly $a$ neighbours in one of the colour classes.
Thus $T^{2k+1}$ can be divided into components according to the value of the odd number $a$.
Non-monochromatic tuples, that is, those in components with $a \neq 2k+1$, can be mapped to their odd-length monochromatic sub-tuple, that is, to $T^a$.
Composing with $\psi$ as defined above for monochromatic tuples, we obtain a homomorphism from all of $T^{2k+1}$ to $T$ with the same properties.
\end{proof}

From these properties of $\psi$, we directly obtain the homomorphism we sought (we can even relax ordering invariance to cyclic invariance).
This part works in full generality.
\begin{lemma}\label{lem:polymorphismIsEnough}
	For a graph $K$ and $k\in\NN$,
	suppose there is a homomorphism $\psi:\Ueq{K}^{2k+1} \to \Ueq{K}$ which is:
	\begin{itemize}
	\item invariant under cyclic shifts: $\psi(x_1,x_2,\dots,x_{2k+1})=\psi(x_2,\dots,x_{2k+1},x_1)$, and
	\item covariant under automorphisms $\alpha_C$ of $\Ueq{K}$, for each closed walk $C$ in $K$:\\
		$\psi(\alpha_C(x_1),\dots,\alpha_C(x_{2k+1})) = \alpha_C(\psi(x_1,\dots,x_{2k+1}))$.
	\end{itemize}
	Then there is a homomorphism $\tau(\Ueq{K}^{C_{2k+1}}) \to K$.
\end{lemma}
\begin{proof}
Let $C=C_{2k+1}$.
We label the vertices of $C$ consecutively with the elements of $\mathbb{Z}_{2k+1}$.
The homomorphism $\psi\colon \Uu(K)^{2k+1} \to \Uu(K)$
can easily be adapted into a homomorphism from $\Uu(K)^C$.
Namely, every vertex $h$ of $\Uu(K)^C$ is by definition a function $h:V(C) \to V(\Uu(K))$.
We can associate to it the tuple $x(h)_i := h(i)$ in $\Uu(K)^{2k+1}$.
Now if $h$ is adjacent to $h'$ in $\Uu(K)^C$, this means that $h(i)$ is adjacent to $h'(i-1)$ and $h'(i+1)$ in $\Uu(K)$.
Thus $x(h)$ is not necessarily adjacent to $x(h')$ in $\Uu(K)^{2k+1}$, but its shift by one is: $x(h)_i$ is adjacent to $x(h')_{i+1}$.
Since $\psi$ is invariant under cyclic shifts, $\psi(x(h))$ is adjacent to $\psi(x(h'))$.
Therefore $\psi'(h) := \psi(x(h))$ is a homomorphism $\Uu(K)^C \to \Uu(K)$, which is again invariant under cyclic shifts and automorphism-covariant.

Recall that the endpoint map $\tau$ gives a homomorphism $\tau\colon \Uu(K) \to K$, which induces a homomorphism $\tau\colon \Uu(K)^C \to K^C$.
The automorphism-covariance of $\psi'\colon \Uu(K)^C \to \Uu(K)$ allows us to define a homomorphism $\psi'' \colon \tau(\Uu(K)^C) \to K$ simply as follows:
for a vertex $h$ of $\tau(\Uu(K)^C)$, let $\psi''(h) := \psi'(\widetilde{h})$ for any $\widetilde{h}  \in \tau^{-1}(h)$.
This does not depend on the choice of $\widetilde{h}$, 
because for any two choices $\widetilde{h}_1,\widetilde{h}_2$ in $\tau^{-1}(h)$
there is an automorphism mapping one to the other.
Indeed, for any walk $\widetilde{W}$ between $\widetilde{h}_1$ and $\widetilde{h}_2$,
$\tau$ maps its vertices to a closed walk $W$ from $\tau(\widetilde{h}_1)=h$ to $\tau(\widetilde{h}_2)=h$
and hence $\alpha_{W}$ is a suitable automorphism of $\Uu(K)$.
\end{proof}

All in all, we obtained a homomorphism from $\tau(\Ueq{K}^{C_{2k+1}})  \subseteq K^C$ to $K$ by viewing each edge of it as a closed walk that can be lifted to the universal cover $\Ueq{K}$, viewing odd and even vertices of that walk as adjacent $(2k+1)$-tuples, and then mapping each tuple to its median, possibly correcting its colour class.
This concludes the proof of the following:
\begin{theorem}[\cite{Wrochna17}] \label{sfgasm}
All square-free graphs are strongly multiplicative.
\end{theorem}

As we will see, the proof generalizes to some more graphs.
In the language of constraint satisfaction, the homomorphism required in the statement of Lemma~\ref{lem:polymorphismIsEnough} is the same as a so called \emph{cyclic polymorphism} of the structure on $V(\Ueq{K})$ with the adjacency relation $E(\Ueq{K})$ and the relations $\alpha_C$, for each closed walk $C$ in $K$ (it suffices to take finitely many: one for each edge outside of a spanning tree of $K$).
Lemma~\ref{lem:Tpower} provides such a polymorphism for the case of square-free $K$.

The algebraic approach to constraint satisfaction gives tools for showing that cyclic polymorphisms exist or do not exist, see e.g.~\cite{BartoK12}.
Since it has been a very fruitful approach for proving various dichotomies,
we hope that it might shed a light on the limits of our method as well.
We explore this in a follow-up paper~\cite{followup}.

\section{Taking quotients by squares}\label{sec:squares}
\subsection{\texorpdfstring{$\eq$}{box}-equivalence of walks}
In general graphs, instead of considering walks up to reductions, we use a coarser equivalence relation,
first studied by Matsushita~\cite{Matsushita12}.
For a graph $G$, we let $\eq$ be the smallest equivalence relation on walks in $G$ such that: 
\begin{align*}
	(u,v)(v,u)\ \eq\ \eps &&&\text{for }\{u,v\}\in E(G),\\
	(a,b)(b,c)(c,d)(d,a)\ \eq\ \eps &&&\text{for }\{a,b\},\{b,c\},\{c,d\},\{d,a\}\in E(G),\\
	W_1 W_2 \eq W_1' W_2' &&&\text{whenever }W_1 \eq W_1',\;W_2 \eq W_2'. 
\end{align*}
We say $W_1 \eq W_2$ are \emph{$\eq$-equivalent} (pronounced \emph{box-equivalent}, in reference to the closely related \emph{box complex} of a graph)
and write $[W]$ for the equivalence class of a walk $W$ in $G$.
Note that equivalent walks have the same initial and final vertices, so the endpoint $\tau([W])$ is well defined. Similarly, equivalent walks have the same length parity.

The \emph{$\eq$-fundamental group} $\pi(G,r)_{/\eq}$ is the set of equivalence classes of closed walks rooted at $r$, with concatenation as multiplication.
Inversion is simply $((v_0,v_1)\dots(v_{\ell-1},v_\ell))^{-1} = (v_\ell,v_{\ell-1})\dots(v_1,v_0)$.
In other words, this is the quotient of $\pi(G,r)$ by all squares (that is, $(a,b)(b,c)(c,d)(d,a)$ is identified with $\eps$, for each square $a,b,c,d$).
As before, for any walk $W$ from $r_1$ to $r_2$, the map $C \mapsto W^{-1}CW$ is an isomorphism between $\pi(G,r_1)_{/\eq}$ and $\pi(G,r_2)_{/\eq}$.

\smallskip

The crucial property of quotients by $\eq$ is that they behave much better with respect to products. 
When $G$ is the product of two odd cycles,  $\pi(G,r)_{/\eq}$ is isomorphic to $\ZZ \times \ZZ$ and more generally,
the $\eq$-fundamental group of $G\times H$ is almost isomorphic to the direct product of the groups of $G$ and $H$.
More precisely, the type of a walk in $G\times H$ is determined just by the pair of types of its projections in $G$ and in $H$:
\begin{lemma}[\cite{Wrochna17}]\label{lem:groupProduct}
The function $[C] \mapsto ([C|_G],[C|_H])$ from $\pi(G \times H, (g,h))_{/\eq}$ to $\pi(G,g)_{/\eq} \times \pi(H,h)_{/\eq}$ is well defined and gives an injective group homomorphism.
\end{lemma}

In fact, for the subgroups given by closed walks of even length (or equivalently, for the groups $\pi(-\times K_2,r)/_{\eq}$), this is an isomorphism between the group of the graph product and the direct product of groups (Corollary~5.5 in~\cite{Matsushita13} and Proposition~6.12 in~\cite{Matsushita12}).

Note also that a graph homomorphism $\phi: G \rightarrow K$ maps closed walks to closed walks and thus naturally induces
a group homomorphism $\overline{\phi}$ from $\pi(G,r)_{/\eq}$ to $\pi(K,\phi(r))_{/\eq}$, for each $r \in V(G)$.

\subsection{Covers (in general)}
For a connected graph $K$, the \emph{$\eq$-universal cover} $\Ueq{K,r}$ is the graph whose vertex-set consists of
all $\eq$-equivalence classes of walks starting at a fixed vertex $r$,
and whose edges are the pairs $\{[W],[W']\}$ such that there exists an arc $(u,v)$ with $[W']=[W(u,v)]$.
Alternatively, one could define a standard universal cover of $K$ as before, as a tree with reduced walks as vertices, and the $\eq$-universal cover as the quotient of it by $\eq$; 
however, the standard universal cover has no use for us in non-square-free graphs.

As before, the map $\tau$ gives a covering map $\Ueq{K,r} \to K$. For any walk $P$ from $r_1$ to $r_2$, $\alpha_P([W]) := [PW]$ defines an isomorphism between $\Ueq{K,r_2}$ and $\Ueq{K,r_1}$, so we will just write $\Ueq{K}$.
For a closed walk $C$ in $K$, $\alpha_C$ is an automorphism, which allows to define the \emph{$\eq$-unicyclic cover} $\Ueq{K}_{/C}$ as before, by identifying each orbit of $\alpha_C$ into a single vertex.

The $\eq$-universal cover is always somewhat tree-like in the sense that it has has a trivial $\eq$-fundamental group; in fact, any two walks in $\Ueq{K}$ between the same endpoints are $\eq$-equivalent.
Note this also implies that $\Ueq{K}$ is bipartite.
Similarly, $\eq$-unicyclic covers $\Ueq{K}_{/C}$ can be shown to have $\eq$-fundamental groups isomorphic to $\ZZ$, with (the lift of) $C$ as a generator.
We give proofs and more facts about covers in Appendix~\ref{app:covers} to justify the definitions
and refer to~\cite{Matsushita12} for a detailed treatment and generalizations.

Theorem~\ref{thm:mainTopoSquareFree} generalizes to any graph with a free $\eq$-fundamental group:

\begin{theorem}\label{thm:mainTopo}
	Let $K$ be a graph whose $\eq$-fundamental group is free. Suppose that:
	\begin{itemize}
		\item the $\eq$-unicyclic cover $\Ueq{K}_{/R}$ is strongly multiplicative, for every odd-length closed walk $R$ in $K$, and
		\item $\tau({\Ueq{K}^{C_n}})$ admits a homomorphism to $K$, for every odd $n$.
	\end{itemize}
	Then $K$ is strongly multiplicative.
\end{theorem}

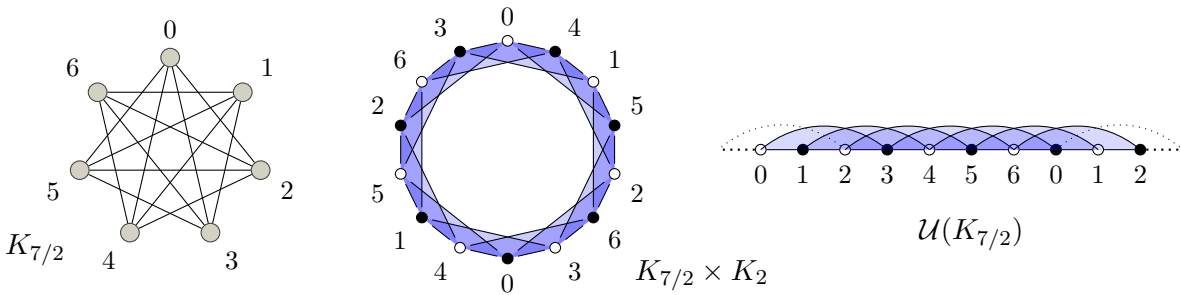
\begin{figure}[H]
	\vspace*{-5pt}
	\tikzset{ %
	v/.style={circle,draw=black!75,inner sep=0pt,minimum size=7pt,fill=yellow!20!gray!50!white},
	L/.style={circle,draw=black,fill=white,inner sep=0pt,minimum size=4pt},
	R/.style={L,fill=black},
	h/.style={draw=black!50},
}
\begin{tikzpicture}[scale=1.11]
\begin{scope}[shift={(0,0)}]
	\foreach \i in {0,...,6} 
	{
		\node[v,label={90-360*\i/7:\small$\i$}] (u\i) at (90-360*\i/7 : 1.1) {};
	}
	\foreach \i in {0,...,6} 
	{
		\pgfmathtruncatemacro\ipp{mod(\i+2,7)};
		\pgfmathtruncatemacro\ippp{mod(\i+3,7)};
		\draw (u\i)--(u\ipp) (u\i)--(u\ippp);
	}
	\node at (-1.6,-1.2) {$K_{7/2}$};
\end{scope}
\begin{scope}[shift={(4,0)}]
	\foreach \i in {0,...,6} 
	{
		\node (u\i) at (90-360*\i/7 : 1.3)  {};
		\node (v\i) at (-90-360*\i/7 : 1.3)  {};
	}
	\foreach \i in {0,...,6}  
	{
		\pgfmathtruncatemacro\im{mod(\i+6,7)};
		\pgfmathtruncatemacro\ipp{mod(\i+2,7)};
		\pgfmathtruncatemacro\ippp{mod(\i+3,7)};
		\fill[fill=blue,opacity=0.2] (u\i.center) -- (v\ippp.center) -- (u\im.center) -- (v\ipp.center) -- cycle;
		\fill[fill=blue,opacity=0.2] (v\i.center) -- (u\ippp.center) -- (v\im.center) -- (u\ipp.center) -- cycle;
	}
	\foreach \i in {0,...,6} 
	{
		\node[L,label={{90-360*\i/7}:\small$\i$}] at (u\i) {};
		\node[R,label={{-90-360*\i/7}:\small$\i$}] at (v\i) {};
	}
	\foreach \i in {0,...,6}  
	{
		\pgfmathtruncatemacro\im{mod(\i+6,7)};
		\pgfmathtruncatemacro\ipp{mod(\i+2,7)};
		\pgfmathtruncatemacro\ippp{mod(\i+3,7)};
		\draw (u\i)--(v\ipp) (u\i)--(v\ippp);
		\draw (v\i)--(u\ipp) (v\i)--(u\ippp);

	}
	\node at (2.3,-1.5) {$K_{7/2} \times K_2$};
\end{scope}
\begin{scope}[shift={(7,0)}]
\tikzmath{
	int \i;
	for \i in {0,...,9}{
			\LR = mod(\i+1,2) ? "L" : "R";
			int \itau;
			\itau = mod(\i,7);
			{ \node[\LR,label={below:\small$\itau$}] (u\i) at (\i*0.5,0)  {}; };
			if \i >= 1 then {
				int \iprev;
				\iprev = \i - 1;
				{ \draw (u\i)--(u\iprev); };
			};
			if \i >= 3 then {
				int \iprev;
				\iprev = \i - 3;
				{
					\draw[fill=blue,fill opacity=0.15] (u\i.center) to[bend right=40] (u\iprev.center);
				};
			};
	};
}
\draw[thick, dotted] (u0)--++(-0.5,0);
\draw[dotted] (u2) to[bend right=40] ($(u0)+(-0.5,0)$);
\draw[thick, dotted] (u9)--++(0.5,0);
\draw[dotted] (u7) to[bend left=40] ($(u9)+(0.5,0)$);

\node at (10*0.25,-1) {$\mathcal{U}(K_{7/2})$};
\end{scope}
\end{tikzpicture}
	\vspace*{-15pt}
	\caption{The circular clique $K_{7/2}$, its bipartite double cover and its universal cover (with squares highlighted and vertices labelled with $\tau$).}
	\vspace*{-10pt}
\end{figure}

Unfortunately the $\eq$-universal cover of the clique $K_4$, for example, is just $K_4 \times K_2$. This is because any two walks between the same vertices and with the same length parity can be shown to be $\eq$-equivalent.
For any odd-length closed walk $C$, the $\eq$-unicyclic cover $\Ueq{K_4}_{/C}$ will be just $K_4$ again, which unfortunately makes our approach is vacuous for $K_4$.

We can show the converse for the second property:
\begin{lemma}\label{lem:constConverse}
Let $K$ be a strongly multiplicative graph.
Then $\tau(\Ueq{K}^{C_{n}})$ admits a homomorphism to $K$, for every odd $n$.
\end{lemma}

See Appendix~\ref{sec:mainTopo} for the proofs.
In Appendix~\ref{app:other} we also show that the second property is equivalent to the one mentioned in the introduction:
\begin{lemma}\label{lem:constProperty}
The following are equivalent, for any graph $K$:
\begin{itemize}
	\item $\tau(\Ueq{K}^{C_{n}})$ admits a homomorphism to $K$, for each odd $n$.
	\item if $G$ is non-bipartite and $H$ is a connected graph with a vertex $h$ such that there exists a homomorphism $\phi\colon G \times H \to K$ with $\phi(-,h)$ constant, then $H$ admits a homomorphism to $K$. 
\end{itemize}
\end{lemma}

\subsection{Lifts and \texorpdfstring{$K^{C_n}_\eps$}{K\^{}C\_eps}}
For general $K$ and odd $n$, we define $K^{C_n}_\eps$ to be the subgraph of $K^{C_n}$ on those edges whose corresponding closed walks are $\eq$-equivalent to $\eps$ in $K$ (this does not depend on the orientation we consider for the edge; we ignore isolated vertices).

Similarly as for square-free graphs, a closed walk $C$ is $\eq$-equivalent to $\eps$ in $K$ if and only if it lifts to $\Ueq{K}$ (meaning there is a closed walk $\tilde{C}$ in $\Ueq{K}$ such that $\tau$ maps the vertices of $\tilde{C}$ to those of $C$, in order).
Because of that, $\tau(\Ueq{K}^{C_n})$ is equal to $K^{C_n}_\eps$.
Formally:

\begin{lemma}[Appendix~\ref{app:covers}]\label{lem:universalIsEps}
	$K^{C_n}_\eps = \tau(\Ueq{K}^{C_n})$, for every graph $K$ and odd integer $n$.

	That is, for any arc $(h,h')$ of $K^{C_n}$,
	the closed walk in $K$ of length $2n$ corresponding to $(h,h')$ is $\eq$-equivalent to $\eps$
	\  iff \  the edge $\{h,h'\}$ is in the image of $\tau \colon \Ueq{K}^{C_n} \to K^{C_n}$.
\end{lemma}

While in general $K^{C_n}_\eps$ is not connected, we show the following:
\begin{lemma}[Appendix~\ref{app:other}]\label{lem:components}
For a connected graph $K$ and an odd $n$, the subgraph $K^{C_n}_\eps$ of $K^{C_n}$ is a sum of connected components of $K^{C_n}$, including in particular the component containing constant functions (though possibly more).
Moreover, there is an odd $m\geq n$ such that $K^{C_n}_\eps$ admits a homomorphism into the connected component of constants in $K^{C_m}_\eps$.
\end{lemma}

Finally, let us mention that the deeper reason behind the definitions in this section is that they reflect the topology of the so called \emph{box complex} $\mathrm{Box}(G)$ of a graph $G$ (also known as the neighbourhood complex, or $\textrm{Hom}(K_2,G)$).
Formally, the fundamental group of $\mathrm{Box}(G)$ (in the topological sense) is isomorphic to $\pi(G \times K_2,r)_{/\eq}$, while the fundamental group of the quotient of $\mathrm{Box}(G)$ by its $\ZZ_2$-action is isomorphic to $\pi(G,r)_{/\eq}$ (for any root vertex $r$ in a connected graph $G$).
More generally, two walks in $G$ are $\eq$-equivalent iff the corresponding curves in $\mathrm{Box}(G)$ are homotopic rel endpoints.
Moreover, the universal cover of $\mathrm{Box}(G)$ (in the topological sense) is homotopy-equivalent to $\mathrm{Box}(\Ueq{K})$. 
See Matou\v{s}ek's book~\cite{matousek2008using} for definitions and an introduction to topological methods in graph theory, and see~\cite{Matsushita12} for proofs (Matsushita writes $\simeq_2$, 2-fundamental groups, 2-covers for $\eq$, $\eq$-fundamental groups and $\eq$-covers; he also considers a natural generalization, essentially making $2r$-cycles equivalent to $\eps$, for any chosen $r\in\NN$).

\section{Generalizing square-free graphs}\label{sec:general}
\subsection{Square-dismantlable graphs}
The following definition gives a simple, but fairly large class of graphs
where $\eq$-equivalence classes are easy to understand and amenable to our approach,
at least to some extent.
We will call a graph $K$ {\em square dismantlable} if there is a sequence 
$K = K_0, K_1, \ldots, K_n$ of graphs where for $i = 1, \ldots, n$,
$K_i$ is obtained from $K_{i-1}$ by removing an edge $e_i$ of $K_{i-1}$
which is in only one square of $K_{i-1}$, and where $K_n$ is square-free.
We then call $(e_1, \ldots, e_n)$ a {\em square-dismantling sequence} of $K$,
and $K_n$ a {\em square-free kernel} of $K$.

Square dismantlable graphs are a basic example of graphs with a free $\eq$-fundamental group.
Indeed, if a graph $K-e$ is obtained from a graph $K$ by removing an edge $e$
that occurred in only one square,
then there is a natural homomorphism $\rho$
from $\pi(K,r)_{/\eq}$ to $\pi(K-e,r)_{/\eq}$
(for any $r \in V(K)$)
obtained by replacing each occurrence of $e$ (including $e^{-1}$) by the detour walk
$\rho(e)$ along the three edges of the unique square containing it.
The inverse homomorphism simply maps the $\eq$-equivalence class of a walk
in $K-e$ to the $\eq$-equivalence class of the same walk in $K$.
Thus $\pi(K,r)_{/\eq}$ is isomorphic to $\pi(K-e,r)_{/\eq}$.
Iterating this, for a square-dismantlable graph $K$ with square-free kernel $L$, we have
$\pi(K,r)_{/\eq} \simeq \pi(L,r)_{/\eq} \simeq \pi(L,r)$;
in particular it is a free group.
See Figure~\ref{fig:moserCover} for an example.

Moreover, $\Ueq{K}$ contains $\Ueq{L}$ as a subgraph and is obtained from it, roughly speaking, by adding back the edges $e_1,\dots,e_n$ to it.
To see this, consider again a graph $K$ with an edge $e$ that occurs in only one square
and define $\rho$ as before, as a map from arbitrary walks in $K$ to walks in $K-e$.
It follows from the definition that $\rho(W) \eq W$ in $K$, for any walk $W$.
For two walks $W,W'$ in $K$, one can easily check that $W \eq W'$ in $K$ if and only if $\rho(W) \eq \rho(W')$ in $K-e$.
Indeed, it suffices to check that the image of the backtracking $e e^{-1}$ and of square $e \rho(e)^{-1}$ (essentially the only square containing $e$) reduces to $\eps$ in $K-e$.
Therefore, $\rho$ gives a bijection between equivalence classes of walks in $K$ and in $K-e$.
An equivalence class in $K$ is mapped simply to the subset of walks in it that do not use the edge $e$.

Therefore, $\rho$ gives a bijection between the vertices of $\Ueq{K}$ and $\Ueq{K-e}$,
and we can use the same walks in $K-e$ to represent them.
The only difference is that $\Ueq{K}$ has more edges:
it is obtained from $\Ueq{K-e}$ by adding an edge from $[W]$ to $[We]=[W\rho(e)]$ for every walk $W$ in $K-e$ to which $e$ can be concatenated.
Specifically, if $e=(a,d)$ and $\rho(e)=(a,b)(b,c)(c,d)$,
then we add an edge from each vertex $\tilde{a} \in \tau^{-1}(a)$ (which means an equivalence class of walks ending in $a$) 
to the endpoint $\tilde{d}$ of the unique walk
	$(\tilde{a},\tilde{b})(\tilde{b},\tilde{c})(\tilde{c},\tilde{d})$   in $\Ueq{K-e}$
such that $\tilde{b} \in \tau^{-1}(b)$, $\tilde{c} \in \tau^{-1}(c)$, $\tilde{d} \in \tau^{-1}(d)$
(that is, if $\tilde{a}=[W]$, then $\tilde{b}=[W(a,b)]$, $\tilde{c}=[W(a,b)(b,c)]$, $\tilde{d}=[W(a,b)(b,c)(c,d)]=[W\rho(e)]$).
In short, we add an edge between the two endpoints of every possible lift of $\rho(e)$ in $\Ueq{K-e}$.

Iterating this, for a square-dismantlable graph $K$
with square-dismantling sequence $(e_1,\dots,e_n)$ to the square-free kernel $L$,
we have a function $\rho$ that maps each walk in $K$ to the walk in $L$ obtained by
iteratively replacing every occurrence of $e_i$ by the detour of length $3$ in $K_i = K_{i-1} - e_i$.
The graph $\Ueq{K}$ is obtained from the tree $\Uu(L)$
by adding an edge between the endpoints of every possible lift of $\rho(e_i)$, for $i=1,\dots,n$.
See Figure~\ref{fig:bigCover} for an example.

\begin{figure}
\definecolor{c0}{RGB}{187,187,187} 
\definecolor{c4}{RGB}{ 34,136, 51} 
\definecolor{c5}{RGB}{204,187, 68} 
\definecolor{c6}{RGB}{238,102,119} 
\definecolor{c1}{RGB}{170, 51,119} 
\definecolor{c2}{RGB}{ 68,119,170} 
\definecolor{c3}{RGB}{102,204,238} 
\tikzset{
	v/.style={circle,inner sep=0pt,minimum size=9pt,fill=yellow!20!gray!80},
	c0/.style={fill=c0},
	c1/.style={fill=c1},
	c2/.style={fill=c3},
	c3/.style={fill=c4},
	c4/.style={fill=c6},
	ca/.style={fill=c2},
	cb/.style={fill=c5},
}
\begin{tikzpicture}[yscale=1.2,xscale=1.3]
\begin{scope}[shift={(0,+1.6)},rotate=90,scale=0.7]
	\node[v,c0,label=above:$0$] (v0) at (0*72:2) {};
	\node[v,c1,label=above:$1$] (v1) at (1*72:2) {};
	\node[v,c2,label=170  :$2$] (v2) at (2*72:2) {};
	\node[v,c3,label= 10  :$3$] (v3) at (3*72:2) {};
	\node[v,c4,label=above:$4$] (v4) at (4*72:2) {};
	\node[v,ca,label=100  :$a$] (va) at ($(  70:0.57)-(0.4,0)$) {};
	\node[v,cb,label= 80  :$b$] (vb) at ($( -70:0.57)-(0.4,0)$) {};
	\draw (v0)--(v1)--(v2)--(v3)--(v4)--(v0);
	\draw (v4)--(vb)--(v0)--(va)--(v1);
	\draw[dashed] (v3)--(vb) (v2)--(va);

	\draw[thick,red!50!yellow!90!black,rounded corners,smooth] (v0.south)--(va.south east)--(v2.north east)--(v3.north)--(vb.east)--(v4.west)--(v0.south east);
\end{scope}
\begin{scope}[shift={(0,-1.6)},rotate=90,scale=0.7]
	\node[v,c0] (v0) at (0*72:2) {};
	\node[v,c1] (v1) at (1*72:2) {};
	\node[v,c2] (v2) at (2*72:2) {};
	\node[v,c3] (v3) at (3*72:2) {};
	\node[v,c4] (v4) at (4*72:2) {};
	\node[v,ca] (va) at ($(  70:0.57)-(0.4,0)$) {};
	\node[v,cb] (vb) at ($( -70:0.57)-(0.4,0)$) {};
	\draw (v0)--(v1)--(v2)--(v3)--(v4)--(v0);
	\draw (v4)--(vb)--(v0)--(va)--(v1);
	\draw[dashed] (v3)--(vb) (v2)--(va);

	\draw[thick,blue!90!black,rounded corners,smooth] (v0)--(va.north)--($(v0.west)+(0.1,0.2)$)--(v1.west)--(v2.south west)--(v3.south east)--(v4.east)--($(v0.north east)+(0,-0.2)$)--(vb.north east)--(v4.west)--(v0.south);
\end{scope}
\begin{scope}[shift={(5.7,0.3)},scale=-0.8,every node/.style={v,minimum size=7pt}]
	\def\r{0.45}
	\def\d{1.3}
	\newcommand\drawme[1]{
		\node[c0] (#1_0) at ( 0 ,0) {};
		\node[c1] (#1_1) at (-1, 1) {};
		\node[ca] (#1_a) at (-1,-1) {};
		\node[c2] (#1_2) at (-2, 0) {};
		\node[cb] (#1_b) at ( 1,-1) {};
		\node[c4] (#1_4) at ( 1, 1) {};
		\node[c3] (#1_3) at ( 2, 0) {};
		\draw (#1_0)--(#1_1)--(#1_2) (#1_a)--(#1_0)--(#1_4)--(#1_3) (#1_b)--(#1_0);
		\draw[densely dashed] (#1_2)--(#1_a) (#1_b)--(#1_3);
	}
	\makeatletter
	\newcommand\ifstreq[2]{%
		\ifnum\pdf@strcmp{\unexpanded{#1}}{\unexpanded{#2}}=0 %
			\expandafter\@firstoftwo
		\else
			\expandafter\@secondoftwo
		\fi
	}
	\makeatother
	\newcommand\drawdots[2]{
		\drawme{#1};
		\ifstreq{#2}{1}{}{\draw[densely dotted] (#1_1)-- ++(-0.5*\d, \d);}
		\ifstreq{#2}{4}{}{\draw[densely dotted] (#1_4)-- ++( 0.5*\d, \d);}
		\ifstreq{#2}{a}{}{\draw[densely dotted] (#1_a)-- ++(-0.5*\d,-\d);}
		\ifstreq{#2}{b}{}{\draw[densely dotted] (#1_b)-- ++( 0.5*\d,-\d);}
		\ifstreq{#2}{2}{}{\draw[densely dotted,path fading=east] (#1_2)-- ++(-0.7*\d,  0);}
		\ifstreq{#2}{3}{}{\draw[densely dotted,path fading=west] (#1_3)-- ++( 0.7*\d,  0);}
	}
	\begin{scope}[shift={( 0  , 0  )},scale= 1,every node/.style={v,minimum size=9pt}]\drawme{v};   \end{scope}
	\begin{scope}[shift={(-1.5, 3  )},scale=\r]\drawdots{p1}{a};\end{scope}
	\draw (v_1)--(p1_a);
	\begin{scope}[shift={( 1.5, 3  )},scale=\r]\drawdots{p4}{b};\end{scope}
	\draw (v_4)--(p4_b);
	\begin{scope}[shift={(-1.5,-3  )},scale=\r]\drawdots{pa}{1};\end{scope}
	\draw (v_a)--(pa_1);
	\begin{scope}[shift={( 1.5,-3  )},scale=\r]\drawdots{pb}{4};\end{scope}
	\draw (v_b)--(pb_4);

	\begin{scope}[shift={(-3.7, 0  )},scale=\r]\drawdots{p2}{3};
		\begin{scope}[shift={(1.1,-3)},scale=0.6,every node/.style={v,minimum size=6pt,draw=none}]\drawdots{p2b}{4};\end{scope}
		\draw (p2_b)--(p2b_4);
	\end{scope}
	\draw (v_2)--(p2_3);
	\begin{scope}[shift={( 3.7, 0  )},scale=\r]\drawdots{p3}{2};\end{scope}
	\draw (v_3)--(p3_2);
	
	\draw[thick,red!50!yellow!90!black,rounded corners,smooth] (v_0.north)--(v_a.north)--(v_2.north)--(p2_3.north)--(p2_b.north west)--(p2b_4.west)--(p2b_0.west);
	\draw[thick,blue!90!black,rounded corners,smooth] (v_0.east)--(v_a.south)--($(v_0.east)-(0.1,0)$)--(v_1.north)--(v_2.south)--(p2_3.south)--(p2_4.south)--(p2_0.east)--(p2_b.north east)--(p2b_4.east)--(p2b_0.south);
\end{scope}
\end{tikzpicture}
	\vspace*{5pt}
	\caption{The universal cover of the Moser spindle, with a square-dismantling sequence $(\{a,2\},\{b,3\})$. In orange: a walk $W$ on vertices $0,a,2,3,b,4,0$. In blue: $\rho(W)$ on vertices $0,a,0,1,2,3,b,4,0$.}
	\label{fig:moserCover}
\end{figure}
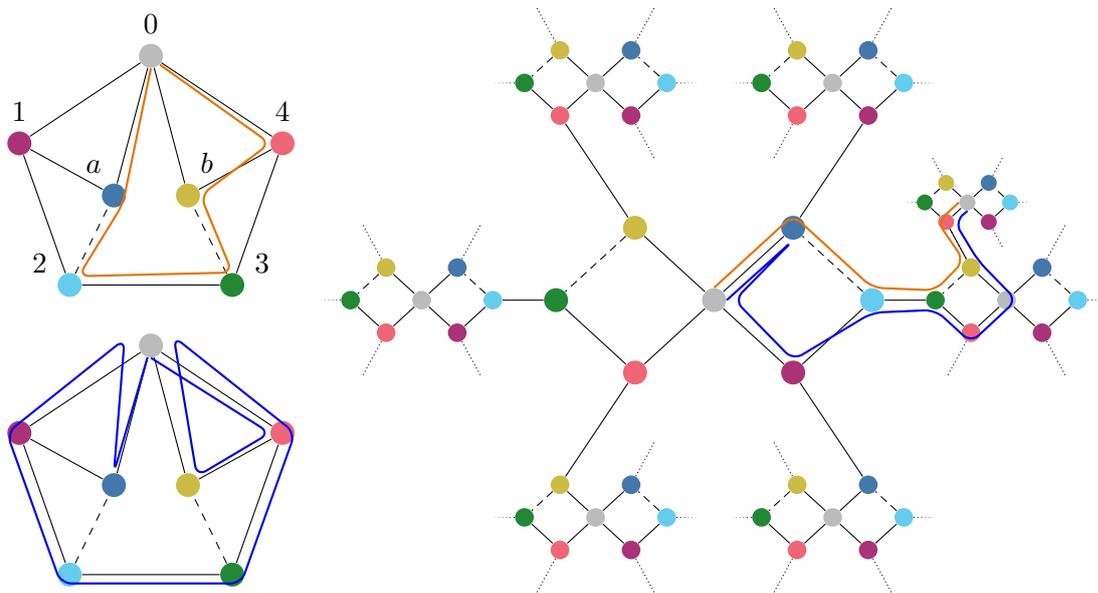

\begin{figure}
\definecolor{c0}{RGB}{187,187,187} 
\definecolor{c4}{RGB}{ 34,136, 51} 
\definecolor{c5}{RGB}{204,187, 68} 
\definecolor{c6}{RGB}{238,102,119} 
\definecolor{c1}{RGB}{170, 51,119} 
\definecolor{c2}{RGB}{ 68,119,170} 
\definecolor{c3}{RGB}{102,204,238} 
\tikzset{
	v/.style={circle,inner sep=0pt,minimum size=10pt,fill=white},
	A/.style={circle,inner sep=0pt,minimum size=10pt,fill=blue!30!gray},
	B/.style={circle,inner sep=0pt,minimum size=10pt,fill=orange!40!white},
	c0/.style={fill=c0},
	cp1/.style={fill=c1},
	c2/.style={fill=c2},
	cp3/.style={fill=c3},
	c4/.style={fill=c4},
	cp5/.style={fill=c5},
	c6/.style={fill=c6},
	c1/.style={fill=c1!33!c6!33!c2},
	c3/.style={fill=c3!33!c2!33!c4},
	c5/.style={fill=c5!33!c4!33!c6},
	c13/.style={fill=c1!66!c3},
	c31/.style={fill=c3!66!c1},
	c35/.style={fill=c3!66!c5},
	c53/.style={fill=c5!66!c3},
	c51/.style={fill=c5!66!c1},
	c15/.style={fill=c1!66!c5},
	mm/.style={minimum size=6.5pt},
}
\begin{tikzpicture}[scale=1,semithick]
\begin{scope}
	\def\rm{0.9}
	\def\ri{2}
	\def\ro{2.7}
	\def\re{2.5}
	\node[v,c0] (v0) at (   0:0) {};
	\node[v,c2] (v2) at (  30:\rm) {};
	\node[v,c4] (v4) at ( -90:\rm) {};
	\node[v,c6] (v6) at ( 150:\rm) {};
	\draw (v0)--(v2) (v0)--(v4) (v0)--(v6);
	\node[v,c1] (v1) at (  90:\ri) {};
	\node[v,c3] (v3) at ( -30:\ri) {};
	\node[v,c5] (v5) at ( 210:\ri) {};	
	\draw (v1)--(v2) (v3)--(v4) (v5)--(v6);
	\draw[dashed] (v2)--(v3) (v4)--(v5) (v6)--(v1);
	\node[v,cp1] (p1) at (  90:\ro) {};
	\node[v,cp3] (p3) at ( -30:\ro) {};
	\node[v,cp5] (p5) at ( 210:\ro) {};
	\draw (v1)--(p1) (v3)--(p3) (v5)--(p5);
	\node[v,c13] (p13) at (  50:\re) {};
	\node[v,c31] (p31) at (  10:\re) {};
	\node[v,c35] (p35) at ( -70:\re) {};
	\node[v,c53] (p53) at (-110:\re) {};
	\node[v,c51] (p51) at ( 170:\re) {};
	\node[v,c15] (p15) at ( 130:\re) {};
	\draw (p1)--(p13)--(p31)--(p3)--(p35)--(p53)--(p5)--(p51)--(p15)--(p1);
\end{scope}
\begin{scope}[shift={(8.4,2)},scale=0.6]
	\def\rm{0.9}
	\def\ri{2}
	\def\ro{1.7}
	\def\re{1.5}
	\def\ree{1.1}
	\def\rv{0.8}
	\def\rvv{0.7}
	\def\rvvv{1.0}
	\node[B,c0] (v0) at (   0:0) {};
	\node[A,c2] (v2) at (  30:\rm) {};
	\node[A,c4] (v4) at ( -90:\rm) {};
	\node[A,c6] (v6) at ( 150:\rm) {};
	\draw (v0)--(v2) (v0)--(v4) (v0)--(v6);
	\node[B,c1] (v1) at (  90:\ri) {};
	\node[B,c3] (v3) at ( -30:\ri) {};
	\node[B,c5] (v5) at ( 210:\ri) {};	
	\draw (v1)--(v2) (v3)--(v4) (v5)--(v6);
	\draw[dashed] (v2)--(v3) (v4)--(v5) (v6)--(v1);
	\node[A,cp1] (p1) at ($(v1)+(  90:\ro)$) {};
	\node[A,cp3] (p3) at ($(v3)+( -30:\ro)$) {};
	\node[A,cp5] (p5) at ($(v5)+( 210:\ro)$) {};
	\draw (v1)--(p1) (v3)--(p3) (v5)--(p5);
	\node[B,c13] (p13) at ($(p1)+(   50:\re)$) {};
	\node[B,c31] (p31) at ($(p3)+(   10:\re)$) {};
	\node[B,c35] (p35) at ($(p3)+(  -70:\re)$) {};
	\node[B,c53] (p53) at ($(p5)+( -110:\re)$) {};
	\node[B,c51] (p51) at ($(p5)+(  170:\re)$) {};
	\node[B,c15] (p15) at ($(p1)+(  130:\re)$) {};
	\draw (p1)--(p13) (p31)--(p3)--(p35) (p53)--(p5)--(p51) (p15)--(p1);
	\node[A,c31] (pp13) at ($(p13)+(   50:\ree)$) {};
	\node[A,c13] (pp31) at ($(p31)+(   10:\ree)$) {};
	\node[A,c53] (pp35) at ($(p35)+(  -70:\ree)$) {};
	\node[A,c35] (pp53) at ($(p53)+( -110:\ree)$) {};
	\node[A,c15] (pp51) at ($(p51)+(  170:\ree)$) {};
	\node[A,c51] (pp15) at ($(p15)+(  130:\ree)$) {};
	\draw (p13)--(pp13) (p31)--(pp31) (p35)--(pp35) (p53)--(pp53) (p51)--(pp51) (p15)--(pp15);
	\node[B,cp3] (v13) at ($(pp13)+(   30:\rv)$) {};
	\node[B,cp1] (v31) at ($(pp31)+(   30:\rv)$) {};
	\node[B,cp5] (v35) at ($(pp35)+(  -90:\rv)$) {};
	\node[B,cp3] (v53) at ($(pp53)+(  -90:\rv)$) {};
	\node[B,cp1] (v51) at ($(pp51)+(  150:\rv)$) {};
	\node[B,cp5] (v15) at ($(pp15)+(  150:\rv)$) {};
	\draw (v13)--(pp13) (v31)--(pp31) (v35)--(pp35) (v53)--(pp53) (v51)--(pp51) (v15)--(pp15);
	\node[A,c35] (v1v35) at ($(v13)+(  90:\rvv)$) {};
	\node[A,c15] (v3v15) at ($(v31)+( -30:\rvv)$) {};
	\node[A,c51] (v3v51) at ($(v35)+( -30:\rvv)$) {};
	\node[A,c31] (v5v31) at ($(v53)+(-150:\rvv)$) {};
	\node[A,c13] (v5v13) at ($(v51)+(-150:\rvv)$) {};
	\node[A,c53] (v1v53) at ($(v15)+(  90:\rvv)$) {};
	\draw (v13)--(v1v35) (v31)--(v3v15) (v35)--(v3v51) (v53)--(v5v31) (v51)--(v5v13) (v15)--(v1v53);
	\node[] (p1v35) at ($(v1v35)+(  90:\rvvv)$) {};
	\node[] (p3v15) at ($(v3v15)+( -30:\rvvv)$) {};
	\node[] (p3v51) at ($(v3v51)+( -30:\rvvv)$) {};
	\node[] (p5v31) at ($(v5v31)+(-150:\rvvv)$) {};
	\node[] (p5v13) at ($(v5v13)+(-150:\rvvv)$) {};
	\node[] (p1v53) at ($(v1v53)+(  90:\rvvv)$) {};
	\draw[dotted] (p1v35)--(v1v35) (p3v15)--(v3v15) (p3v51)--(v3v51) (p5v31)--(v5v31) (p5v13)--(v5v13) (p1v53)--(v1v53);
	\pgfmathparse{0.5*(\ri+\ro)}
	\edef\rio{\pgfmathresult}
	\def\drawme{%
		\node[A,c0,mm] (v0) at (   0:0) {};
		\node[B,c2,mm] (v2) at (  30:\rm) {};
		\node[B,c4,mm] (v4) at ( -90:\rm) {};
		\node[B,c6,mm] (v6) at ( 150:\rm) {};
		\draw[thick] (v0)--(v2) (v0)--(v4) (v0)--(v6);
		\node[A,c1,mm] (v1) at (  90:\ri) {};
		\node[A,c3,mm] (v3) at ( -30:\ri) {};
		\node[A,c5,mm] (v5) at ( 210:\ri) {};	
		\draw (v1)--(v2) (v3)--(v4) (v5)--(v6);
		\draw[densely dashed] (v2)--(v3) (v4)--(v5) (v6)--(v1);
		\node[B,cp1,mm] (p1) at ($(v1)+(  90:\ro)$) {};
		\node[B,cp3,mm] (p3) at ($(v3)+( -30:\ro)$) {};
		\node[B,cp5,mm] (p5) at ($(v5)+( 210:\ro)$) {};
		\draw (v1)--(p1) (v3)--(p3) (v5)--(p5);
		\node[] (p13) at ($(p1)+(   30:\re)$) {};
		\node[] (p31) at ($(p3)+(   30:\re)$) {};
		\node[] (p35) at ($(p3)+(  -90:\re)$) {};
		\node[] (p53) at ($(p5)+(  -90:\re)$) {};
		\node[] (p51) at ($(p5)+(  150:\re)$) {};
		\node[] (p15) at ($(p1)+(  150:\re)$) {};
		\draw[dotted] (p1)--(p13) (p31)--(p3)--(p35) (p53)--(p5)--(p51) (p15)--(p1);
	}
	\begin{scope}[shift={($(v13)+( -30:\rio)$)},rotate=180,scale=0.5,on background layer]
		\drawme
	\end{scope}
	\begin{scope}[shift={($(v31)+(  90:\rio)$)},rotate=180,scale=0.5,on background layer]
		\drawme
	\end{scope}
	\begin{scope}[shift={($(v35)+(-150:\rio)$)},rotate=180,scale=0.5,on background layer]
		\drawme
	\end{scope}
	\begin{scope}[shift={($(v53)+( -30:\rio)$)},rotate=180,scale=0.5,on background layer]
		\drawme
	\end{scope}
	\begin{scope}[shift={($(v51)+(  90:\rio)$)},rotate=180,scale=0.5,on background layer]
		\drawme
	\end{scope}
	\begin{scope}[shift={($(v15)+(-150:\rio)$)},rotate=180,scale=0.5,on background layer]
		\drawme
	\end{scope}
\end{scope}
\end{tikzpicture}
	\vspace*{-25pt}
	\caption{The universal cover of another square-dismantlable graph.}
	\label{fig:bigCover}
\end{figure}
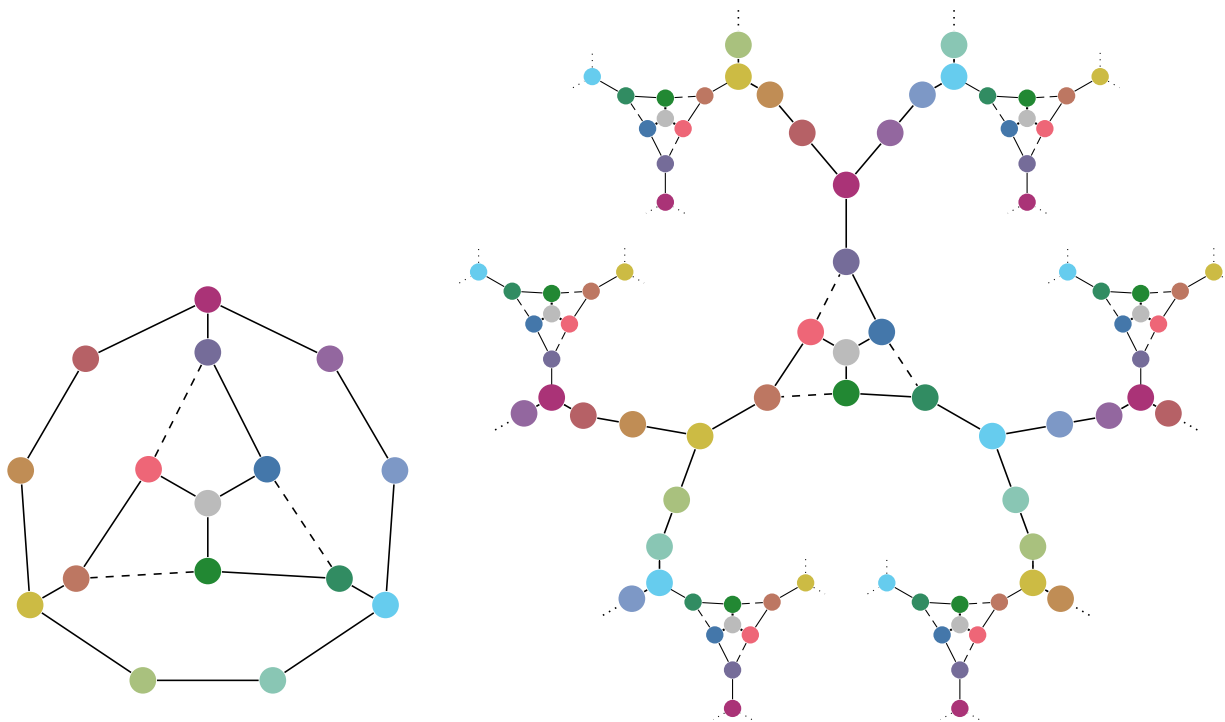

\bigskip 

Examples of square-dismantlable graphs include all subcubic core graphs (i.e. graphs with vertex degrees $\leq 3$ that are not homomorphically equivalent to any proper subgraph) other than $K_4$.
In fact, a straightforward case analysis shows that in every subcubic graph except $K_4$ and the cube graph $K_4 \times K_2$, either every square has a private edge (an edge not included in any other square), or the graph can be folded.
Here we say that a graph can be \emph{folded} if the neighbourhood of some vertex is contained in the neighbourhood of another; this implies a homomorphism into the subgraph with the first vertex removed, which means such a graph is not a core.

\subsection{Graphs with each edge in at most one square} 
Let $S^{(1)}, \ldots, S^{(n)}$ be the squares in a graph $K$ with each edge in at most one square.
Let the consecutive vertices of $S^{(i)}$ be $a^{(i)},b^{(i)},c^{(i)},d^{(i)}$
(with the starting vertex arbitrarily chosen).
Then $(\{a^{(1)},d^{(1)}\}, \ldots, \{a^{(n)},d^{(n)}\})$
is a square-dismantling sequence of $K$.
Let $L$ be the corresponding square-free kernel.
Then $\Ueq{K}$ is obtained from $\Uu(L)$ by
adding edges $\{u,x\}$ whenever $u$ and $x$ are joined by a path $u,v,w,x$
such that $\tau(u) = a^{(i)}, \tau(v) = b^{(i)}, \tau(w) = c^{(i)}$ and $\tau(x) = d^{(i)}$ for some $i$.

We will use the terminology of the proof of Theorem~\ref{sfgasm}.
The argument is in fact similar.
We first show the analogue of Lemma~\ref{lem:Tpower} by ``correcting medians'' more carefully
and conclude that $\tau(\Ueq{K}^C) \to K$ via Lemma~\ref{lem:polymorphismIsEnough}.

\begin{lemma}
Let $k\in\NN$ and let $K$ be a graph with each edge in at most one square.
There is a homomorphism $\psi:\Ueq{K}^{2k+1} \to \Ueq{K}$ that is ordering-invariant and covariant under automorphisms $\alpha_C$ of $\Ueq{K}$ for closed walks $C$ in $K$.
\end{lemma}
\begin{proof}
The homomorphism that we construct is again closely related to the median $\mu$ in $\Uu(L)$,
where $L$ is the square-free kernel of $K$.
As before, we need to define $\psi$ only for monochromatic tuples,
we put $\psi(x) = \mu(x)$ if $\mu(x)$ is correctly coloured
and we will put $\psi(x)$ to a certain neighbour of $\mu(x)$ otherwise.

To choose the neighbour $\psi(x)$ of an incorrectly colored median $\mu(x)$,
consider the squares that contain $\mu(x)$ in $\Ueq{K}$.
Let $v_i,w_i$ be the neighbours of $\mu(x)$ contained in the $i$-th square that contains $\mu(x)$;
let $u_1,u_2,\dots$ be the remaining neighbors of $\mu(x)$ in $\Ueq{K}$.
Observe that the connected components of $\Ueq{K}$ with $\mu(x)$ removed
are precisely $T_{u_i,\mu(x)}$ and $T_{v_i,\mu(x)} \cup T_{w_i,\mu(x)}$
(since $v_i,w_i$ have a common neighbour other than $\mu(x)$).
If one of these components contains at least $k+1$ vertices of $x_i$,
it must be a component $T_{v_i,\mu(x)} \cup T_{w_i,\mu(x)}$,
since if it were $T_{u_i,\mu(x)}$, the median would have been on $u_i$.
In that case, we put $\psi(x)$ to an arbitrary neighbour in that component, that is,
to one of $v_i,w_i$.
Otherwise, we put $\psi(x_i)$ to an arbitrary neighbour of $\mu(x_i)$.
To fix the arbitrary choice consistently with respect to automorphisms $\alpha_C$, the neighbour we choose should have the smallest possible value of $\tau$, according to an arbitrary fixed ordering of $V(K)$.
We now show that $\psi$ is a homomorphism.

Let $(x,y)$ be an arc of $\Ueq{K}^{2k+1}$.
If $\mu(x) = \mu(y)$, then one of them is correctly coloured and the other is incorrectly coloured and moved to a neighbour, so $\psi(x)$ and $\psi(y)$ are adjacent.

If $\mu(x) \neq \mu(y)$,
let $P$ be the shortest path in $\Uu(L)$ from $\mu(x)$ to $\mu(y)$.
Then the orientation $\overrightarrow{\Uu(L)^x}$ (towards $\mu(x)$) differs from $\overrightarrow{\Uu(L)^y}$ (towards $\mu(y)$) precisely on $P$.
Let $\{\mu(x),u\}$ and $\{v,\mu(y)\}$ be the first and last edge of $P$, respectively.
Their opposing orientations imply that $T_{\mu(x),u}$ contains at least $k+1$ of the vertices $x_i$, but the disjoint subtree $T_{\mu(y),v}$ contains at least $k+1$ of the vertices $y_i$.
Thus for some $i \in [2k+1]$, $x_i$ is in $T_{\mu(x),u}$, but $y_i$ is in $T_{\mu(y),v}$.
Equivalently, the shortest path from $x_i$ to $y_i$ in $\Uu(L)$ contains (in order) $\mu(x)$ and $\mu(y)$.

Since $x_i$ and $y_i$ are adjacent in $\Ueq{K}$, they are at distance 1 or 3 in $\Uu(L)$.
If they are adjacent in $\Uu(L)$, then $\mu(x)=x_i$ and $\mu(y)=y_i$ (since they occur on $P$, in order).
Thus both are correctly coloured and $\psi(x)=\mu(x)$ is adjacent to $\psi(y)=\mu(y)$.
If they are at distance 3, then the path in between in $\Uu(L)$ together with the edge $\{x_i,y_i\}$ in $\Ueq{K}$ forms a square, which contains the vertices $\mu(x)$ and $\mu(y)$.
We claim that $\psi(x),\psi(y)$ are still contained in that square.
Since they are distinctly coloured, this guarantees they are adjacent.

Indeed, let $v_i,w_i$ be the neighbours of $\mu(x)$ in that square.
One of these is equal to $u$, the neighbour of $\mu(x)$ on $P$, say $v_i$.
The edge $\{\mu(x),v_i\}$ is thus oriented in the direction of $v_i$ in $\overrightarrow{\Uu(L)^y}$
(since $P$ is oriented from $\mu(x)$ to $\mu(y)$).
Hence at least $k+1$ of the vertices $y_i$ are contained in $T_{v_i,\mu(x)}$.
Therefore, at least $k+1$ of the vertices $x_i$ are contained in the neighbourhood of $T_{v_i,\mu(x)}$ in $\Ueq{K}$, which is contained in $T_{v_i,\mu(x)} \cup T_{w_i,\mu(x)} \cup \{\mu(x)\}$.
If $\mu(x)$ is correctly coloured then $\psi(x)=\mu(x)$ is still on the square anyway.
Otherwise none of the $x_i$ are on $\mu(x)$,
which means at least $k+1$ are in $T_{v_i,\mu(x)} \cup T_{w_i,\mu(x)}$,
in which case $\psi(x)$ is put to $v_i$ or $w_i$.
Symmetrically, $\psi(y)$ is also still on the square.
\end{proof}

\begin{corollary} \label{sductok}
Let $K$ be a graph with each edge in at most one square. For every $k\in\NN$, 
there is a homomorphism $\tau(\Ueq{K}^{C_{2k+1}}) \to K$.
\end{corollary}

Next, we show that unicyclic covers retract to cycles, and are therefore homomorphically equivalent to cycle graphs, which are known to be strongly multiplicative.

\begin{lemma} \label{sducroc}
	Let $K$ be a graph with every edge in at most one square,
	and $L$ its square-free kernel. 
	Then for every odd-length closed walk $R$ in $K$, 
	$\Ueq{K}_{/R}$ retracts to an odd-length cycle.
\end{lemma}
\begin{proof}
Define the {\em anchor} $\alpha(u)$ of a vertex $u$ of $\Uu(L)_{/R}$
as the vertex  closest to $u$ on the unique cycle of $\Uu(L)_{/R}$.
There is a homomorphism  $\phi: \Uu(L)_{/R} \rightarrow \Uu(L)_{/R}$
where  $\phi(u)$ is $\alpha(u)$ if $u$ is at even distance from $\alpha(u)$,
and $\phi(u)$ is the neighbour of $\alpha(u)$ on the unique path
from $u$ to $\alpha(u)$ if $u$ is at odd distance from $\alpha(u)$.
It is easy to see that $\phi$ preserves the edges of $\Ueq{K}_{/R}$ as well.
Indeed, endpoints of an edge in $\Ueq{K}_{/R}$ are at distance 1 or 3 in $\Uu(L)_{/R}$.
If both endpoints have the same anchor,
then their distances from this anchor have different parities.
If the endpoints of an edge have adjacent anchors,
then either one of them is its own anchor and the other is at distance two from its anchor,
or both are adjacent to their anchor.
Finally, if the endpoints of an edge have anchors at distance at least two,
then they are fixed by~$\psi$.
Therefore $\psi$ retracts $\Ueq{K}_{/R}$ to the subgraph induced on vertices at distance at most one from the unique cycle in $\Uu(L)_{/R}$.

Now let $G$ be the restriction of $\Ueq{K}_{/R}$ to the image of $\psi$.
Label the vertices of the unique cycle of  $\Uu(L)_{/R}$ consecutively $u_0,u_1,\ldots,u_{m-1}$.
A vertex $v$ not on the cycle is called {\em linked clockwise}
if $\alpha(v) = u_i$ and $v$ has a neighbour $w$ with $\alpha(w) \in \{u_{i+1}, u_{i+2}\}$.
Note that $\{v,\alpha(v)\}$ is in at most one square, $v$ has no other neighbour that $u_i,w$.
Hence the map $\psi'$ defined by 
$$\psi'(v) = \left \{ 
\begin{array}{ll}
v & \mbox{if $v \in \{u_0,u_1,\dots\}$, } \\
u_{i+1}  & \mbox{if $v \not\in \{u_0,u_1,\dots\}$, $\alpha(v) = u_i$, and $v$ is linked clockwise,} \\
u_{i-1} & \mbox{otherwise}
\end{array}
\right.
$$
is a retraction onto the subgraph induced on $\{u_0,u_1,\dots\}$.
This subgraph is a cycle on  $u_0, u_1, \ldots, u_{m-1}$,
with possible additional edges of the form $\{u_i, u_{i+3}\}$,
where $u_{i+1}$ and $u_{i+2}$ can be in no other square and thus have degree $2$.
Mapping these degree $2$ vertices to $u_{i+3}$ and $u_i$ respectively
retracts the graph to an odd-length cycle in $\Ueq{K}_{/R}$ (which admits the homomorphism $\tau$ to~$K$).
\end{proof}

Corollary~\ref{sductok} and Lemma~\ref{sducroc} yield the following.
\begin{theorem} \label{gwospvam}
Every graph $K$ with each edge in at most one square is multiplicative.
\end{theorem}

\subsection{Extensions and obstacles} 
The proof of Lemma~\ref{sductok} adapts to many more graphs by ad-hoc case analysis.
It would be tempting to conjecture that all square-dismantlable graphs,
or at least all subcubic core graphs other than $K_4$, are strongly multiplicative.
However, in future work~\cite{followup} we show that this is not the case,
e.g. the graph $K_A$ from Figure~\ref{ambig} is not strongly multiplicative.
Here, let us just sketch two examples for which our attempts fail.
Consider the graph $K_A$.
It is a subcubic core graph, square-dismantlable to
its square-free kernel $L_A$ spanned by the solid edges on the figure.

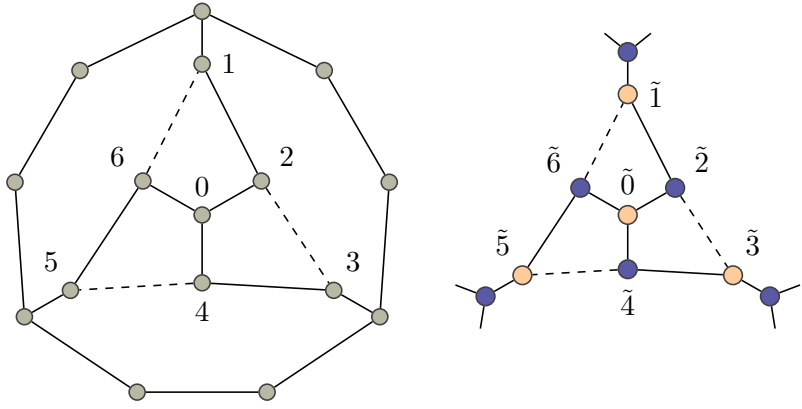
\begin{figure}[H]
	\centering
	\tikzset{
	v/.style={circle,draw=black!75,inner sep=0pt,minimum size=6pt,fill=yellow!20!gray!80},
	A/.style={circle,draw=black!75,inner sep=0pt,minimum size=7pt,fill=blue!30!gray},
	B/.style={circle,draw=black!75,inner sep=0pt,minimum size=7pt,fill=orange!40!white},
}
\begin{tikzpicture}[scale=1,semithick]
\begin{scope}
	\def\rm{0.9}
	\def\ri{2}
	\def\ro{2.7}
	\def\re{2.5}
	\node[v,label=  90:0] (v0) at (   0:0) {};
	\node[v,label=  30:2] (v2) at (  30:\rm) {};
	\node[v,label= -90:4] (v4) at ( -90:\rm) {};
	\node[v,label= 150:6] (v6) at ( 150:\rm) {};
	\draw (v0)--(v2) (v0)--(v4) (v0)--(v6);
	\node[v,label=   0:1] (v1) at (  90:\ri) {};
	\node[v,label=  80:3] (v3) at ( -30:\ri) {};
	\node[v,label= 100:5] (v5) at ( 210:\ri) {};	
	\draw (v1)--(v2) (v3)--(v4) (v5)--(v6);
	\draw[dashed] (v2)--(v3) (v4)--(v5) (v6)--(v1);
	\node[v] (p1) at (  90:\ro) {};
	\node[v] (p3) at ( -30:\ro) {};
	\node[v] (p5) at ( 210:\ro) {};
	\draw (v1)--(p1) (v3)--(p3) (v5)--(p5);
	\node[v] (p13) at (  50:\re) {};
	\node[v] (p31) at (  10:\re) {};
	\node[v] (p35) at ( -70:\re) {};
	\node[v] (p53) at (-110:\re) {};
	\node[v] (p51) at ( 170:\re) {};
	\node[v] (p15) at ( 130:\re) {};
	\draw (p1)--(p13)--(p31)--(p3)--(p35)--(p53)--(p5)--(p51)--(p15)--(p1);
\end{scope}
\begin{scope}[scale=0.8,shift={(7,0)}]
	\def\rm{0.9}
	\def\ri{2}
	\def\ro{2.7}
	\def\re{3.2}
	\node[B,label=  90:$\tilde{0}$] (v0) at (   0:0) {};
	\node[A,label=  30:$\tilde{2}$] (v2) at (  30:\rm) {};
	\node[A,label= -90:$\tilde{4}$] (v4) at ( -90:\rm) {};
	\node[A,label= 150:$\tilde{6}$] (v6) at ( 150:\rm) {};
	\draw (v0)--(v2) (v0)--(v4) (v0)--(v6);
	\node[B,label=   0:$\tilde{1}$] (v1) at (  90:\ri) {};
	\node[B,label=  80:$\tilde{3}$] (v3) at ( -30:\ri) {};
	\node[B,label= 100:$\tilde{5}$] (v5) at ( 210:\ri) {};	
	\draw (v1)--(v2) (v3)--(v4) (v5)--(v6);
	\draw[dashed] (v2)--(v3) (v4)--(v5) (v6)--(v1);
	\node[A] (p1) at (  90:\ro) {};
	\node[A] (p3) at ( -30:\ro) {};
	\node[A] (p5) at ( 210:\ro) {};
	\draw (v1)--(p1) (v3)--(p3) (v5)--(p5);
	\node[] (p13) at (  80:\re) {};
	\node[] (p31) at ( -20:\re) {};
	\node[] (p35) at ( -40:\re) {};
	\node[] (p53) at ( 200:\re) {};
	\node[] (p51) at ( 220:\re) {};
	\node[] (p15) at ( 100:\re) {};
	\draw (p1)--(p13) (p31)--(p3)--(p35) (p53)--(p5)--(p51) (p15)--(p1);
\end{scope}
\end{tikzpicture}
	\caption{A graph $K_A$ with no simple disambiguation of medians. Right: a part of its universal cover $\Ueq{K_A}$}  
	\label{ambig}
\end{figure}

Now consider the following vertices of $K_A^{C_9}$:
	$$x = (0,0,0,0,0,0,0,0,0)\text{ and }y =(6,2,6,4,6,4,2,4,2).$$
They form an edge of $K_A^{C_9}$ which corresponds to an 18-cycle that reduces to $\eps$ (since every other vertex is 0) and hence factors through $\Ueq{K_A}$.
Thus they lift to an edge $\tilde{x}\tilde{y}$ of $\Ueq{K_A}^{C_9}$, or $\Ueq{K_A}^{9}$ if we shift one of the tuples cyclically (say $\tilde{x}$, which stays unchanged).
Each vertex of each tuple is contained within the lift $\tilde{1},\dots,\tilde{6}$ of the first and second neighbourhood of some lift $\tilde{0}$ of $0$.

Their medians in $\Uu(L_A)$ are $\mu(\tilde{x}) = \mu(\tilde{y}) = \tilde{0}$ and $\tilde{x}$ is correctly coloured.
However, if we set $\psi(\tilde{x}) = \mu(\tilde{x})$ for correctly coloured medians as in the proof of Lemma~\ref{sductok},
there is no way to disambiguate the location of $\psi(\tilde{y})$.
Indeed $\tilde{y}=(\tilde{6},\tilde{2},\tilde{6},\tilde{4},\tilde{6},\tilde{4},\tilde{2},\tilde{4},\tilde{2})$ has in $\Ueq{K_A}^{C_9}$ neighbours $\tilde{u}, \tilde{v}, \tilde{w}$ with
$$\begin{array}{rl}
\tilde{u} & \mbox{$= (\tilde{1},\tilde{1},\tilde{3},\tilde{1},\tilde{3},\tilde{1},\tilde{3},\tilde{1},\tilde{5})$, so that $\mu(\tilde{u}) = \tilde{1}$,} \\
\tilde{v} & \mbox{$= (\tilde{3},\tilde{5},\tilde{3},\tilde{5},\tilde{3},\tilde{1},\tilde{3},\tilde{3},\tilde{5})$, so that $\mu(\tilde{v}) = \tilde{3}$,} \\
\tilde{w} & \mbox{$= (\tilde{1},\tilde{5},\tilde{3},\tilde{5},\tilde{5},\tilde{1},\tilde{5},\tilde{1},\tilde{5})$, so that $\mu(\tilde{w}) = \tilde{5}$.}
\end{array}$$
(In particular the shift $(\tilde{2},\tilde{6},\tilde{4},\tilde{6},\tilde{4},\tilde{2},\tilde{4},\tilde{2},\tilde{6})$ of $\tilde{y}$ is adjacent to all three tuples in $\Ueq{K}^9$).
In the proof of Lemma~\ref{sductok}, we set $\psi(\tilde{u}) = \tilde{1}, \psi(\tilde{v}) = \tilde{3}$ and $\psi(\tilde{w}) = \tilde{5}$.
No value of $\psi(\tilde{y})$ can be adjacent to these three values.

\begin{figure}[ht!]
	\centering
	\tikzset{
	v/.style={circle,draw=black!75,inner sep=0pt,minimum size=12pt,fill=blue!50!gray!10!white},
	u/.style={circle,draw=black!75,inner sep=0pt,minimum size=12pt,fill=blue!70!red!50!black,text=white},
}
\centering
\begin{tikzpicture}
\begin{scope}[shift={(11,0)},rotate=30,yscale=-1] 
	\node[v] (v1) at (   0:2) {1};
	\node[v] (v2) at (  72:2) {2};
	\node[v] (v3) at ( 144:2) {3};
	\node[v] (v4) at (-144:2) {4};
	\node[v] (v5) at ( -72:2) {5};
	\node[v] (va) at (-36+0*72:1) {a};
	\node[v] (vb) at (-36+2*72:1) {b};
	\node[v] (vc) at (-36+4*72:1) {c};
	\node[v] (vd) at (-36+6*72:1) {d};
	\node[v] (ve) at (-36+8*72:1) {e};
	\draw[semithick,red!50!black,densely dashed] (v1)--(v2)--(v3)--(v4)--(v5)--(v1);
	\draw[thick] (va)--(vb)--(vc)--(vd)--(ve)--(va);
	\draw                 (v1)--(vd)  (v2)--(vb)  (v3)--(ve)  (v4)--(vc)  (v5)--(va);
	\draw[densely dashed] (v1)  (vd)--(v2)  (vb)--(v3)  (ve)--(v4)  (vc)--(v5)  (va)--(v1);
\end{scope}
\begin{scope}[shift={(0,0)},rotate=90,yscale=-1] 
	\node[v] (va) at (0*72:2) {a};
	\node[v] (vb) at (1*72:2) {b};
	\node[v] (vc) at (2*72:2) {c};
	\node[v] (vd) at (3*72:2) {d};
	\node[v] (ve) at (4*72:2) {e};
	\draw[thick] (va)--(vb)--(vc)--(vd)--(ve)--(va);	
	\node[v] (v1) at (-72:1) {1};
	\node[v] (v2) at (-3*72:1) {2};
	\node[v] (v3) at (-5*72:1) {3};
	\node[v] (v4) at (-7*72:1) {4};
	\node[v] (v5) at (-9*72:1) {5};
	\draw[semithick,red!50!black,densely dashed] (v1)--(v2)--(v3)--(v4)--(v5)--(v1);
	\draw                 (v1)--(vd)  (v2)--(vb)  (v3)--(ve)  (v4)--(vc)  (v5)--(va);
	\draw[densely dashed] (v1)  (vd)--(v2)  (vb)--(v3)  (ve)--(v4)  (vc)--(v5)  (va)--(v1);
\end{scope}
\begin{scope}[shift={(5.55,-2.1)},rotate=-90,scale=0.53,yscale=-1]
	\def\r{2.5}
	\node[v] (va) at (0*36:\r) {a};
	\node[u] (ub) at (1*36:\r) {b};
	\node[v] (vc) at (2*36:\r) {c};
	\node[u] (ud) at (3*36:\r) {d};
	\node[v] (ve) at (4*36:\r) {e};
	\node[u] (ua) at (5*36:\r) {a};
	\node[v] (vb) at (6*36:\r) {b};
	\node[u] (uc) at (7*36:\r) {c};
	\node[v] (vd) at (8*36:\r) {d};
	\node[u] (ue) at (9*36:\r) {e};
	\draw[ultra thick] (va)--(ub)--(vc)--(ud)--(ve)--(ua)--(vb)--(uc)--(vd)--(ue)--(va);	
	\node[v] (v1) at (4*36:6) {1};
	\node[v] (v2) at (2*36:6) {2};
	\node[v] (v3) at (0*36:6) {3};
	\node[v] (v4) at (8*36:6) {4};
	\node[v] (v5) at (6*36:6) {5};
	\node[u] (u1) at (180+4*36:6) {1};
	\node[u] (u2) at (180+2*36:6) {2};
	\node[u] (u3) at (180+0*36:6) {3};
	\node[u] (u4) at (180+8*36:6) {4};
	\node[u] (u5) at (180+6*36:6) {5};
	\begin{scope}[on background layer]
	\draw[red!50!black!70] plot[smooth cycle] coordinates {
		(v1) (u2) (v3) (u4) (v5) (u1) (v2) (u3) (v4) (u5) 
	};
	\draw[red!50!black,dashed] plot[smooth cycle] coordinates {
		(v1) (u2) (v3) (u4) (v5) (u1) (v2) (u3) (v4) (u5) 
	};
	\end{scope}
	\draw (v1)--(ud)--(v2)--(ub)--(v3)--(ue)--(v4)--(uc)--(v5)--(ua)--(v1);
	\draw (u1)--(vd)--(u2)--(vb)--(u3)--(ve)--(u4)--(vc)--(u5)--(va)--(u1);
\end{scope}
\end{tikzpicture}
	\caption{A graph $K_B$ that is square-dismantlable to a cycle,
	but is a core with $\Ueq{K_B}_{/R}\simeq K_B$.
	Left, right: two different drawings of $K_B$.
	Middle: a drawing of $K_B \times K_2$,
	which makes the structure of squares more apparent:
	every red edge is contained in a unique square. 
	The squares form a band that winds thrice around the thick (inner) cycle,
	with one border of the band coloured red
	and the other border identified three-fold into the thick cycle.}
	\label{graphb}
\end{figure}
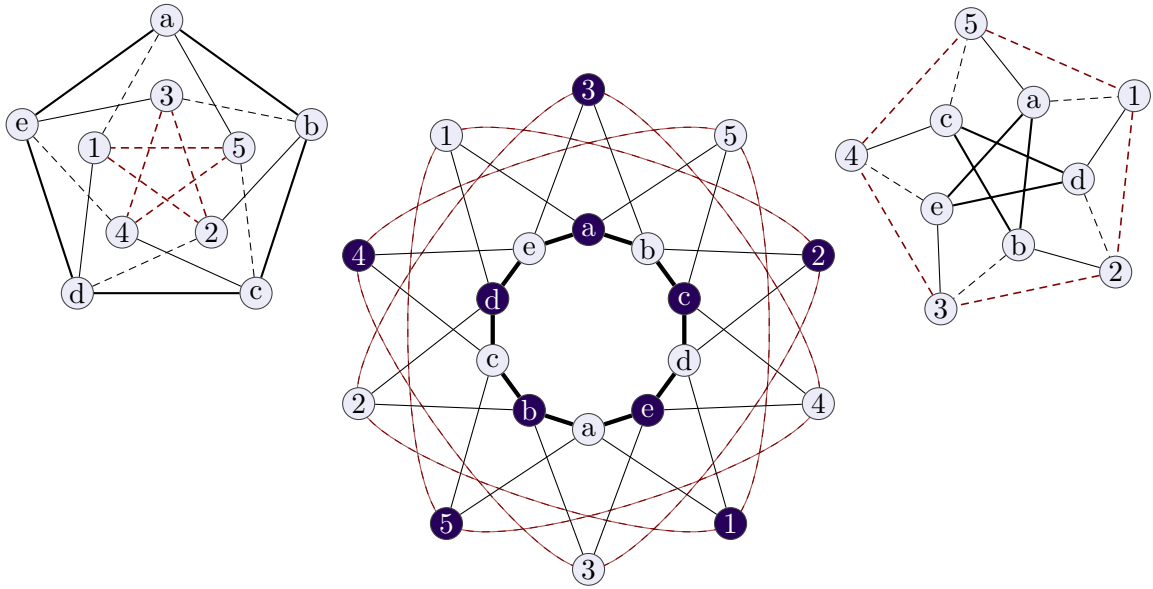

\begin{figure}[ht!]
	\centering
	\tikzset{
	v/.style={circle,draw=black!75,inner sep=0pt,minimum size=12pt,fill=blue!50!gray!10!white},
	u/.style={circle,draw=black!75,inner sep=0pt,minimum size=12pt,fill=blue!70!red!50!black,text=white},
}
\begin{tikzpicture}
\begin{scope}[xscale=0.8,yscale=0.8]
	\node[]  (em) at (-0.7,0) {};
	\node[u] (a)  at ( 0,0) {$a$};
	\node[v] (b)  at ( 1,0) {$b$};
	\node[u] (c)  at ( 2,0) {$c$};
	\node[v] (d)  at ( 3,0) {$d$};
	\node[u] (e)  at ( 4,0) {$e$};
	\node[v] (ap) at ( 5,0) {$a'$};
	\node[u] (bp) at ( 6,0) {$b'$};
	\node[v] (cp) at ( 7,0) {$c'$};
	\node[u] (dp) at ( 8,0) {$d'$};
	\node[v] (ep) at ( 9,0) {$e'$};
	\node[u] (app) at (10,0) {$a''$};
	\node[]  (bpp) at (11,0) {};

	\node[] (v3m)  at (0,1.5) {};
	
	\node[v] (v5)  at ( 1,1.5) {$5$};
	\node[u] (v1)  at ( 4,1.5) {$1$};
	\node[v] (v2p) at ( 7,1.5) {$2'$};
	\node[] (v3p) at (10,1.5) {};
	
	\node[u] (v2)  at ( 2,1.5) {$2$};
	\node[v] (v3)  at ( 5,1.5) {$3$};
	\node[u] (v4p) at ( 8,1.5) {$4'$};
	
	\node[v] (v4)  at ( 3,1.5) {$4$};
	\node[u] (v5p) at ( 6,1.5) {$5'$};
	\node[v] (v1p) at ( 9,1.5) {$1'$};

	\draw[very thick] (a)--(b)--(c)--(d)--(e)--(ap)--(bp)--(cp)--(dp)--(ep)--(app);
	\draw[dotted] (app)--(bpp);
	\draw[red!50!black,densely dashed] (v5) to[bend left] (v1) to[bend left] (v2p);
	\draw[red!50!black,densely dashed] (v2) to[bend left] (v3) to[bend left] (v4p);
	\draw[red!50!black,densely dashed] (v4) to[bend left] (v5p) to[bend left] (v1p);
	\draw (a)--(v5)--(c) (d)--(v1)--(ap) (bp)--(v2p)--(dp);
	\draw (c)--(v4)--(e) (ap)--(v5p)--(cp) (dp)--(v1p)--(app);
	\draw (b)--(v2)--(d) (e)--(v3)--(bp) (cp)--(v4p)--(ep);
	\draw[dotted] (a)--(em) (b)--(v3m.center) (ep)--(v3p.center);
	\draw[red!50!black,dotted] (v3m) to[bend left] (v4);
	\clip (-0.3,0) rectangle (11,2);
	\draw[red!50!black,dotted,path fading=west] (v2) to[bend right] (-1,1.5);
	\draw[red!50!black!50,dotted,path fading=west] (v5) to[out=150,in=0] (-0.7,2);
	\draw[red!50!black,dotted,path fading=east] (v2p) to[bend left] (v3p);
	\draw[red!50!black,dotted,path fading=east] (v4p) to[out=30,in=180] (10,2);
\end{scope}
\begin{scope}[shift={(5.,6.5)},xscale=0.85,yscale=0.75]
	\node (lt) at (9,-7) {};
	\node (lss) at (-1,-1) {};
	\node (ls) at  ($(lss)!0.26!(lt)$) {};
	\node (l1t) at ($(lt)+(  90:3.3)$) {};
	\node (l2t) at ($(lt)+(-120:3.0)$) {};
	\node (l3t) at ($(lt)+(  10:2.0)$) {};
	\node (l1s) at ($(lss)!0.29!(l1t)$) {};
	\node (l2s) at ($(lss)!0.26!(l2t)$) {};
	\node (l3s) at ($(lss)!0.22!(l3t)$) {};	
	\node (l1ss) at ($(lss)!0.18!(l1t)$) {};
	\node (l2ss) at ($(lss)!0.18!(l2t)$) {};
	\node (l3ss) at ($(lss)!0.18!(l3t)$) {};	
	\draw[very thick] ($(lss)!0.18!(lt)$)--(lt);
	\draw[red!50!black,densely dashed] (l1ss)--(l1t);
	\draw[red!50!black,densely dashed] (l2ss)--(l2t);
	\draw[red!50!black,densely dashed] (l3ss)--(l3t);
	\node[u] (a)  at ($(ls)!0.2/9!(lt)$) {$a$};
	\node[v] (b)  at ($(ls)!0.9/9!(lt)$) {$b$};
	\node[u] (c)  at ($(ls)!1.7/9!(lt)$) {$c$};
	\node[v] (d)  at ($(ls)!2.6/9!(lt)$) {$d$};
	\node[u] (e)  at ($(ls)!3.7/9!(lt)$) {$e$};
	\node[v] (ap) at ($(ls)!4.9/9!(lt)$) {$a'$};
	\node[u] (bp) at ($(ls)!6.1/9!(lt)$) {$b'$};
	\node[v] (cp) at ($(ls)!7.3/9!(lt)$) {$c'$};
	\node[u] (dp) at ($(ls)!8.5/9!(lt)$) {$d'$};
	
	\node[v] (v5)  at ($(l1s)!0.9/9!(l1t)$) {$5$};
	\node[u] (v1)  at ($(l1s)!4./9!(l1t)$) {$1$};
	\node[v] (v2p) at ($(l1s)!7.6/9!(l1t)$) {$2'$};
	
	\node[u] (v2)  at ($(l2s)!1.7/9!(l2t)$) {$2$};
	\node[v] (v3)  at ($(l2s)!4.5/9!(l2t)$) {$3$};
	\node[u] (v4p) at ($(l2s)!8./9!(l2t)$) {$4'$};
	
	\node[v] (v4)  at ($(l3s)!2.9/9!(l3t)$) {$4$};
	\node[u] (v5p) at ($(l3s)!5.7/9!(l3t)$) {$5'$};
	\node[v] (v1p) at ($(l3s)!8.4/9!(l3t)$) {$1'$};	
%
%
	\draw (a)--(v5)--(c) (d)--(v1)--(ap) (bp)--(v2p)--(dp)--(v1p);
	\draw (c)--(v4)--(e) (ap)--(v5p)--(cp);
	\draw (b)--(v2)--(d) (e)--(v3)--(bp) (cp)--(v4p);
\end{scope}
\end{tikzpicture}
	\caption{The universal cover of $K_B$, drawn in two ways.}
\end{figure}
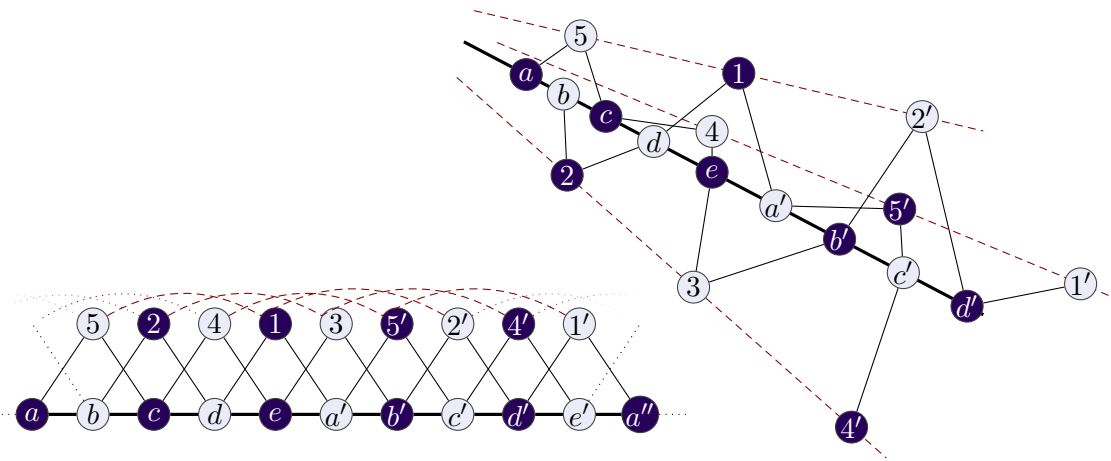

Now consider graph $K_B$ in Figure~\ref{graphb},
which square-dismantles to the 5-cycle plus dangling edges (the solid edges) $L_B$.
For the reduced walk $R$ going once around the thick 5-cycle,
we have $K_B \simeq \Uu(K_B)_{/R}$,
hence the use of Theorem~\ref{thm:mainTopo} amounts to showing
that $K_B$ is strongly multiplicative
by using the fact that $K_B$ is strongly multiplicative.
We also cannot use known results for cycles or circular cliques,
since $K_B$ is a different core graph.
In fact, it does not even admit a homomorphism to a circular clique $K_r$ with $r\in [2,4)$;
it can be shown that the circular chromatic number of $K_B$ is exactly~4
(Matsushita~\cite{Matsushita13} gave a general method that can prove this; using his language, the stable length of the generator of $\pi(K_B)_{/\eq}$ is $\frac{5}{3}$, while for the circular clique $K_{p/q}$ with $\frac{p}{q}<4$, it is $\frac{p}{p-2q}>2>\frac{5}{3}$, which makes a homomorphism impossible).

In summary, obstacles to proving strong multiplicativity appear already for square-dismantlable graphs,
and even subcubic core graphs.
In a follow-up paper~\cite{followup}, we show a method for proving that some graphs, including the above example $K_A$, are not strongly multiplicative.
They do not even satisfy the `constant' property of Lemma~\ref{lem:constProperty}.
On the other hand, we will see next that the use of adjoint functors
allows to prove the strong multiplicativity of some graphs that are not square-dismantlable.

\clearpage
\section{Adjoint functors}\label{sec:adjoint}
The use of adjoint functors in proofs of multiplicativity can be 
summarised as follows: The existence of a homomorphism $\phi: G \times H \rightarrow K$
is used to infer the existence of a homomorphism $\phi': G' \times H' \rightarrow K'$,
where $K'$ is known to be multiplicative. This implies that one of the latter factors,
say $G'$ admits a homomorphism to $K'$. This is used to conclude that the corresponding
former factor $G$ admits a homomorphism to $K$. Of course, it takes very special 
correspondences between graphs and their primed counterparts to make all parts of the
argument work. This is the case for correspondences given by the adjoint functors
described here (see also~\cite{FoniokT18} for more on such functors in graph theory). 

For a graph $G$, let $P_3(G)$ be the `third power' of $G$, defined by
\begin{eqnarray*}
V(P_3(G)) & = & V(G), \\
E(P_3(G)) & = & \{ \{u,v\} \mid \mbox{$u$ and $v$ are endpoints of a walk of length $3$ in $G$}\}.
\end{eqnarray*}
In other words, the adjacency matrix is taken to the third power. Note that vertices at distance 2 in $G$ do not necessarily become adjacent in $P_3(G)$.
We denote $N_G(u)$ the neighbourhood of $u$ in $G$. For two subsets
$A$, $B$ of $V(G)$, we write $A \Join B$ ($A$ is \emph{joined} to $B$)
if every vertex of $A$ is adjacent to every vertex of $B$.
The graph $P_3^{-1}(G)$ is defined as follows.
 \begin{eqnarray*}
V(P_3^{-1}(G)) & = & \{ (u,A) \mid u \in V(G), A \subseteq N_G(u) \mbox{ non-empty}\},\\
E(P_3^{-1}(G)) & = & \{ \{(u,A),(v,B)\} \mid u \in B \Join A \ni v \}.
\end{eqnarray*}

These two graph operations are left and right adjoints, respectively, meaning that
$P_3(G) \to H$ holds if and only if $G \to P_3^{-1}(H)$.
In particular this implies that for any graph, $P_3(P_3^{-1}(G)) \to G \to P_3^{-1}(P_3(G))$.
Moreover, $P_3(P_3^{-1}(G))$ is always homomorphically equivalent to $G$, though this is not always true for $P_3^{-1}(P_3(G))$ (see e.g.~\cite{Wrochna17b}, Lemma 2.3 for proofs). 
The first author used this to infer the multiplicativity of circular cliques $K_{p/q}$ with $\frac{p}{q} < 4$ from that of odd cycles~\cite{Tardif05}.
Here, we infer the multiplicativity of powers of high-girth graphs from that of square-free graphs and stress that again, the method works for strong multiplicativity as well.

\begin{theorem}\label{thm:p3}
For a graph $K$, $K$ is strongly multiplicative if and only if $P_3^{-1}(K)$
is strongly multiplicative.
\end{theorem}
\begin{proof}
Suppose that $K$ is strongly multiplicative.
Let $\phi \colon (G \times C') \cup (C \times H) \rightarrow P_3^{-1}(K)$ be a homomorphism,
where $G, H$ are connected graphs, and $C, C'$ odd cycles contained in $G$ and $H$ respectively.
Then there is a homomorphism $\psi \colon P_3\left( (G \times C') \cup (C \times H)\right) \rightarrow K$.
Now $P_3\left( (G \times C') \cup (C \times H)\right)$ contains 
$P_3(G \times C') \cup P_3(C \times H)$ as a subgraph and $P_3\left( (G \times C')\right)$ is equal to $P_3(G) \times P_3(C')$, which contains $P_3(G) \times C'$ (similarly for $C \times H$).
Therefore $\psi$ is also a homomorphism from $(P_3(G) \times C') \cup (C \times P_3(H))$ to $K$.
Since $K$ is strongly multiplicative, this implies that $P_3(G)$ or $P_3(H)$ admits a homomorphism
to $K$, whence $G$ or $H$ admits a homomorphism to $P_3^{-1}(K)$.

Now suppose that $P_3^{-1}(K)$ is strongly multiplicative. 
Let $\phi: (G \times C') \cup (C \times H) \rightarrow K$ be a homomorphism,
where $G, H$ are connected graphs, and $C, C'$ odd cycles contained in $G$ and $H$ respectively.
Then $P_3^{-1}(C), P_3^{-1}(C')$ are odd cycles, and
$(P_3^{-1}(G) \times P_3^{-1}(C')) \cup (P_3^{-1}(C) \times P_3^{-1}(H))$
admits a natural homomorphism to $P_3^{-1}((G \times C') \cup (C \times H))$,
namely $((g,A_g),(h,A_h)) \mapsto ((g,h),A_g \times A_h)$.
This, composed with $P_3^{-1}(\phi)$ yields a homomorphism $\psi$ from
$(P_3^{-1}(G) \times P_3^{-1}(C')) \cup (P_3^{-1}(C) \times P_3^{-1}(H))$ to $P_3^{-1}(K)$.
Since $P_3^{-1}(K)$ is strongly multiplicative, this means that $P_3^{-1}(G)$ or $P_3^{-1}(H)$
admits a homomorphism to $P_3^{-1}(K)$, whence $P_3(P_3^{-1}(G))$ or $P_3(P_3^{-1}(H))$
admits a homomorphism to $K$. 
Now $G$ admits a homomorphism to $P_3(P_3^{-1}(G))$ (namely $g \mapsto (g,N_G(g))$, an edge $\{g_1,g_2\}$ corresponding to the length-3 walk $(g_1,N_G(g_1)), (g_2,g_1), (g_1,g_2), (g_2,N_G(g_2))$), similarly for $H$. Therefore $G$ or $H$ admits a homomorphism to $K$.
\end{proof}

For graphs $K$ that are known to be strongly multiplicative, the proof methods often already generalize to $P_3^{-1}(K)$ without using Theorem~\ref{thm:p3}.
For example if $K$ is square-free, then $P_3^{-1}(K)$ is just its subdivision, replacing each edge by a path of length 3 (see~\cite{Wrochna17b}, Lemma 2.3) and hence also square-free.
On the other hand, the generalization of Theorem~\ref{gwospvam} from graphs $K$ with each edge in at most one square to graphs $P_3^{-1}(K)$ is not obvious.
It would be interesting to find a natural class of graphs that contains such $P_3^{-1}(K)$ and for which methods similar to Theorem~\ref{gwospvam} can be applied.

In some cases, however, we can also infer the multiplicativity of $P_3(K)$ from that of $K$.
This is the case if $P_3^{-1}(P_3(K))$ is homomorphically equivalent to $K$, by Theorem~\ref{thm:p3}.
Here we show this for graphs of high girth: graphs with no cycles of length 12 or less.
The number 12 comes from the fact that $P_3$ applied to a graph with such a cycle will `patch' the cycle with squares (making it $\eq$-equivalent to $\eps$, in particular).

\begin{theorem}\label{thm:p3equiv}
Let $K$ be a graph with girth at least thirteen. Then $P_3^{-1}(P_3(K))$
is homomorphically equivalent to $K$.
\end{theorem}
\begin{proof}
\def\CN{\ensuremath{\mathrm{CN}}}
The map $\phi: K \rightarrow P_3^{-1}(P_3(K))$
defined by $\phi(u) = (u,N_K(u))$ is a homomorphism,
so it suffices to define a homomorphism in the other direction.

For $A \subseteq V(K)$, define the \emph{common neighbourhood} of $A$ as $\CN(A) := \bigcap_{a \in A} N_K(a)$.
Throughout the proof all neighbourhoods and walks are always meant in $K$, not in $P_{3}(K)$ (only $A \Join B$ will mean joined sets in $P_{3}(K)$).

Let $(u,A)$ be a vertex of $P_3^{-1}(P_3(K))$.
Since $K$ has girth at least thirteen, it is square-free, hence $|\CN(A)| \leq 1$ if $|A| \geq 2$.
If on the other hand $|A|=1$, then $\CN(A)$ has a unique element closest to $u$ (otherwise $K$ would have a cycle of length $3$, $4$ or $6$).
We can therefore define $\psi: P_3^{-1}(P_3(K)) \rightarrow K$ by
$$
\psi(u,A) =
\left \{ \begin{array}{l}
\mbox{ $u$ if  $\CN(A) = \emptyset$,} \\
\mbox{ the only element of $\CN(A)$ that is closest to $u$ otherwise.}
\end{array} \right.
$$
We claim that $\psi$ is a homomorphism.

Consider two adjacent vertices $(x,A), (y,B)$ in $P_{3}^{-1}(P_{3}(K))$.
That is, $x \in B$, $y \in A$ and $A \Join B$ in $P_{3}(K)$.
We first show that\vspace*{-7pt}
\begin{equation}
	\CN(A) \neq \emptyset\text{ or }\CN(B) \neq \emptyset\tag{*}\label{eq:power}\vspace*{-6pt}
\end{equation}
Suppose that to the contrary $\CN(A) = \CN(B) = \emptyset$.
Then there are $a,a' \in A$ with no common neighbour in $K$ (since otherwise there would be distinct $a,a',a'' \in A$ with pairwise common neighbours, contradicting that the girth is $>6$) and similarly $b,b' \in B$ with no common neighbour in $K$. Since $A \Join B$ in $P_{3}(K)$, there are walks of length 3 between $a,a'$ and $b,b'$ in $K$.
Concatenated together, these four walks give one closed walk of length 12 going through $a,a',b,b'$ in order.
Since $K$ has girth $>12$, this closed walk must be a walk in a tree subgraph $T$ of $K$.
Hence the walks of length 3 between $a,a'$ and $b,b'$ are in $T \subseteq K$.

Let $P$ be the shortest path between $a$ and $a'$ in $T$: since there is a walk of length exactly 6 between them (going though $b$ or $b'$), the shortest path has even length $\leq 6$ and $>2$, since $a,a'$ have no common neighbour.
Hence $P$ has length 4 or 6.
If $P$ has length 6, then $b$ and $b'$, which are accessible via walks of length 3 from both endpoints of $P$ in $T$, must both be equal to middle vertex of $P$, a contradiction.
If $P$ has length 4, then similarly $b$ and $b'$ must be adjacent to the middle vertex of $P$ in $T$, hence they have a common neighbour in $T \subseteq K$, a  contradiction.
This proves~\eqref{eq:power}.
\medskip

If $\CN(A) \neq \emptyset$ and $\psi(y,B) = y$, then $y \in A$ is adjacent by definition to any vertex in $\CN(A)$, so $\psi(y,B)$ is adjacent to $\psi(x,A) \in \CN(A)$.
In particular if $\CN(B) = \emptyset$, then $\psi(y,B) = y$ and $\CN(A) \neq \emptyset$ by~\eqref{eq:power}, hence the mapping is correct.
Symmetrically if $\CN(B) \neq \emptyset$ and $\psi(x,A) = x$, then the mapping is correct.
In particular if $\CN(A) = \emptyset$, then we are done.

\medskip

It thus remains to consider the case where $\psi(x,A) \neq x$ and $\psi(y,B) \neq y$.
In this case, we have that $\CN(A) \neq \emptyset$ but $x \not\in \CN(A)$ and similarly for $y,B$.
Let $a$ be a vertex in $A$ not adjacent to $x$ in $K$.
Since $x \in B \Join A \ni a$ in $P_{3}(K)$, there is a (unique, since $K$ has girth $>6$)  walk of length 3 between $a$ and $x$ in $K$; let $x,p,q,a$ be the vertices on it.
If there is any vertex $a' \in A$ other than $a$, there is a walk of length 2 between $a$ and $a'$ (as they have a common neighbour), a~walk of length 3 from $a'$ to $x$, and the walk of length 3 $x,p,q,a$; together they form a closed walk of length 8, so they must map to a tree in $K$ and it is easy to see that the common neighbour of $a$ and $a'$ must be $q$.
Hence $q \in \CN(A)$.
If $|A|>1$ there can be only one vertex in $\CN(A)$ (since $K$ has no $C_4$), while
if $A = \{a\}$, then the vertex in $\CN(A)=N(a)$ closest to $x$ is $q$.
Hence $\psi(\{x\},A)=q$.\looseness=-1

Since $\psi(y,B) \neq y$, we have $\CN(B) \neq \emptyset$ but $y \not\in \CN(B)$.
Then $y \in A$ is adjacent to $q \in \CN(A)$.
If $y$ is not adjacent to $x$, then $y,q,p,x$ is a (unique) walk of length 3 to $x \in B$ and just as above we conclude that $\psi(y,B)$ is the second vertex on this walk, namely $p$; thus $\psi(x,A)=q$ and $\psi(y,B)$ are adjacent.
Otherwise, if $y$ is adjacent to $x$, then it is a common neighbour of $x$ and $q$, hence $y=p$ (since $K$ has no $C_4$).
For any vertex $b \in B$ other than $x$, similarly as before there is a walk of length $2$ between $x$ and $b$ (since $\CN(B) \neq \emptyset$), a walk of length $3$ between $b$ and $a$, an the walk $a,q,p,x$;
together, they form a closed walk of length 8, so they must map to a tree in $K$ and it is easy to see that the common neighbour of $x$ and $b$ must be $p=y$.
Hence $y \in \CN(B)$, which we assumed was not the case.
Therefore in all cases $\psi(x,A)$ is adjacent to $\psi(y,B)$ in~$K$, concluding the proof.
\end{proof}

\begin{corollary}
If $K$ has girth at least thirteen, then $P_3(K)$ is strongly multiplicative.
\end{corollary}

One can define higher powers $P_k(G)$ for odd $k$, which also admit right adjoints $P_k^{-1}$ with similar properties, as first proved by Hajiabolhassan and Taherkhani~\cite{HajiabolhassanT10}, see also~\cite{Hajiabolhassan09}.
We conjecture that $P_k^{-1}(P_k(G))$ is homomorphically equivalent to $G$ for graphs of girth $>4k$.
This would imply that $k$-th powers of graphs of girth $>4k$ are strongly multiplicative.
It would be also interesting to find more graphs which satisfy $P_3^{-1}(P_3(G)) \leftrightarrow G$, hopefully generalizing both Theorem~\ref{thm:p3equiv} and the case of circular cliques $K_{p/q}$ with $\frac{p}{q} < \frac{12}{5}$, as proved in~\cite{Tardif05}.

\pagebreak
\appendix

\section{Proof of Theorem~\ref{thm:mainTopo} and Lemma~\ref{lem:constConverse}}\label{sec:mainTopo}
Let us first prove Lemma~\ref{lem:constConverse}, the partial converse to Theorem~\ref{thm:mainTopo} which we restate here.
\begin{lemma}
Let $K$ be a strongly multiplicative graph.
Then $K^{C_{2k+1}}_\eps$ admits a homomorphism to $K$ (for each $k\in\NN$).
\end{lemma}
\begin{proof}
Suppose it does not, for some $k$.
Let $H=K^{C_{2k+1}}_\eps$.
So $H \not\to K$, but we have the natural `evaluation' homomorphism $e:C_{2k+1} \times H \to K$, defined as $e(c_i,h) := h(c_i)$.
Consider the constant maps $f_v \in V(H)$ defined as $f_v(\cdot)=v$ for each $v \in V(K)$.
Then $e$ has the property that $e(c_i,f_v)=v$ for each $c_i \in V(C_{2k+1}), v\in V(K)$.

Let $G$ be the graph obtained from $C_{2k+1}$ by adding a vertex $z$ adjacent to one vertex of the cycle, with a loop on $z$.
Then clearly $G \not\to K$ (if we don't want loops we can use instead a clique of size $>\chi(K)$).
Let $C$ be the original cycle in $G$, let $d_1,d_2,\dots$ be an odd cycle in $K$ and let $D$ be the odd cycle on the corresponding constants $f_{d_1},f_{d_2},\dots$ in $H$.
Let $p: G \times D \to K$ be the projection to $D$ composed with the natural map to $K$, namely $p(g,f_{d_j}) := d_j$.
Then the homomorphisms $e: C\times H \to K$ and $p: G \times D \to K$ agree on $C \times D$, since
$ e(c_i,f_{d_j})=d_j = p(c_i,f_{d_j})$.
This contradicts strong multiplicativity.
\end{proof}

In the remainder of this section we recall some proofs from~\cite{Wrochna17}, add a few details and rephrase the conclusions in terms of universal covers to obtain Theorem~\ref{thm:mainTopo}.

More precisely, we shall prove the following strengthening.
Recall that $\phi: G\to K$ \emph{factors through} $H$ if $\phi=\phi_1 \circ \phi_2$ for some $\phi_1: G \to H$ and $\phi_2: H \to K$.
For a homomorphism $\phi: G \times H \to K$, we denote by $\phi^*: G \to K^H$ its currying: $\phi^*(g) = (h \mapsto \phi(g,h))$.
For $\phi \colon (G \times D) \cup (C \times H) \rightarrow K$ we define $\phi^*: G \to K^{D}$ and $\phi^*: H \to K^{C}$ the same way.

\begin{theorem}\label{thm:mainFactor}
	Let $K$ be a graph such that $\pi(K,r)_{/\eq}$ is a free group.
	Let $\phi \colon (G \times D) \cup (C \times H) \rightarrow K$ be a homomorphism, for some odd-length cycles $C,D$ in connected graphs $G,H$.
	Then one of the following holds:
	\begin{itemize}
		\item $\phi^* : H \to K^C$ factors through $\tau({\Ueq{K}^C})$.
		\item $\phi^* : G \to K^{D}$ factors through $\tau(\Ueq{K}^{D})$.
		\item $\phi$ factors through the $\eq$-unicyclic cover $\Ueq{K}_{/R}$ for some closed walk $R$ in $K$.
	\end{itemize}
\end{theorem}

Theorem~\ref{thm:mainFactor} implies that to show that a graph is strongly multiplicative, it suffices to show that its unicyclic covers are, and that the subgraph $\tau({\Ueq{K}^C})$ of $K^C$ admits a homomorphism to $K$, for any odd cycle graph $C$.
This shows Theorem~\ref{thm:mainTopo}.

To prove Theorem~\ref{thm:mainFactor}, let $K$ be a graph such that $\pi(K,r)_{/\eq}$ is a free group, let $G$ and $H$ be a connected graphs with odd-length cycles $C,D$, respectively, and let $\phi: (G \times D) \cup (C \times H) \rightarrow K$ be a homomorphism.

Recall that we denote by $[W]$ the $\eq$-equivalence class of a walk $W$ in a graph. From Lemma~\ref{lem:groupProduct}, the equivalence class of a closed walk in a product of graphs is determined by the classes of it's projections to $G$ and to $H$.
Because of that, we consider closed walks that have a trivial projection to one of $G$ or $H$ and later we essentially compose arbitrary walks from them, up to $\eq$-equivalence.
Following~\cite{Wrochna17}, for a closed walk $C'$ in $G$ and an edge $\{h_0,h_1\}$ of $H$, we write $C' \otimes h_0 h_1$ for the closed walk in $G\times H$ of length $2|C'|$ whose projection to $G$ is $C' C'$ and whose projection to $H$ is $(h_0,h_1)(h_1,h_0)$ repeated $|C'|$ times (that is, the simplest walk of length $2|C'|$ that is equivalent to $\eps$ in $H$).
Note that this is a closed walk in $(G \times D) \cup (C \times H)$ whenever $C'=C$ or $\{h_0,h_1\}$ is an edge of $D$.

Fix arbitrary arcs $(g_0,g_1)$ of $C$ and $(h_0,h_1)$ of $D$ and let $r=\phi(g_0,h_0)$.
One of the following three cases holds:
\begin{itemize}
\item $[\phi(C \otimes h_0 h_1)] = [\eps]$,
\item $[\phi(g_0 g_1 \otimes D)] = [\eps]$,
\item $[\phi(C \otimes h_0 h_1)]\neq [\eps]$ and $[\phi(g_0 g_1 \otimes D)] \neq [\eps]$.
\end{itemize}

In the first case, we show that $[\phi(C \otimes h_0' h_1')]=[\eps]$ for any arc $(h_0',h_1')$ of $H$.
This allows to conclude that $\phi^*(\{h_0',h_1'\})$ is an edge of $K^C$ which corresponds to a closed walk that can be lifted to $\Ueq{K}$, and thus an edge of $K^C$ contained in the image $\tau(\Ueq{K}^C)$.
The second case is symmetric (swapping the roles of $G$ with $H$ and $C$ with $D$).

In the third case,
we show that there is a closed walk $R$ in $K$ such that for all $C'$,  $[\phi(C')]$ is a power of $R$.
This allows to lift $\phi$ to $\Ueq{K}_{/R}$.

We now make this more formal with a sequence of lemmas.
First, we show that $[C \otimes h_0 h_1]$ does not depend on the choice of $h_0, h_1$, up to conjugation.

\begin{lemma}\label{lem:invariant}
	For every arc $(h,h')$ of $H$, $[C \otimes h h'] = [W (C \otimes h_0 h_1) W^{-1}]$ for some walk~$W$.
	In particular, $[\phi(C \otimes h h')]=[\eps]$ iff $[\phi(C \otimes h_0 h_1)]=[\eps]$.
\end{lemma}
\begin{proof}
	Here we view $C$ also as a closed walk in $C$, starting from the fixed edge $(g_0,g_1)$.
	First, observe that for any two arcs $(a,b)$, $(a,c)$ in $H$ starting from the same vertex $a$, we have $C \otimes ab\ \eq\ C \otimes ac$.
	Indeed, the walks share all vertices with even indices (starting with $(g_0,a)$ at index 0).
	Thus between any two such vertices, both walks have a sub-walk of length two with shared endpoints.
	Together this forms a square, so by definition of $\eq$, these sub-walks are equivalent.
	Thus $[C \otimes ab] = [C \otimes ac]$.

	Let $W$ be an even walk in $H$ and let $(a,b)$ and $(c,d)$ be its first and last edge, respectively.
	Let $g_0 g_1 \times W$ be the walk in $C\times H$ whose projections are $(g_0,g_1)$ alternating with $(g_1,g_0)$ in $C$ and $W$ in $H$.
	We prove by induction on the length of $W$ that
	\begin{equation}
	[C \otimes a b] = [(g_0 g_1 \times W) (C \otimes d c) (g_0 g_1 \times W)^{-1}]\label{eq:conjugate}
	\end{equation}
	Suppose first that $W$ has length 2.
	Then $b=c$ and $W=(a,b)(b,d)$.
	Hence the closed walks $C\otimes a b$ and $C\otimes d b$ share all vertices with odd indices.
	Similarly as before, $C\otimes a b$ and $C\otimes d b$, but with their first and last edge removed, are $\eq$-equivalent.
	The first edge of $C\otimes a b$ is equivalent to the first three edges of $(g_0 g_1 \times W) (C \otimes d c) (g_0 g_1 \times W)^{-1}$
	and the same for the last and last three edges.
	Together this implies \eqref{eq:conjugate}.

	If $W$ has length greater than 2, we split it into two shorter walks $W=W_1 W_2$ of even length: $W_1$ starting with $(a,b)$ and ending with $(x,y)$,
	and $W_2$ from $(y,z)$ to $(c,d)$.
	Then:
	$$C \otimes a b\ \eq\ (g_0 g_1 \times W_1) (C \otimes y x) (g_0 g_1 \times W_1)^{-1}$$
	$$C \otimes y x\ \eq\  C \otimes y z\ \eq\ (g_0 g_1 \times W_2) (C \otimes d c) (g_0 g_1 \times W_2)^{-1}.$$
	From this, \eqref{eq:conjugate} follows.

	Since $H$ is connected and non-bipartite, there is a walk of even length starting with $(h_0,h_1)$ and ending with $(h',h)$.
	Thus $[C \otimes h_0 h_1] = [W (C \otimes h h') W^{-1}]$ for some walk $W$ in $C \times H$.
	This implies also $[W^{-1}(C \otimes h_0 h_1)W] = [W^{-1}W (C \otimes h h') W^{-1}W]=[C \otimes h h']$.
	The last conclusion holds because $\phi$ is a homomorphism,
	so $[\phi(C \otimes h h')]=[\eps]$ implies 
	\begin{equation*}
	[\phi(C \otimes h_0 h_1)] = [\phi(W)\phi(C \otimes h h')\phi(W^{-1})]=[\phi(W)\phi(W)^{-1}]=[\eps].\tag*{\qedhere}
	\end{equation*}
\end{proof}

Therefore, if $[\phi(C \otimes h_0 h_1)] = [\eps]$, then for every arc $(h,h')$ of $H$, the closed walk $\phi(C \otimes h h')$, which is the closed walk corresponding to the arc $(\phi^*(h),\phi^*(h'))$ of $K^C$, is $\eq$-equivalent to $\eps$.
By Lemma~\ref{lem:universalIsEps}, this means that $\phi^*$ maps each arc to an arc in $\tau(\Ueq{K}^C) \subseteq K^C$.

\begin{corollary}
	Suppose $[\phi(C \otimes h_0 h_1)] = [\eps]$.
	Then  $\phi^* : H \to K^C$ factors through $\tau({\Ueq{K}^C})$.
\end{corollary}

For the remaining case, we first observe that closed walks in $(G \times D) \cup (C \times H)$ are composed of closed walks in $G \times D$ and in $C \times H$.

\begin{lemma}\label{lem:compose}
Every closed walk in $(G \times D) \cup (C \times H)$ rooted at $(g_0,h_0)$ is $\eq$-equivalent to a concatenation of closed walks in $G \times D$ and in $C \times H$.
\end{lemma}
\begin{proof}
Take a spanning tree of $C \times D$ and extend it to $G\times D$ and to $C \times H$.
Every edge $e$ of $(G \times D) \cup (C \times H)$ outside of the tree defines a \emph{fundamental cycle} $C_e$, going from $(g_0,h_0)$ to one endpoint by a shortest path in the tree, through the edge $e$ and then back to $(g_0,h_0)$ through the tree.
By choice of the tree, fundamental cycles are contained in $G\times D$ or $C\times H$.
Furthermore, every closed walk $O$ rooted $(g_0,h_0)$ is $\eq$-equivalent to a concatenation of fundamental cycles.
Indeed, it suffices to take the concatenation of $C_e$ for consecutive edges $e$ of $O$, which is the same as $O$ up to reductions (we do not even need to use squares).
\end{proof}

Finally, when both $[\phi(C \otimes h_0 h_1)]$ and $[\phi(g_0 g_1 \otimes D)]$ wind non-trivially in $K$, we show that all of $(G \times D) \cup (C \times H)$  essentially winds around a common cycle in $K$, allowing to lift $\phi$ to the corresponding unicylic cover.

\begin{corollary}
	Suppose $[\phi(C \otimes h_0 h_1)]\neq [\eps]$ and $[\phi(g_0 g_1 \otimes D)] \neq [\eps]$.
	Then $\phi$ factors through the $\eq$-unicyclic cover $\Ueq{K}_{/R}$, for some odd-length closed walk $R$ in $K$.
\end{corollary}
\begin{proof}
	We first show that there is a closed walk $R$ in $K$ such that for each closed walk $O$ in $(G\times D)\cup(C\times H)$, $[\phi(O)]$ is some power of $[R]$. 
	Here we shall consider only closed walks in $G,H$, products, and $K$ that are rooted at $g_0,h_0,(g_0,h_0)$ and $r=\phi(g_0,h_0)$, respectively.

	Since $[(g_0,g_1)(g_1,g_0)] = [\eps]$ trivially commutes with every element of $\pi(C,g_0)_{/\eq}$,
	similarly for $[(h_0,h_1)(h_1,h_0)]=[\eps]$ in $\pi(D,h_0)_{/\eq}$,
	and since elements of $\pi(C\times D,h_0)_{/\eq}$ are determined by their projections (Lemma~\ref{lem:groupProduct}),
	it follows that $[C \otimes h_0 h_1]$ commutes with $[g_0 g_1 \otimes D]$.
	Therefore, $[\phi(C \otimes h_0 h_1)]$ commutes with $[\phi(g_0 g_1 \otimes D)]$ in $\pi(K,r)_{/\eq}$ (where $r=\phi(g_0,h_0)$).

	We assumed that $\pi(K,r)_{/\eq}$ is a free group, hence every element $[C']\neq[\eps]$ has a primitive root $[R]$, defined as an element satisfying $[C']=[R]^i$ with $i$ maximized. 
	Moreover, since $[\phi(C \otimes h_0 h_1)]\neq[\eps]$ and $[\phi(g_0 g_1 \otimes D)]\neq[\eps]$ are commuting elements in a free group, both are powers of the same primitive root $[R]$.

	Similarly, for every closed walk $C'$ in $G$, $[\phi(C'\otimes h_0 h_1)]$ commutes with $[\phi(g_0 g_1 \otimes D)]$, and so is a power of $R$.
	Also similarly, for every closed walk $D'$ in $H$, $[\phi(g_0 g_1 \otimes D')]$ is a power of $[R]$ (that is, $[R]^i$ for $i\in\ZZ$, possibly negative).

	For every closed walk $O$ in $G \times D$, it follows from Lemma~\ref{lem:groupProduct} that $[\phi(O)^2]=[\phi(O|_G \otimes h_0 h_1)] \cdot [\phi(g_0 g_1 \otimes O|_H)]$.
	Hence $[\phi(O)^2]$ is a power of $R$, and thus so is $[\phi(O)]$.
	Similarly for every closed walk in $C \times H$.
	By Lemma~\ref{lem:compose}, we conclude that for every closed walk $O$ in $(G\times D)\cup(C\times H)$, $[\phi(O)]$ is $\eq$-equivalent to some power of $R$. 
	This in particular holds for some closed walk of odd length, hence $R$ must have odd length.

	This allows to lift $\phi$ to $\Ueq{K}_{/R}$.
	That is, we define $\widetilde{\phi} \colon (G\times D)\cup(C\times H) \to \Ueq{K}_{/R}$
	by mapping each vertex $(g,h)$ to the image in $[\phi(W)]$ of any walk $W$ from $(g_0,h_0)$ to $(g,h)$.
	For any two such walks $W$,$W'$, concatenating $W W'^{-1}$ gives a closed walk, hence $[\phi(W')]=[R^i \phi(W)]$, for some $i\in\ZZ$.
	Thus the image is a well-defined vertex of $\Ueq{K}_{/R}$.
	Clearly $\phi=\widetilde{\phi} \circ \tau$ (where $\tau \colon (\Ueq{K}_{/R}) \to K$ is the endpoint map).
\end{proof}

\section{Proof of Lemma~\ref{lem:constProperty} and Lemma~\ref{lem:components}}\label{app:other}
In this section we show that connected components of $K^{C_n}$ in the limit of increasing $n$ can be characterized by $\eq$-equivalence, in a sense.
In particular we show Lemma~\ref{lem:components} about $K^{C_n}_\eps$, from which Lemma~\ref{lem:constProperty} will follow.

Let us first define a simpler notion of a ``path'' between homomorphisms.
Let $G,K$ be graphs.
We think of $V(K)$ as of a set of colours; a \emph{$K$-colouring} of $G$ is a homomorphism from $G$ to $K$.
For graphs $G,K$, a \emph{$K$-recolouring sequence} of length $\ell$ is a sequence of homomorphisms $\phi_i \colon G \to K$ for $i=0,\dots,\ell$ such that $\phi_{i+1}$ differs from $\phi_{i}$ at only one vertex of $G$.
Additionally, if the recoloured vertex $v$ has a loop in $G$, we require that $\phi_{i}(v)$ is adjacent to $\phi_{i+1}(v)$ in $K$.
We say that two homomorphisms $\alpha,\beta\colon G \to K$ are \emph{recolourable} if there is a recolouring sequence from $\alpha$ to $\beta$.
Recolourings have recently been researched a lot, see e.g.~\cite{CerecedaHJ11,BonsmaC09} for classical colourings and~\cite{Wrochna15,BrewsterMMN15,BrewsterLMNS17} for homomorphisms.
They have also been studied under the name $\times$-homotopy in~\cite{Dochtermann09}.

\def\looped{loop}
By definition, a function $V(G)\to V(K)$ is a homomorphism if and only if it is vertex of $K^G$ with a loop (an edge to itself).
Recolouring corresponds to paths on looped vertices in the exponential graph $K^G$.
That is, let $P_{n+1}$ be the path graph of length $n$, with vertices $0,\dots,n$ (the $+1$ is for consistency with the common convention to put the number of vertices in the subscript).
Let $P_{n+1}^{\looped}$ be the graph obtained by adding a loop at each vertex.
\begin{lemma}[folklore, e.g. Lem 6.2. in \cite{BrewsterN15}]
	For $\alpha,\beta\colon G \to K$, the following are equivalent:\looseness=-1
	\begin{enumerate}[label={(\roman*)},itemsep=0pt]
		\item $\alpha$ and $\beta$ are recolourable;
		\item there is a path on looped vertices between $\alpha$ and $\beta$ in $K^G$.
		\item for some $m \in \NN$, the precolouring of $G \times P_{m+1}^{\looped}$ with $\alpha$ on $G \times \{0\}$ and $\beta$ on $G\times\{1\}$ can be extended to the whole graph.
	\end{enumerate}
\end{lemma}

\def\kL{0}
\def\kR{1}
We are instead interested with edges of $K^G$, not just loops.
Let $V(K_2)=\{\kL,\kR\}$.
A homomorphism $\alpha\colon G\times K_2 \to K$ corresponds to an (oriented) edge $\alpha^* \colon K_2 \to K^G$
from the function $\alpha(-,\kL)$ to $\alpha(-,\kR)$ in $V(K^G)$.
Two such homomorphisms are recolourable if and only if the corresponding edges are in same component of $K^G$ and reachable via a path of appropriate parity (the parity only matters if the component is bipartite and we care about the distinction between an oriented edge and its inverse):

\begin{lemma}\label{lem:recolIsPath}
	For $\alpha,\beta\colon G \times K_2 \to K$, the following are equivalent:
	\begin{enumerate}[label={(\roman*)},itemsep=0pt]
		\item $\alpha$ and $\beta$ are recolourable;
		\item There is a walk of even length from $\alpha(-,\kL)$ to $\beta(-,\kL)$ in $K^G$.
		\item For some even $m$, the precolouring of $G \times P_{m+1}$
			with $\alpha(-,\kL)$ on $G \times \{0\}$ and $\beta(-,\kL)$ on $G \times \{m\}$ can be extended to the whole graph.
	\end{enumerate}
\end{lemma}
\begin{proof}
	$(ii)$ and $(iii)$ are equivalent by definition:
	a walk of length $m$ in $K^G$ corresponds
	to a homomorphism $\phi\colon P_{m+1} \to K^G$
	and thus to $\phi_* \colon G \times P_{m+1} \to K$, given by $\phi_*(g,i)=\phi(i)(g)$.

	Suppose $(i)$ holds and consider a recolouring sequence between $\alpha$ and $\beta$.
	To prove $(ii)$ and $(iii)$, it suffices to show
	that $\alpha'(-,0)$ has a walk of even length to $\beta'(-,0)$
	for any two adjacent $\alpha',\beta' \colon G \times K_2 \to K$ in that sequence.
	Let $(g,i) \in V(G\times K_2)$ be the only vertex on which $\alpha'$ and $\beta'$ differ.
	If $i=1$, then $\alpha'(-,0) = \beta'(-,0)$, so the trivial walk works.
	If $i=0$, then the walk on three vertices
	$\alpha'(-,0),\ \alpha'(-,1)=\beta'(-,1),\ \beta'(-,0)$ works.
	
	Conversely, suppose there is a walk of even length
		from $\alpha(-,\kL)$ to $\beta(-,\kL)$ in $K^G$.
	Without loss of generality we can assume that
		the first edge is from $\alpha(-,\kL)$ to $\alpha(-,\kR)$
		and
		the last edge is from $\beta(-,\kR)$ to $\beta(-,\kL)$
	(otherwise we can add a backtracking $e e^{-1}$ at the beginning and at the end of the walk, without changing its parity).
	Let the vertices of the walk be $f_0,f_1,\dots,f_m$.
	For $0\leq i < m$, let $\phi_i$ be the $K$-colouring of $G\times K_2$ with:
	\begin{align*}
		\phi_i(-,\kL)=f_i(-) \quad & \text{ and }\phi_i(-,\kR)=f_{i+1}(-)\text{ for even $i$, while}\\ \phi_i(-,\kL)=f_{i+1}(-) & \text{ and }\phi_i(-,\kR)=f_i(-)\text{ for odd $i$.}
	\end{align*}
	Then $\phi_0=\alpha$ and $\phi_{m-1}=\beta$.
	It suffices to show that $\phi_i$ is recolourable to $\phi_{i+1}$.
	Indeed, if $i$ is even, then $\phi_i(-,\kR)=\phi_{i+1}(-,\kR)$,
	whereas if $i$ is odd, then $\phi_i(-,\kL)=\phi_{i+1}(-,\kL)$.
	So $\phi_i$ and $\phi_{i+1}$ are the same on one half of $G \times K_2$.
	The other half is an independent set, so the colours on it can be changed from 
	$\phi_i$ to $\phi_{i+1}$ in any order.
\end{proof}

Hence recolouring is a handy way of expressing walks in $K^G$.
For example, we can characterize bipartite components of $K^G$ by recolouring:

\begin{corollary}
	Let $K,G$ be a graphs, and let $\{h,h'\}$ be an edge of $K^G$.
	The component containing $\{h,h'\}$ is bipartite if and only if
	the homomorphism $G \times K_2 \to K$ corresponding to $(h,h')$ cannot be recoloured to the one corresponding to $(h',h)$.
\end{corollary}
\begin{proof}
	By the previous lemma, the two homomorphism can be recoloured if and only if there is a walk of even length between the endpoints $h$ and $h'$ of the edge, which holds if and only if the component containing it is non-bipartite.
\end{proof}

For $K^{C_n}$ the existence of walks is closely related to $\eq$-equivalence,
but $\eq$-equivalence allows to lengthen and shorten walks and closed walks.
Let us first formalize this for walks with fixed endpoints.
For a walk $W$ of length $n$ in $K$ and an integer $m\geq n$ of the same parity as $n$,
we define an \emph{$m$-lengthening} of $W$ to be walk of length $m$ obtained from $W$ by introducing subwalks of the form $e e^{-1}$ for $e \in E(K)$.
Note that the starting and ending vertex of an $m$-lengthening are the same as those of $W$ (empty walks $W$ also have specified endpoints, $\iota(W)=\tau(W)$ arbitrary in $V(K)$) and an $m$-lengthening is $\eq$-equivalent to $W$.
For $m$ of different parity than $n$ or $m<n$, there are no $m$-lengthenings of $W$.

We say that two walks of length $m$ in $K$ are recolourable \emph{rel endpoints} if the corresponding homomorphisms $P_{m+1} \to K$ are recolourable by a sequence that never changes the colour of the endpoints.

\begin{lemma}
	For $m \in \NN$,  any two $m$-lengthenings of a single walk are recolourable rel endpoints.
\end{lemma}
\begin{proof}
	It suffices to observe two things.
	Firstly, a walk obtained by introducing a subwalk $(a,b)(b,a)$ for some edge $\{a,b\}$ of $K$ can be recoloured (rel endpoints) to any other walk obtained by introducing a subwalk $(c,d)(d,c)$ at the same position,
	because endpoints $a$ and $c$ have to be equal, so one can simply move the middle vertex from $b$ to $d$.
	Secondly, a walk obtained by introducing a subwalk $(a,b)(b,a)$ \emph{before} an edge $e$ can be recoloured to a walk obtained by introducing a subwalk $(c,d)(d,c)$ \emph{after} the edge $e$,
	because the endpoints of $e$ must be $a$ and $c$, respectively, so $(a,b)(b,a) e = (a,b)(b,a)(a,c)$ can be recoloured to $(a,c)(c,d)(d,c) = e (c,d)(d,c)$ by moving the second vertex from $b$ to $c$ and the third vertex from $a$ to $d$.
	Repeating such steps, one can replace any lengthening with any other lengthening at any other position.
\end{proof}

Note that the previous lemma implies that for $m\in\NN$, if some $m$-lengthenings of $W$ and of $W'$ are recolourable rel endpoints, then all of them are.
Moreover, so are any $m'$-lengthenings of them, for $m' \geq m$.

\begin{lemma}
	Two walks $W,W'$ are $\eq$-equivalent if and only if $m$-lengthenings of $W$ and $W'$ are recolourable rel endpoints, for large enough $m$.
\end{lemma}
\begin{proof}
	Suppose that for some $m\in\NN$, some $m$-lengthenings $P$ and $P'$ of $W$ and $W'$ are recolourable rel endpoints.
	Clearly a walk is $\eq$-equivalent to any $m$-lengthening of it, so it suffices to show that 
	 $P$ and $P'$ are $\eq$-equivalent.
	But this is easy to see, since any recolouring step between them only changes one vertex (other than the first and last vertex) and hence it only replaces a subwalk $(a,b)(b,c)$ with a subwalk $(a,d)(d,c)$ for some $a,b,c,d \in V(K)$, which is $\eq$-equivalent.

	To show the converse we observe three things.
	Firstly, if $\{a,b\}$ is an edge of $K$, then the walk $(a,b)(b,a)$ is a $2$-lengthening of the empty walk at $a$.
	Hence by the previous lemma, $m$-lengthenings of $(a,b)(b,a)$ are recolourable rel endpoints to $m$-lengthenings of the empty walk at $a$, for $m \geq 2$.
	Second, if $a,b,c,d \in V(K)$ form a square in $K$, then the walk $(a,b)(b,c)(c,d)(d,a)$ is recolourable rel endpoints to $(a,b)(b,a)(a,d)(d,a)$, which is a $4$-lengthening of the empty walk at $a$. 
	Hence by the previous lemma, $m$-lengthenings of $(a,b)(b,c)(c,d)(d,a)$ are recolourable rel endpoints to $m$-lengthenings of the empty walk at $a$, for $m \geq 4$.
	Finally, suppose $m_1$-lengthenings of $W_1$ and $W_1'$ are recolourable rel endpoints, that $m_2$-lengthenings of $W_2$ and $W_2'$, and that $W_1$ can be concatenated with $W_2$.
	Then the respective lengthenings can be concatenated and still recoloured rel endpoints,
	which means that $(m_1+m_2)$-lengthenings of $W_1 W_2$ and $W_1'W_2'$ are recolourable.

	Together, by the definition of $\eq$, this implies that the relation between walks defined as ``large enough lengthenings are recolourable rel endpoints'' contains the relation $\eq$.
\end{proof}

We now consider the same for closed walks.
We say two closed walks of length $n$ are recolourable if the corresponding homomorphisms from $C_n$ are.
That is, the position of the initial vertex can change, though the walk must remain closed.

\begin{lemma}\label{lem:recolIffConjugate}
	Let $C,C'$ be closed walks in a graph $K$.
	Then $C \eq W C' W^{-1}$ for some walk $W$ of even length if and only if
	$m$-lengthenings of $C,C'$ are recolourable for large enough $m\in\mathbb{N}$.
\end{lemma}
\begin{proof}
	Suppose $C \eq W C' W^{-1}$ for some walk $W$ of even length.	
	Then for large enough $m \in\mathbb{N}$, $m$-lengthenings of $C$ and $W C' W^{-1}$ are recolourable rel endpoints.
	Hence they are recolourable as closed walks.
	It remains to show that $m$-lengthenings of  $W C' W^{-1}$ and $C'$ are recolourable
	(for $m \geq |WC' W^{-1}|$).
	This can be proved by induction on the length of $W$.
	If $W$ is empty this is trivial, so let $W = e f W'$ for some arcs $e,f$ and a shorter walk $W'$.
	Then $W C' W^{-1} = e f W' C' W'^{-1} f^{-1} e^{-1}$ can be recoloured to 
	$f^{-1} f W' C' W'^{-1} f^{-1} f$ by moving the initial vertex.
	This is a $|W C' W^{-1}|$-lengthening of $W' C' W'^{-1}$, hence $m$-lengthenings of $W C' W^{-1}$ can be recoloured to $m$-lengthenings of $W' C' W'^{-1}$ and by induction, to $m$-lengthenings of $C'$.

	For the other direction, suppose some $m$-lengthenings of $C$ and $C'$ are recolourable.
	It suffices to show that any two consecutive closed walks  $C_1,C_2$ in a recolouring sequence between them satisfy $C_1 \eq W C_2 W^{-1}$ for some walk $W$ of even length.
	If $C_1,C_2$ differ in the placement of a vertex other than the initial vertex, then they are recolourable rel endpoints, hence $C_1 \eq C_2$.
	Otherwise, one is obtained from the other by moving the root vertex, that is, $C_1 = (a,b) P (c,a)$ and $C_2 = (d,b) P (c,d)$ for some walk $P$ and $a,b,c,d\in V(K)$.
	Then $C_1 \eq W C_2 W^{-1}$ for $W=(a,b)(b,d)$ (because $W \eq (a,c)(c,d)$).	
\end{proof}

We say that two closed walks $C,C'$ are $\eq$-conjugate if $C \eq W C' W^{-1}$ for some walk $W$ (of arbitrary length).
We use the above lemma to conclude that $\eq$-conjugacy characterizes the connected components of $K^{C_n}$ in the limit of increasing $n$, in a sense.

For two integers $m\geq n\geq 3$ of the same parity, let $\ell^m_n \colon C_{m} \to C_{n}$ be a homomorphism mapping 0 to 0, say $\ell^m_n(i) = i$ for $i< n$ and $\ell^m_n(n-1+i)=n-1-(i\mod 2)$ for $i=1,\dots,m-n$.
The homomorphism $\ell^m_n \colon C_{m} \to C_{n}$ induces a dual homomorphism $\ell^{m\dagger}_n \colon K^{C_n} \to K^{C_m}$, mapping $h \in V(K^{C_n})$ to $\ell^m_n \circ h \in V(K^{C_m})$.
More generally, a closed walk $C$ of length $n$, as a homomorphism from $C_n$, can be composed with $\ell^m_n$ to give a closed walk of length $m$, which we can also denote as $\ell^{m\dagger}_n(C)$.
Note this is an $m$-lengthening of $C$.
We drop the subscript $n$ when it is clear from context.

The chain of homomorphisms 	$C_3 \leftarrow C_5 \leftarrow C_7 \leftarrow \dots$ given by $\ell^{n+2}_n$
induces the dual chain $K^{C_3} \to K^{C_5} \to K^{C_7} \to \dots$ given by $\ell^{n+2\dagger}_n$.
The latter homomorphisms are injective, hence they make $K^{C_n}$ a subgraph of $K^{C_{n+2}}$.
Each connected component of $K^{C_n}$ is mapped into a component of $K^{C_{n+2}}$, but different components of $K^{C_n}$ may get mapped into the same component.
The following theorem characterizes which components remain different in the limit.\footnote{We could actually define a limit graph $K^{C_\infty}$, as the sum of $K^{C_n}$ over odd $n$, identifying $K^{C_n}$ with $\ell^{n+2\dagger}_n(K^{C_n}) \subseteq K^{C_{n+2}}$; this coincides with the notion of \emph{colimit} of the diagram $K^{C_3} \to K^{C_5} \to K^{C_7} \to \dots$, in category theory.
The theorem then characterizes connected components of $K^{C_\infty}$.
See~\cite{Dochtermann09b}, \cite{Schultz09}, and~\cite{Matsushita16} for related limits.
}

\begin{theorem}\label{thm:charComponents}
	Let $K$ be a graph, let $n$ be odd, and let $e_1, e_2$ be edges of $K^{C_n}$.
	The following are equivalent:
	\begin{enumerate}[label={(\roman*)}]
		\item $\ell^{m\dagger}$ maps $e_1$ and $e_2$ into the same connected component of $K^{C_m}$, for large enough $m$;
		\item $m$-lengthenings of the closed walks of length $2n$ corresponding to some orientations of $e_1$ and $e_2$ are recolourable, for large enough $m$;
		\item closed walks of length $2n$ corresponding to some orientations of $e_1, e_2$ are $\eq$-conjugate.
	\end{enumerate}
\end{theorem}
\begin{proof}
	$(i)$ is equivalent to saying there is a walk between $\ell^{m\dagger}(e_1)$ and $\ell^{m\dagger}(e_2)$ in $K^{C_m}$.
	By choosing the orientation of each edge appropriately, we can assume this is a walk of even length between their initial vertices.
	By Lemma~\ref{lem:recolIsPath}, this is equivalent to saying that the closed walks of length $2m$ corresponding to some orientations of $\ell^{m\dagger}(e_1)$ and $\ell^{m\dagger}(e_2)$ are recolourable.
	Since these are some $m$-lengthenings of the closed walks of length $2n$ corresponding to $e_1$ and $e_2$, this is equivalent to $(ii)$.

	By Lemma~\ref{lem:recolIffConjugate}, $(ii)$ is equivalent to saying that the closed walks $C_{(1)}, C_{(2)}$ corresponding to some orientations of $e_1, e_2$, satisfy $C_{(1)} \eq W C_{(2)} W^{-1}$, for some walk $W$ of even length.
	However, if this is true for a walk $W$ of odd length instead, then it is true for a walk of even length as follows.
	Swapping the orientation of $e_1$ changes the corresponding closed walk $C_{(1)}$ of length $2n$ to the closed walk $C_{(1)}'$ obtained by an $n$-fold cyclic shift.
	That is, $C_{(1)}' \eq P^{-1} C_{(1)} P$, where $P$ is the first half of the walk $C_{(1)}$, which means $P$ is a walk of odd length $n$.
	Hence $C_{(1)}' \eq P^{-1} C_{(1)} P \eq P^{-1} W C_{(2)} W^{-1} P = W' C_{(2)} W'^{-1}$,
	where $W'=P^{-1} W$ is a walk of different parity than $W$.
	Therefore $(ii)$ is equivalent to $(iii)$.
\end{proof}

We note that we could add another equivalent, topological statement: $(iv)$ closed curves in the box-complex of $K$ (or equivalently, in $\mathrm{Hom}(K_2,K)$) corresponding to these closed walks are homotopic.
We skip the details.

For the conjugacy class of trivial walks Theorem~\ref{thm:charComponents} implies the following:

\begin{lemma}
For a connected graph $K$ and an odd $n$, the subgraph $K^{C_n}_\eps$ of $K^{C_n}$ consists of connected components of $K^{C_n}$, including in particular the component containing constant functions (though possibly more).
Moreover, for large enough odd $m$, $K^{C_n}_\eps$ admits a homomorphism into the connected component of constants of $K^{C_m}$.
\end{lemma}
\begin{proof}
Consider an edge $e_1$ of $K^{C_n}$.
By definition, it is included in $K^{C_n}_\eps$ if and only if the corresponding closed walk $C$ of length $2n$ in $K$ is $\eq$-equivalent to $\eps$.
Let $u,v\in V(K)$, let $f_u \colon i \mapsto u, f_v \colon i \mapsto v$ be constant functions in $V(K^{C_n})$ and let $e_2$ be the edge between them.
Then the closed walk corresponding to $e_2$ is also $\eq$-equivalent to an empty walk (though possibly rooted at a different vertex).
Therefore, $e_1$ is in $K^{C_n}_\eps$ if and only if the closed walks corresponding to $e_1$ and $e_2$ are $\eq$-conjugate (the choice of orientation does not matter).
By Theorem~\ref{thm:charComponents}, this is equivalent to saying that for large enough $m$, $\ell^{m\dagger}$ maps $e_1$ and $e_2$ into the same connected component of $K^{C_m}$.
The component containing $e_2$ is the component of constants.
Hence $e_1$ is in $K^{C_n}_\eps$ if and only if $\ell^{m\dagger}$ maps it to the component of constants, for large enough $m$.

Taking the maximum of the requirements on $m$ over all edges $e_1$ in $K^{C_n}_\eps$ we conclude that for $m$ at least as large, $K^{C_n}_\eps$ is mapped by $\ell^{m\dagger}$ into the components of constants in $K^{C_m}$.

Moreover, suppose $e_1$ is an edge of $K^{C_n}_\eps$.
Then $\ell^{m\dagger}$ maps it to the component of constants, for large enough $m$.
Let $H$ be the connected component of $K^{C_n}$ containing $e_1$.
Since $\ell^{m\dagger}$ maps the connected subgraph $H$ to a single component of $K^{C_m}$,
all of its edges are mapped into the component of constants.
Hence all edges of $H$ are in $K^{C_n}_\eps$.
Therefore, $K^{C_n}_\eps$ consists of connected components of $K^{C_n}$.
\end{proof}

Finally, the above lemma implies Lemma~\ref{lem:constProperty}, which we restate here.
\begin{lemma}
The following are equivalent, for any connected graph $K$:
\begin{enumerate}[label={(\roman*)}]
	\item $K^{C_{n}}_\eps$ admits a homomorphism to $K$, for each odd $n$.
	\item the component of constants in $K^{C_{n}}$ admits a homomorphism to $K$, for each odd $n$.
	\item if $G$ is non-bipartite and $H$ is a connected graph with a vertex $h$ such that there exists a homomorphism $\phi\colon G \times H \to K$ with $\phi(-,h)$ constant, then $H$ admits a homomorphism to $K$. 
\end{enumerate}
\end{lemma}
\begin{proof}
Clearly $(i)$ implies $(ii)$.

Suppose that $(ii)$ holds.
Let $\phi: G \times H \to K$ be a homomorphism with $G$ non-bipartite and $H$ connected with a vertex $h_0$ such that $\phi(-,h_0)$ is constant.
Without loss of generality $G$ is an odd cycle, since we can take any odd cycle $C$ in it and restrict $\phi$ to $C \times H$.
We hence have a map $\phi^* \colon H \to K^{C_n}$ for some odd $n$.
Since $\phi(-,h_0)$ is constant, $\phi^*(h_0)$ is a constant function in $V(K^{C_n})$.
Hence $\phi^*$ maps the connected graph $H$ to the component of constants in $K^{C_n}$ and therefore $H$ admits a homomorphism to $K$.
This proves $(iii)$.

Suppose $(iii)$ holds.
Let $H=K^{C_{n}}_\eps$ for an odd $n$.
Then there is an odd $m$ such that $H$ admits a homomorphism into the connected component of constants in $K^{C_m}$.
Let us denote this component as $H'$.
We have a natural ``evaluation'' homomorphism $\phi \colon C_m \times H' \to K$, namely $\phi(i,h) := h(i)$.
Moreover, a constant function $h_0 \in  V(H')$ means that $\phi(-,h_0)$ is constant.
Therefore $H'$ admits a homomorphism to $K$, by assumption, and $H$ admits a homomorphism into $H'$.
This proves $(i)$.
\end{proof}

\section{Some basic covering theory}\label{app:covers}
We refer the reader to Matsushita's work~\cite{Matsushita12} for a detailed treatment of covering theory in graphs.
Here we only give a few simple proofs to illustrate it.

Unravelling definitions of the universal cover gives us the following two lemmas.
First, consider a walk $\widetilde{W}$ in $\Ueq{K,r}$. 
Let its consecutive vertices be equivalence classes $[W_0],\dots,[W_n]$ of some walks in $K$.
There are two ways to interpret $\tau(\widetilde{W})$: formally $\tau(\widetilde{W})$ is its endpoint $[W_n]$.
Alternatively, we can consider the walk going through the vertices $\tau(W_0),\tau(W_1),\dots,\tau(W_n)$ in $K$; we will denote this walk as $\bar{\tau}(\widetilde{W})$ for clarity.
However, the definitions coincide, in a sense, for walks $\widetilde{W}$ starting at $[\eps]$:

\begin{lemma}\label{lem:tauIsTau}
	Let $\widetilde{W}$ be a walk in $\Ueq{K}$ starting at the vertex $[\eps]$.
	Then $[\bar{\tau}(\widetilde{W})] = \tau(\widetilde{W})$.
\end{lemma}
\begin{proof}
	By induction on the length $n$ of $W$, we show that $\bar{\tau}(\widetilde{W}) \eq W_n$.
	When adding another edge, from $[W_n]$ to $[W_{n+1}]$, observe that $[W_{n+1}]$ is adjacent to $[W_n]$, which means $W_{n+1}$ is $\eq$-equivalent to the walk $W_n$ with the edge $(\tau(W_n),\tau(W_{n+1}))$ appended.
\end{proof}

Second, we show formally that the subgraph $K^{C}_\eps$ of $K^{C}$ given by edges that correspond to closed walks $\eq$-equivalent to $\eps$ is equal to $\tau(\Ueq{K}^C)$.

\begin{lemma}\label{lem:universalIsEps2}
	Let $(h,h')$ be an arc of $K^{C_n}$, for some odd $n\in\NN$.
	The closed walk of length $2n$ corresponding to $(h,h')$ in $K$ is $\eq$-equivalent to $\eps$ iff the edge is in $\tau(\Ueq{K}^C) \subseteq K^{C}$.
\end{lemma}
\begin{proof}
	Let the vertices of $C$ be $g_0,g_1,g_2,\dots,g_n$ ($n$ even).
	Let $O$ be the closed walk corresponding to $(h,h')$, that is, the walk through vertices $h(g_0),h'(g_1),h(g_2),\dots,h(g_n),h'(g_1),\dots,h(g_0)$.
	
	Suppose $O$ is $\eq$-equivalent to $\eps$.
	Then consecutive prefixes of $O$ give a walk of the same length $2n$ in $\Ueq{K,h(g_0)}$ that is closed, and which is mapped by $\tau$ to $O$.
	This closed walk in $\Ueq{K}$ corresponds to an edge of $\Ueq{K}^C$.
	Specifically, one endpoint of this edge is the function that maps $g_i$ to the prefix ending in $h(g_i)$, the other endpoint is the function that maps $g_i$ to the prefix ending in $h'(g_i)$.
	The image of that edge through $\tau$ is $(h,h')$.

	In the other direction, suppose that $(h,h') = \tau((\widetilde{h},\widetilde{h'}))$ for some root $r \in V(K)$ and some arc $(\widetilde{h},\widetilde{h'})$ of $\Ueq{K,r}^C$.
	Consider the corresponding closed walk $\widetilde{O}$ of length $2n$ in $\Ueq{K,r}$, that is, the walk through vertices $\widetilde{h}(g_0),\widetilde{h'}(g_1),\widetilde{h}(g_2),\dots,\widetilde{h}(g_0)$.
	By prepending $\widetilde{h}(g_0)^{-1}$, which is a walk in $K$ from $r' := \tau(\widetilde{h}(g_0))$ to $r$, to each of these, we obtain a closed walk in $\Ueq{K,r'}$ whose first and last vertex is $[\eps]$.
	On one hand, $\tau$ maps this walk to its endpoint $[\eps]$.
	On the other hand, $\tau$ maps consecutive vertices of this walk to consecutive vertices of the walk corresponding to $(h,h')$, which is therefore $\eq$-equivalent to $\eps$, by Lemma~\ref{lem:tauIsTau}.
\end{proof}

As an addition, we show a basic fact about lifts, which allows to conclude that the $\eq$-fundamental group of $\Ueq{K}$ is trivial (though we will not need these statements later in our proofs). 
We show it for any \emph{$\eq$-covering map} $\tau: \widetilde{K} \to K$, that is, a surjective homomorphism between two graphs $\widetilde{K}$ and $K$ such that the neighbourhood of every vertex is mapped bijectively and so is the second neighbourhood (the set of vertices reachable by a walk of length 2).

\begin{lemma}\label{lem:lift}
	Let $\tau: \widetilde{K} \to K$ be a $\eq$-covering map.
	For every walk $W$ from $v$ in $K$ and every $\widetilde{v} \in V(\widetilde{K})$ such that $\tau(\widetilde{v})=v$, there is a unique walk $\widetilde{W}$ in $\widetilde{K}$ from $\widetilde{v}$ such that $\bar{\tau}(\widetilde{W}) = W$, called the \emph{lift} of $W$ at $\widetilde{v}$.
	Moreover, $\eq$-equivalent walks have $\eq$-equivalent lifts (in particular, the same endpoints).
\end{lemma}
\begin{proof}
	The first statement holds by induction on the length of $W$:
	when adding an edge $(a,b)$ to $W$ with lift $\widetilde{W}$,
	we can add to $\widetilde{W}$ the edge from the endpoint of $\widetilde{W}$
	to its unique neighbour in $\tau^{-1}(b)$.

	To prove that lifts of $W$ and $W'$ are $\eq$-equivalent whenever $W$ and $W'$ are, it suffices to observe that the lift of a concatenation of two walks is a concatenation of their lifts, and that the lift of a backtracking $(a,b)(b,a)$ or a walk around a square $(a,b)(b,c)(c,d)(d,a)$ is itself a backtracking or a walk around a square.
	Indeed, for a lift $\widetilde{a}$ of $a$, there is a unique vertex $\widetilde{c}$ in its second neighbourhood that is mapped to $c$ by $\tau$; similarly there are unique $\widetilde{b}$ and $\widetilde{d}$ in the first neighbourhood, adjacent to $\widetilde{c}$.
	This implies that the lift of $(a,b)(b,c)(c,d)(d,a)$ at $\widetilde{a}$ must pass through $\widetilde{b}$, $\widetilde{c}$, $\widetilde{d}$, and return to $\widetilde{a}$
	(as opposed to some other vertex in $\tau^{-1}(a)$).
\end{proof}

For the $\eq$-universal cover $\Ueq{K,r}$ and a walk $W$ from $r$, the lift at $[\eps]$ is simply given by the sequence of (equivalence classes of) prefixes of $W$.

Lemma~\ref{lem:lift} implies the following two corollaries,
which explain why we call $\Ueq{K}$ ``universal''.
\begin{corollary}
For any $\eq$-covering map $\tau_1: \widetilde{K} \to K$ with $\widetilde{K}$ connected,
there is a $\eq$-covering map $\tau_2: \Ueq{K} \to \widetilde{K}$,
which moreover satisfies $\tau = \tau_2 \circ \tau_1$.
\end{corollary}
\begin{proof}
$\tau_2$ maps $[W]$ in $\Ueq{K,r}$ to the endpoint of the lift of $W$ at $\widetilde{r}$ (for some arbitrary $r \in V(K)$ and $\widetilde{r} \in \tau_1^{-1}(r)$).
\end{proof}

\begin{corollary}\label{cor:universalGroup}
	For every graph $K$, every two walks between the same endpoints in $\Ueq{K}$ are $\eq$-equivalent.
	Hence the $\eq$-fundamental group of $\Ueq{K}$ is trivial.
\end{corollary}
\begin{proof}
	Let $\widetilde{W},\widetilde{W}'$ be two walk between the same endpoints in $\Ueq{K,r}$, for some root $r \in V(K)$.
	Let their common initial vertex be $[W_0]$ for some walk $W_0$ from $r$ to $r'$ in $K$.
	We can assume $[W_0]=[\eps]$, since otherwise applying the isomorphism $\alpha_{{W_0}^{-1}}$ maps $\widetilde{W},\widetilde{W}'$ to two walks starting from $[\eps]$ in $\Ueq{K,r'}$.
	Since they have a common endpoint $\tau(\widetilde{W}) = \tau(\widetilde{W}')$, it follows from Lemma~\ref{lem:tauIsTau} that $\bar{\tau}(\widetilde{W}) \eq \bar{\tau}(\widetilde{W}')$.
	Thus by Lemma~\ref{lem:lift}, their lifts $\widetilde{W},\widetilde{W}'$ are $\eq$-equivalent.
\end{proof}

As Matsushita~\cite{Matsushita12} shows, virtually all of the classical theory of coverings in topology is mirrored with $\eq$-coverings.
In particular, he proved that $\eq$-coverings of a graph naturally and bijectively correspond to subgroups of the $\eq$-fundamental group, and that they induce coverings (in the topological sense) of the neighbourhood complex (or box complex) of a graph.

\printbibliography

\end{document}